\documentclass[12pt]{amsart}
\usepackage[utf8]{inputenc}
\pdfoutput=1
\usepackage{amsmath,amssymb,amsthm,amsfonts}
\usepackage{mathtools}%

\usepackage[english]{babel}%
\numberwithin{equation}{section}

\usepackage{hyperref}%
\usepackage{bbm}%
\usepackage{bm}%
\usepackage{mathrsfs}%
\usepackage{xparse}
\usepackage{xstring}

\usepackage{tikz,graphicx,color}
\usepackage{tikz-cd}%
\usepackage{tikz-3dplot}%
\usetikzlibrary{calc}%
\usetikzlibrary{arrows}%
\usetikzlibrary{shapes}%
\usetikzlibrary{patterns}%
\usetikzlibrary{positioning}%
\usetikzlibrary{arrows.meta}
\usetikzlibrary{decorations.markings}
\usetikzlibrary{knots}

\usepackage{epstopdf}%
\usepackage{blkarray}

\usepackage[arrow]{xy}%
\usepackage{diagbox}%
\usepackage{caption}
\usepackage{subcaption}
\usepackage{arcs}%
\usepackage{xcolor}%
\usepackage{xspace}

\usepackage{microtype}
\usepackage[margin=1.0in]{geometry}%

\usepackage{enumitem}
\usepackage{letltxmacro}
\usepackage{etoolbox}

\def\myarabic#1{\normalfont(\roman{#1})}
\newlist{theoremlist}{enumerate}{1}
\setlist[theoremlist]{label=\myarabic{theoremlisti},ref={\myarabic{theoremlisti}},itemindent=0pt,labelindent=0pt,
 leftmargin=*,noitemsep}

\makeatletter
\renewcommand{\p@theoremlisti}{\perh@ps{\thetheorem}}
\protected\def\perh@ps#1#2{\textup{#1#2}}
\newcommand{\itemrefperh@ps}[2]{\textup{#2}}
\newcommand{\itemref}[1]{\begingroup\let\perh@ps\itemrefperh@ps\ref{#1}\endgroup}

\newcommand\restr[2]{{
  \left.\kern-\nulldelimiterspace 
  #1 
  \vphantom{\big|} 
  \right|_{#2} 
  }}

\usepackage{nameref}
\usepackage[capitalize]{cleveref}

\newtheorem{theorem}{Theorem}[section]

\newtheorem{lemma}[theorem]{Lemma}
\newtheorem{proposition}[theorem]{Proposition}

\theoremstyle{definition}
\newtheorem{remark}[theorem]{Remark}

\newtheorem{example}[theorem]{Example}

\newtheorem{assumption}[theorem]{Assumption}



\newcommand{\Z}{\mathbb{Z}}
\newcommand{\C}{\mathbb{C}}
\newcommand{\R}{\mathbb{R}}

\def\G{\tilde{G}}
\def\Gtor{G}

\def\vertices{V}
\def\blackvertices{B}
\def\whitevertices{W}
\def\faces{F}
\def\edges{E}

\def\Pcomplex{\Sigma} 
\NewDocumentCommand{\cells}{mm}{%
  \IfSubStr{#1}{_}{\left(#1\right)_{#2}}{#1_{#2}}%
}

\definecolor{curveblue}{RGB}{61,153,204}

\newcommand{\wt}{\mathrm{wt}}
\newcommand{\val}{\operatorname{val}}
\newcommand{\ra}{\rightarrow}
\def\res{\operatorname{res}}

\def\divisor{D}

\def\casimir{C}

\def\Log{\operatorname{Log}}

\usepackage{mathtools}

\usetikzlibrary{arrows.meta,decorations.markings}

\crefname{figure}{Figure}{Figures}

\def\figref#1(#2){Figure~\hyperref[#1]{\ref*{#1}(#2)}}

\definecolor{calpolypomonagreen}{rgb}{0, 0.6, 0.2}

\newcounter{todofigure}

\def\Rpos{\R_{>0}}

\def\Puis{\mathbb K}

\def\Puispos{\Puis_{>0}}
\def\val{\operatorname{val}}

\def\trop{\dagger}

\def\init{\operatorname{init}}
\def\lc{\operatorname{lc}}

\def\imunit{{\mathrm{i}}}

\def\bw{\mathsf{w}}
\def\bb{\mathsf{b}}
\def\be{{\mathsf{e}}}
\def\bface{\mathsf{f}}

\def\dam{{\bf d}}

\def\torus{\mathbb{T}}

\def\bQ{\mathsf{Q}}
\def\bK{\mathsf{K}}

\def\onebf{\mathbf{1}}

\def\zzG{{{Z}(G)}}

\def\toricsurface{V_N^{\mathrm{tor}}}
\def\toricsurfacesmall{V^{\mathrm{tor}}}
\def\toricsurfacedegeneration{\mathcal V_\Subd^\mathrm{tor}}
\def\toricsurfacedegenerationfibert{\mathcal V_{\Subd}^\mathrm{tor}(t)}
\def\toricsurfacedegenerationfiberzero{\mathcal V_{\Subd}^\mathrm{tor}(\infty)}

\def\spectralcurve{\Sigma}
\def\openspectralcurve{{\Sigma^{\circ}}}

\def\degeneratecurve{\Sigma(\infty)}
\def\amoeba{\mathbb A}

\def\sym{\operatorname{Sym}}
\def\jac{{\operatorname{Jac}}}
\def\jacpos{{\operatorname{Jac}_{>0}}}

\def\divisor{D}

\def\vecRC{\bm{\Delta}}
\def\RC{{\Delta}}
\def\vecptjac{\bm{\xi}}

\def\veccasimirs{\bm{C}}

\tikzset{qvert/.style={draw,black,circle,fill=gray,minimum size=5pt,inner sep=0pt}  } 
\tikzset{bvert/.style={draw,circle,fill=black,minimum size=5pt,inner sep=0pt}  }  
\tikzset{gbvert/.style={draw, gray, circle,fill=gray,minimum size=5pt,inner sep=0pt}  } 
\tikzset{gvert/.style={draw,gray,circle,fill=white,minimum size=5pt,inner sep=0pt}  } 
\tikzset{wvert/.style={draw,circle,fill=white,minimum size=5pt,inner sep=0pt}  } 
\tikzset{fvert/.style={text=MidnightBlue}  } 
\tikzset{sqvert/.style={draw,black,rectangle,fill=black,minimum size=5pt,inner sep=0pt}  } 
\tikzset{lvert/.style={draw,circle,fill=black,minimum size=4pt,inner sep=0pt}  }  
\usetikzlibrary{arrows}
\tikzcdset{arrow style=tikz, diagrams={>={Stealth[round,length=4pt,width=4.95pt,inset=2.75pt]}}}
\tikzset{nvert/.style={draw,circle,fill=black,minimum size=3pt,inner sep=0pt}  }

\def\pt{{\bm{p}}} 

\def\ptR{{\bm{x}}} 

\def\ptC{{\bm{z}}} 
\def\ptZ{\bm{m}} 

\def\NZ{N_{\Z}}
\def\No{N^\circ}
\def\NZo{N^\circ_{\Z}}

\def\hei{c^{\trop}}

\def\Subd{\Upsilon}

\def\<{\langle}
\def\>{\rangle}
\def\wbsup{{\bw\bb}}
\def\wxsup_#1{{\bw\bb_{#1}}}

\def\Divtr{\divisor^{\trop}}
\def\Ctr{C^\trop}

\def\Pcomplex{\Sigma} 

\def\chisup^#1{\chi^{#1}}

\def\Eside{{E^\ast}}

\def\Nbd{N^\partial}
\def\divptgeom{\zeta}

\def\det{\operatorname{det}}

\def\Conv{\operatorname{conv}}

\def\Q{\mathbb{Q}}

\def\Latcap_#1{\mathbb{L}_{#1}}
\def\starof(#1){\operatorname{star}_{\Pcomplex}(#1)}
\def\normout_#1{[#1]^\perp_{\mathtt{out}}}
\def\normoutpar_#1{(#1)^\perp_{\mathtt{out}}}

\def\abeldiffbw_#1{\delta^{\wbsup}_{#1}}
\def\tor<#1,#2>{\<#1,#2\>_{\torus}}

\def\WEDGE_#1{(\bb_{#1}\xrightarrow{e_{2#1-1}}\bw_{#1}\xrightarrow{e_{2#1}}\bb_{#1+1})}

\def\GGKphi(#1){\hom[#1]}
\def\GGKrw_#1{r^\trop_{#1}}

\def\hom[#1]{[\mkern-3mu[#1]\mkern-3mu]}
\def\ptsign(#1){\varepsilon_{#1}}
\def\Pbb{\mathbb{P}}
\def\amoebaof#1{\amoeba_{#1}}
\def\amoebaSig{\amoebaof{\openspectralcurve}}

\def\ptCg{\vecptjac}
\def\ptSig{\zeta}
\def\ptSigT{\zeta^\trop}
\def\basept{{\ptSig_0}}
\def\baseptT{{\ptSig_0^\trop}}

\def\mcurve{\Sigma}
\def\tropmcurve{\mcurve^\trop}

\begin{document}
	
	\title{Tropicalization of Fock's inverse spectral transform}

  \author{Terrence George}
\address{
Tata Institute of Fundamental Research-Centre for Applicable Mathematics
}
\email{terrence@tifrbng.res.in}
	\date{}
	
\begin{abstract}
The dimer spectral transform gives algebro-geometric action-angle coordinates for cluster integrable systems. Fock gave an explicit formula for its inverse. We study Fock's inverse map in the tropical limit where Harnack spectral curves degenerate to nodal curves. We determine the leading order asymptotics of the classical objects appearing in Fock's formula and show that they are governed by their tropical counterparts. This yields a tropical inverse spectral transform on tropical spectral data. We show that Fock's inverse spectral transform commutes with tropicalization, and also derive a tropical version of Fay's trisecant identity.
\end{abstract}

	\maketitle

\section{Introduction}

 Goncharov--Kenyon~\cite{GK} constructed a remarkable class of integrable systems from periodic bipartite dimer models called cluster integrable systems. For these systems, the spectral transform of Kenyon--Okounkov~\cite{KO07b} identifies the phase space of weights on the dimer model with spectral data on a Harnack spectral curve, giving algebro-geometric action-angle coordinates. Fock~\cite{Fock} constructed an explicit inverse to the spectral transform in terms of theta functions and abelian differentials. 

In this paper, we study Fock's map in the tropical limit. Applying the spectral transform to a positive Puiseux-valued family of weights, we obtain a family of Harnack spectral curves degenerating to a nodal curve. This degeneration determines a tropical spectral curve together with tropical spectral data. Our main result is that Fock's map has a tropical limit which makes the following diagram commute:

\[
\begin{tikzcd}[column sep=huge,row sep=large]
\left\{
\text{Puiseux-valued weights}
\right\}
\arrow[r,"\text{spectral transform}"]
\arrow[d,"\operatorname{tropicalization}"']
&
\{
\text{family of spectral data}
\}
\arrow[d,"\operatorname{tropicalization}"]
\\
\{
\text{tropical weights}
\}
&
\{
\text{tropical spectral data}
\}
\arrow[l,"\text{tropical Fock's map}"]
\end{tikzcd}
\]

Fock's map is given by an explicit formula in terms of classical
objects on the spectral curve: the Abel map, the period matrix, the
theta function, the vector of Riemann constants, and normalized
differentials of the third kind. Our main technical contribution is to determine the leading asymptotics of these classical objects and show that they are governed by the corresponding tropical objects introduced by Mikhalkin--Zharkov~\cite{MZ}. The Harnack property of the spectral curves plays a crucial role by ensuring that the arguments of the theta functions appearing in Fock's formula are real. This implies that the terms in the series expansions of the theta functions are positive, and so there are no cancellations among leading terms.

As a second application of our asymptotic results, we tropicalize Fay's trisecant identity. Again, the Harnack property implies that the three terms in the identity have definite signs. Thus, the tropicalized identity takes the stronger form that a term is equal to the maximum of two other terms rather than saying that maximum of three terms is attained at least twice.

Although we consider tropical curves arising from embedded spectral curves in this paper, the underlying degeneration and asymptotic results admit a more intrinsic formulation. We may instead start with a planar abstract tropical curve and construct a corresponding family of degenerating M-curves by plumbing. Thus, many of the tropical-geometric results below are naturally statements about planar abstract tropical curves rather than embedded tropical curves arising from dimers.

Tropical limits of the dimer model on a torus have been studied by Ma in her thesis~\cite{Ma2015} and by Bocklandt~\cite{bocklandt2015dimerabc,BocklandtStrebel}. We have several motivations for studying this tropicalization. Firstly, it is a natural step toward understanding action-angle coordinates for tropical cluster
integrable systems. As shown by
Inoue--Lam--Pylyavskyy~\cite{InoueLamPylyavskyy}, periodic box-ball
systems are examples of tropical cluster integrable systems. The dynamics of such tropical integrable systems can be
linearized on the Jacobian of a tropical spectral curve and expressed
in terms of tropical theta functions. Tropical theta functions and tropical Fay's
trisecant identity play an
important role in this theory; see, for example,
Inoue--Takenawa~\cite{InoueTakenawaSpectral,InoueTakenawaFay,InoueTakenawaJacobian},
Inoue--Iwao~\cite{InoueIwao}, and the review of
Inoue--Kuniba--Takagi~\cite{InoueKunibaTakagi}. Related degenerations of Riemann theta functions have also been studied by Agostini--Çelik--Struwe--Sturmfels~\cite{thetasurfaces}, Ichikawa~\cite{ichikwaCMP} and Agostini--Fevola--Mandelshtam--Sturmfels~\cite{AFMS} in connection with tropical limits of KP solutions. In forthcoming work with Galashin~\cite{GalashinGeorge2}, we will study
the tropicalization of the forward spectral transform for cluster integrable
systems.

Secondly, Fock's weights have recently found several applications in the statistical mechanics of dimers; see Boutillier--Cimasoni--de Tili\`ere~\cite{BDdT}, Berggren--Borodin~\cite{Berggren-Borodin-Aztec}, Bobenko--Bobenko~\cite{Bob2}, Boutillier--de Tili\`ere~\cite{BoutillierTiliereAztec}, and Berggren--Borodin--George~\cite{BerggrenBorodinGeorge}. Tropicalization has proved to be a powerful tool in the analysis of dimer limit shapes; see Berggren--Borodin~\cite{Berggren-Borodin-tropical} and Berggren--Borodin--George~\cite{BerggrenBorodinGeorge}. In particular, the appendix of~\cite{BerggrenBorodinGeorge} shows that suitable tropical Fock weights can be used to realize dimer models on more general bipartite graphs within the Aztec diamond.

Thirdly, in George--Galashin~\cite{GalashinGeorge2024}, we initiated
the study of the boundary stratification of a compactification of the
space of spectral data inspired by Postnikov's totally nonnegative
Grassmannian~\cite{postnikov2006total}. The polyhedral structure of the
positive tropical Grassmannian is closely related to the stratification
of the totally nonnegative Grassmannian; see
Speyer--Williams~\cite{SpeyerWilliams1,SpeyerWilliams2} and
Arkani-Hamed--Lam--Spradlin~\cite{ArkaniHamedLamSpradlin}. We hope that
understanding the positive tropical points of spectral data will shed
some light on the corresponding stratification of the compactification
of the space of spectral data.

Finally, we mention a recent work of Ichikawa that is related to the present work. In~\cite{IchikawaDegenerating}, Fock's dimer model is studied for families of M-curves degenerating by pinching real ovals. This degeneration is, in a sense, complementary to the one considered here. 

\subsection{Main results}

Let $\Gtor$ be a minimal periodic bipartite graph. In \cite{KO07b}, Kenyon--Okounkov constructed the spectral transform, a map
\begin{align*}
 \lambda: \{\text{positive edge weights on $\Gtor$}\}/\text{gauge} & \to \{\text{spectral data associated to $\Gtor$}\}
\end{align*}
assigning to each choice $[\wt]$ of positive edge weights modulo gauge equivalence on $\Gtor$ a spectral datum consisting of a Harnack curve $\mcurve(t)$,  a standard divisor $\divisor(t)$, and an angle map $\nu(t)$ assigning a point at infinity of $\mcurve(t)$ to each
zig-zag path of $\Gtor$.

Fock~\cite{Fock} constructed an inverse to the spectral transform which we now describe. The notation used below will be defined precisely in \cref{sec:backgroudd}. At this stage, the reader should regard the
formulas as displaying the structure of Fock's inverse spectral
transform and its tropical analogue. The important point is that the
classical objects appearing in Fock's formula are replaced, term by
term, by their tropical counterparts.\footnote{Fock's formula is usually stated for edge weights and in terms of the prime form. We choose to state it using wedges because it allows us to express the resulting ratio of prime forms in terms of differentials of the third kind which admit a simple tropicalization. The tropicalization of the prime form appears in the physics literature in~\cite{Tourkine}.} 

A \emph{wedge} in $\Gtor$ is a pair of edges that share a black vertex. A wedge $\mathfrak w$ determines three zig-zag paths
$\alpha,\beta_-,\beta_+$ and two adjacent faces
$\bface_-,\bface_+$, as shown in \cref{fig:wedge}. We define the wedge weight 
\begin{equation}
\label{eq:fock-weight-intro}
\wt_{\mathrm{Fock}}(\mathfrak w)
=
-\frac{
\theta\left(
\mu_{\basept}(\dam(\bface_-))
-\vecptjac+\vecRC
\middle|
\Pi
\right)
}{
\theta\left(
\mu_{\basept}(\dam(\bface_+))
-\vecptjac+\vecRC
\middle|
\Pi
\right)
}
\exp\left(
\int_{\basept}^{\alpha}
\omega_{\beta_+-\beta_-}
\right).
\end{equation}
Here, $\dam$ is Fock's discrete Abel map,
$\mu_{\basept(t)}$ is the Abel map on $\mcurve(t)$,
$\Pi(t)$ is its period matrix, $\theta$ is the Riemann theta function,
$\vecRC(t)$ is the vector of Riemann constants, and
$\omega_{\beta_+(t)-\beta_-(t)}(t)$ is the normalized differential of
the third kind with residue divisor
$\beta_+(t)-\beta_-(t)$. 

Any cycle $\gamma$ in $\Gtor$ can be written as a sum of wedges,
$
\gamma
=
\mathfrak w_1+\cdots+\mathfrak w_r.
$
We assign to $\gamma$ a cycle weight
\begin{equation}\label{eq:fock_intro_cycle}
   [\wt_{\mathrm{Fock}}](t)(\gamma) := \pm \prod_{j=1}^r  \wt_{\mathrm{Fock}}(t)(\mathfrak w_j).
\end{equation}
These cycle weights determine a
gauge-equivalence class of the edge weights $[\wt_{\mathrm{Fock}}]$, and thus we obtain a map
\[
\lambda^{-1}_{\mathrm{Fock}}:
\{
\text{spectral data associated to $\Gtor$}
\}
\rightarrow
\{
\text{positive edge weights on $\Gtor$}
\}/\text{gauge}.
\]
Fock proved that $\lambda^{-1}_{\mathrm{Fock}}  = \lambda^{-1}$.

Now let $[\wt](t)$ be a positive Puiseux-valued family of edge weights
modulo gauge equivalence on $\Gtor$, and let
\[
[\wt^\trop]:=\val([\wt](t))
\]
be its tropicalization. Applying the spectral transform gives a family
of spectral data
\[
\lambda([\wt](t))
=
(
\mcurve(t),\divisor(t),\nu(t)
).
\]

\begin{figure}
\centering
\includegraphics[
    width=.82\textwidth
]{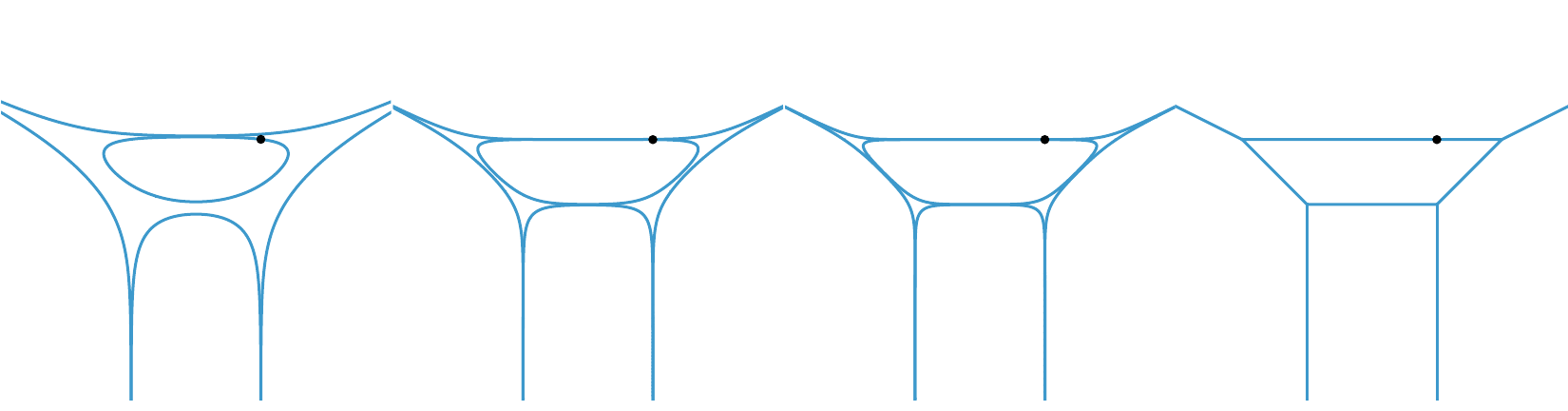}

\caption{
Convergence of the rescaled amoebas to the (embedded) tropical spectral
curve. The panels, from left to right and top to bottom, correspond to
$t=10$, $t=100$, $t=1000$, and the tropical limit. The black point is
the image of the standard divisor; its rescaled logarithmic image
converges to the corresponding point on the tropical spectral curve.
}
\label{fig:running-example-kapranov}
\end{figure}

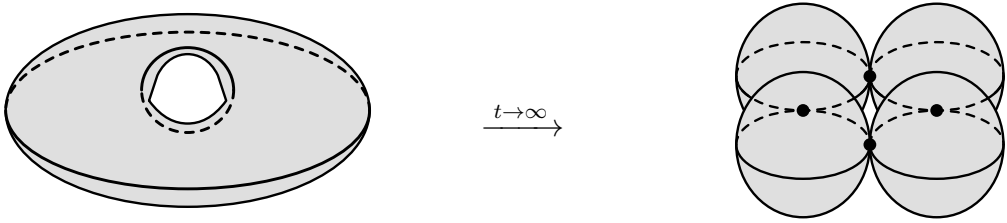
\begin{figure}
\centering

\begin{minipage}[c]{0.42\textwidth}
\centering

\begin{tikzpicture}[
  xscale=0.58,
  yscale=1.08,
  >=Stealth,
  line cap=round,
  line join=round
]

\tikzset{
  Bfront/.style={
    line width=1.1pt,
  },
  Bouterfront/.style={
    line width=1.1pt,
  },
  Bback/.style={
    line width=1.1pt,
    dashed
  }
}

\filldraw[
  fill=gray!25,
  draw=black,
  line width=0.9pt,
  even odd rule
]

(5.00,0.28)
  ellipse[
    x radius=4.15,
    y radius=1.18
  ]

(4.12,0.40)
  -- (4.31,0.71)
  .. controls (4.49,0.91) and (4.78,0.97) .. (5.00,0.97)
  .. controls (5.22,0.97) and (5.51,0.91) .. (5.69,0.71)
  -- (5.88,0.40)
  .. controls (5.58,0.18) and (5.27,0.12) .. (5.00,0.12)
  .. controls (4.73,0.12) and (4.42,0.18) .. cycle;


\draw[Bouterfront]
  (0.85,0.28)
  arc[
    start angle=180,
    end angle=360,
    x radius=4.15,
    y radius=0.96
  ];

\draw[Bback]
  (9.15,0.28)
  arc[
    start angle=0,
    end angle=180,
    x radius=4.15,
    y radius=0.96
  ];


\draw[Bfront]
  (3.95,0.53)
  arc[
    start angle=180,
    end angle=0,
    x radius=1.05,
    y radius=0.52
  ];

\draw[Bback]
  (6.05,0.53)
  arc[
    start angle=0,
    end angle=-180,
    x radius=1.05,
    y radius=0.52
  ];

\end{tikzpicture}

\end{minipage}
%
\begin{minipage}[c]{0.10\textwidth}
\centering
\[
\xrightarrow{\;t\to\infty\;}
\]
\end{minipage}
%
\begin{minipage}[c]{0.44\textwidth}
\centering

\begin{tikzpicture}[
  xscale=0.68,
  yscale=1.13,
  >=Stealth,
  line cap=round,
  line join=round
]

\tikzset{
  surface/.style={
    fill=gray!25,
    draw=black,
    line width=0.9pt
  },
  front/.style={
    black,
    line width=0.9pt
  },
  back/.style={
    black,
    line width=0.9pt,
    dashed
  },
  hiddenequator/.style={
    black,
    line width=0.9pt,
    dashed
  },
  nodept/.style={
    circle,
    fill=black,
    inner sep=1.7pt
  }
}


\filldraw[surface]
  (2.00,0.65)
  ellipse[
    x radius=1.30,
    y radius=0.85
  ];

\draw[back]
  (0.70,0.65)
  arc[
    start angle=180,
    end angle=0,
    x radius=1.30,
    y radius=0.40
  ];

\draw[front]
  (0.70,0.65)
  arc[
    start angle=180,
    end angle=360,
    x radius=1.30,
    y radius=0.40
  ];

\filldraw[surface]
  (4.60,0.65)
  ellipse[
    x radius=1.30,
    y radius=0.85
  ];

\draw[back]
  (3.30,0.65)
  arc[
    start angle=180,
    end angle=0,
    x radius=1.30,
    y radius=0.40
  ];

\draw[front]
  (3.30,0.65)
  arc[
    start angle=180,
    end angle=360,
    x radius=1.30,
    y radius=0.40
  ];


\filldraw[surface]
  (2.00,-0.15)
  ellipse[
    x radius=1.30,
    y radius=0.85
  ];

\draw[back]
  (0.70,-0.15)
  arc[
    start angle=180,
    end angle=0,
    x radius=1.30,
    y radius=0.40
  ];

\draw[front]
  (0.70,-0.15)
  arc[
    start angle=180,
    end angle=360,
    x radius=1.30,
    y radius=0.40
  ];

\filldraw[surface]
  (4.60,-0.15)
  ellipse[
    x radius=1.30,
    y radius=0.85
  ];

\draw[back]
  (3.30,-0.15)
  arc[
    start angle=180,
    end angle=0,
    x radius=1.30,
    y radius=0.40
  ];

\draw[front]
  (3.30,-0.15)
  arc[
    start angle=180,
    end angle=360,
    x radius=1.30,
    y radius=0.40
  ];


\draw[hiddenequator]
  (0.70,0.65)
  arc[
    start angle=180,
    end angle=360,
    x radius=1.30,
    y radius=0.40
  ];

\draw[hiddenequator]
  (3.30,0.65)
  arc[
    start angle=180,
    end angle=360,
    x radius=1.30,
    y radius=0.40
  ];


\node[nodept] at (3.30,0.65) {};
\node[nodept] at (3.30,-0.15) {};
\node[nodept] at (2.00,0.25) {};
\node[nodept] at (4.60,0.25) {};

\end{tikzpicture}

\end{minipage}

\caption{
A degeneration of a smooth genus-one curve to a nodal curve consisting
of four genus-zero components.
}
\label{fig:genus-one-degeneration}
\end{figure}

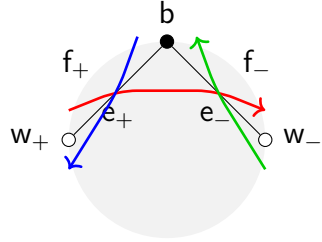
\begin{figure}
\centering
\begin{tikzpicture}[scale = 1.3]
  \def\ep{0.3}
  \def\rc{7}
  \def\lw{1pt}
  \def\r{1}

  \fill[black!5] (0,0) circle (\r cm);

  \coordinate[bvert, label=above:${\bb}$] (w) at (90:\r);
  \coordinate[wvert, label=right:${\bw_-}$] (b1) at (0:\r);
  \coordinate[wvert, label=left:${\bw_+}$] (b2) at (180:\r);

  \path let \p1 = (b1), \p2 = (w) in
    coordinate (m1) at ($(\p1)!.5!(\p2)$);
  \path let \p1 = (b2), \p2 = (w) in
    coordinate (m2) at ($(\p1)!.5!(\p2)$);

  \node[below=2pt] at (m1) {${\be_-}$};
  \node[below=2pt] at (m2) {${\be_+}$};

  \node at (40:1.2) {${\bface_-}$};
  \node at (140:1.2) {${\bface_+}$};

  \draw[-] (b1) -- (w) -- (b2);

  \draw[red,<-,line width=\lw,rounded corners=\rc]
    (1,\ep) -- (m1) -- (m2) -- (-1,\ep);

  \draw[green!80!black,<-,line width=\lw,rounded corners=\rc]
    (0.3,1.05) -- (m1) -- (1,-\ep);

  \draw[blue,->,line width=\lw,rounded corners=\rc]
    (-0.3,1.05) -- (m2) -- (-1,-\ep);
\end{tikzpicture}

\caption{
A wedge in $\Gtor$ defined by two edges
$\be_-=\bb\bw_-$ and $\be_+=\bb\bw_+$ sharing a black vertex $\bb$.
The zig-zag paths $\alpha$, $\beta_-$, and $\beta_+$ are drawn in
red, green, and blue respectively.
}
\label{fig:wedge}
\end{figure}

As $t\to\infty$, the rescaled amoebas of $\mcurve(t)$ converge to an
embedded tropical spectral curve
$\tropmcurve_{\mathrm{emb}}$; see
\cref{fig:running-example-kapranov}. More geometrically, the family $\mcurve(t)$ together with the angle map
forms a family of pointed curves degenerating to a stable nodal curve
$\mcurve(\infty)$ whose irreducible components are themselves Harnack
curves; see \cref{fig:genus-one-degeneration}. Such degenerations were
also studied by Olarte~\cite{Olarte} in his construction of the
compactification of the moduli space of Harnack curves. Associated to $\tropmcurve_{\mathrm{emb}}$ is an abstract tropical
curve
\[
(\tropmcurve,\ell,\bm g).
\]
Its underlying graph $\tropmcurve$ is obtained from
$\tropmcurve_{\mathrm{emb}}$ by contracting the unbounded edges and
replacing edges of multiplicity greater than one by parallel edges, and
records the combinatorics of the degeneration. The edge lengths $\ell$
record the rates at which the corresponding nodes are smoothed, while
the vertex genera $\bm g$ record the genera of the irreducible
components of $\mcurve(\infty)$.

The standard divisor tropicalizes to an effective divisor
$
\divisor^\trop
$
on $\tropmcurve$ (\cref{fig:running-example-kapranov}) and the angle map tropicalizes to a tropical angle map
$
\nu^\trop:\zzG\rightarrow\vertices(\tropmcurve).
$
Thus, the family of spectral data determines a tropical spectral datum
$
(
(\tropmcurve,\ell,\bm g),
\divisor^\trop,
\nu^\trop
).
$

Our main technical results determine the leading asymptotics of every
classical object appearing in \eqref{eq:fock-weight-intro}. We obtain
the asymptotics of the period matrix and Abel map using the results of
Hu--Norton~\cite{HuNorton}, and also determine the asymptotics of the
normalized differentials of the third kind, the vector of Riemann
constants, and the theta function. Their leading terms are governed by
the corresponding tropical objects.

The Harnack property is particularly important for the theta function:
the theta arguments appearing in \eqref{eq:fock-weight-intro} are real,
so the terms of the theta series are positive and no cancellation can
occur among terms of maximal order.

Combining these asymptotics gives a tropical version of Fock's formula:
\[
\begin{aligned}
\wt_{\mathrm{Fock}}^\trop(\mathfrak w)
:={}&
\theta^\trop\left(
\mu_{\baseptT}^{\trop}
    (\dam^\trop(\bface_-))
-\mu_{\baseptT}^{\trop}
    (\divisor^\trop)
+\vecRC^{\trop}
\middle|
\Pi^{\trop}
\right)
\\
&-
\theta^\trop\left(
\mu_{\baseptT}^{\trop}
    (\dam^\trop(\bface_+))
-\mu_{\baseptT}^{\trop}
    (\divisor^\trop)
+\vecRC^{\trop}
\middle|
\Pi^{\trop}
\right)
\\
&-
\int_{\baseptT}^{\alpha^\trop}
\omega_{\beta_+^\trop-\beta_-^\trop}^\trop.
\end{aligned}
\]
Here and throughout, we use a $\dagger$ in the superscript to denote the tropical counterpart of a classical object. If
$
\gamma=\mathfrak w_1+\cdots+\mathfrak w_r
$
is a cycle, set
\begin{equation}
\label{eq:tropical-fock-weight-intro}
[\wt_{\mathrm{Fock}}^\trop](\gamma)
:=
\sum_{j=1}^r
\wt_{\mathrm{Fock}}^\trop(\mathfrak w_j).
\end{equation}
These cycle weights determine a tropical gauge-equivalence class of
edge weights. Thus, we obtain a map
\[
(\lambda^{-1}_{\mathrm{Fock}})^\trop:
\{
\text{tropical spectral data associated to $\Gtor$}
\}
\rightarrow
\{
\text{tropical edge weights on $\Gtor$}
\}/\text{gauge}.
\]

Our first main result is:

\begin{theorem}[cf.~\cref{thm:fock_weight_tropicalization}]
\label{thm:main}
Let $[\wt](t)$ be $\Puispos$-valued edge weights modulo gauge
equivalence satisfying a genericity condition (Assumption~\ref{ass:generic-leading-coefficients}), and let
$
\lambda([\wt](t))
=
(
\mcurve(t),\divisor(t),\nu(t)
).
$
Let
$
(
(\tropmcurve,\ell,\bm g),
\divisor^\trop,
\nu^\trop
)
$
be the corresponding tropical spectral datum. Then,
\[
(\lambda^{-1}_{\mathrm{Fock}})^\trop
(
(\tropmcurve,\ell,\bm g),
\divisor^\trop,
\nu^\trop
)
=
[\wt^\trop].
\]
In other words, tropicalization commutes with Fock's inverse spectral
transform.
\end{theorem}

Our second main result is a tropicalization of Fay's trisecant identity~\cite{Fay}. 

\begin{theorem}[cf. \cref{thm:tropical_fay}]
Let
$
\ptSig^\trop_1,\ptSig^\trop_3,\ptSig^\trop_2,\ptSig^\trop_4
$
be four vertices in cyclic order along the outer face of $\tropmcurve$. Let $\vecptjac^\trop \in \R^{g^\trop}$, let $\bm \eta_{ij}^\trop
:=\mu_{\baseptT}^{\trop}(\ptSig_j^\trop-\ptSig_i^\trop)$, and let 
\[
\begin{aligned}
F_1^\trop
:={}&
\theta^\trop(
\vecptjac^\trop+\bm \eta_{13}^\trop
\mid
\Pi^{\trop}
)
+
\theta^\trop(
\vecptjac^\trop+\bm \eta_{24}^\trop
\mid
\Pi^{\trop}
)
-
\int_{\ptSig_1^\trop}^{\ptSig_3^\trop}
\omega_{\ptSig_2^\trop-\ptSig_4^\trop}^{\trop},
\\
F_2^\trop
:={}&
\theta^\trop(
\vecptjac^\trop+\bm \eta_{23}^\trop
\mid
\Pi^{\trop}
)
+
\theta^\trop(
\vecptjac^\trop+\bm \eta_{14}^\trop
\mid
\Pi^{\trop}
)
-
\int_{\ptSig_2^\trop}^{\ptSig_3^\trop}
\omega_{\ptSig_1^\trop-\ptSig_4^\trop}^{\trop},
\\
F_3^\trop
:={}&
\theta^\trop(
\vecptjac^\trop
\mid
\Pi^{\trop}
)
+
\theta^\trop(
\vecptjac^\trop
+
\bm \eta_{13}^\trop
+
\bm \eta_{24}^\trop
\mid
\Pi^{\trop}
).
\end{aligned}
\]
Then,
\[
F_3^\trop
=
\max
\{
F_1^\trop,
F_2^\trop
\}.
\]
\end{theorem}

Tropical versions of Fay's identity were first proved by Inoue--Takenawa~\cite{InoueTakenawaFay} and later generalized by Inoue--Iwao~\cite{InoueIwao2011} to smooth planar tropical curves. Our formulation allows general planar tropical curves. However, unlike their formulations, we restrict the points appearing in the identity to the boundary of the outer face. This restriction implies positivity of the terms appearing in Fay's identity and therefore removes the sign data appearing in the formulations of \cite{InoueTakenawaFay,InoueIwao2011}.

\subsection{Organization of the paper}

In \cref{sec:backgroudd}, we recall the dimer spectral transform and
Fock's inverse in terms of spectral data on Harnack curves. In \cref{sec:trp_spec_curve}, we associate embedded and abstract tropical spectral
curves to positive Puiseux-valued edge weights and introduce the
tropical differentials, Abel map, and Jacobian that will appear in the
tropicalization of Fock's formula. In \cref{sec:plumbing}, we describe
the degeneration of the spectral curve, first as a toric degeneration
and then in plumbing coordinates, and construct the
degeneration-adapted paths and symplectic basis used below.
In \cref{sec:asymptotics}, we determine the asymptotics of the period
matrix, the Abel map, differentials of the third kind, the vector of Riemann
constants, and theta function. In \cref{sec:tropical_fock_weights}, we combine these results to tropicalize Fock's inverse spectral transform and prove \cref{thm:main}. Finally, in \cref{sec:tropical_fay}, we apply the asymptotic results to obtain a tropical version of Fay's trisecant identity.

\subsection*{Acknowledgements}

This work grew out of a joint project with Pavel Galashin, and the author is grateful to him for many stimulating discussions. Part of this paper was written while the author was visiting the Indian Statistical Institute, Bangalore, and the author thanks the institute for its hospitality. The author also thanks Tomas Berggren, Alexei Borodin and David Speyer for helpful discussions.

\section{Dimer spectral data and Fock weights} \label{sec:backgroudd}

\subsection{Bipartite graphs on a torus}
\label{sec:bip_graph_torus}

Let $\Gtor \subset \torus:=\R^2/\Z^2$ be a bipartite graph embedded such that every face is an open disk, and let $\G \subset \R^2$ be the corresponding periodic graph in the plane. 

An \emph{edge weight} on $\Gtor$ is a function $\wt:\edges(\Gtor) \to \Rpos$.  
Two edge weights $\wt_1$ and $\wt_2$ are \emph{gauge equivalent} if there exists a function 
\[
f:\blackvertices(\Gtor) \sqcup \whitevertices(\Gtor) \rightarrow \Rpos
\]
such that
\[
\wt_2(\bb \bw) = f(\bb)^{-1} \wt_1(\bb \bw) f(\bw) \qquad \text{for all $\bb \bw \in \edges(\Gtor)$}.
\]
The gauge equivalence class of $\wt$ is denoted by $[\wt]$.

For every cycle 
$\gamma = (\bb_1 \xrightarrow{\be_1} \bw_1 \xrightarrow{\be_2} \bb_2 \xrightarrow{\be_3} \cdots \xrightarrow{\be_{2k-2}} \bb_k \xrightarrow{\be_{2k-1}} \bw_k \xrightarrow{\be_{2k}} \bb_{k+1}=\bb_1)$ in $\Gtor$, define the \emph{cycle weight}
\begin{equation}\label{eq:wt_L_dfn}
[\wt](\gamma) := 
\frac{\wt(\be_1)\cdot \wt(\be_3)\cdots \wt(\be_{2k-1})}{\wt(\be_2)\cdot\wt(\be_4) \cdots \wt(\be_{2k})}.
\end{equation}

When $\gamma = \partial \bface$ is the (counterclockwise-oriented) boundary of a face $\bface \in \faces(\Gtor)$, we write
\[
X_{\bface}:= [\wt](\partial \bface)
\]
for the corresponding \emph{face weight}. The face weights $(X_{\bface})_{\bface \in \faces(\Gtor)}$ are the \emph{$X$-cluster variables} in the sense of Fock--Goncharov~\cite{FoGo_cluster}, or \emph{$y$-cluster variables} in the sense of Fomin--Zelevinsky~\cite{FZ}. The face weights together with the weights $\rho_z,\rho_w$ of two cycles $\gamma_z, \gamma_w$ that generate the homology $H_1(\torus,\Z) \cong \Z^2$ of the torus provide coordinates on the space of edge weights modulo gauge equivalence.

\subsection{Zig-zag paths, minimality and Newton polygon} \label{sec:NewtonPolygon}

A \emph{zig-zag path} in $\Gtor$ or $\G$ is an oriented path that turns maximally left at white vertices and maximally right at black vertices. We denote the set of zig-zag paths in $\Gtor$ by $\zzG$.

We say that $\Gtor$ is \emph{minimal} if, in $\G$,
\begin{itemize}
\item No zig-zag path is a closed loop.
\item No zig-zag path has a self-intersection.
\item No two zig-zag paths have two edges $\tilde \be_1 \neq \tilde \be_2$ in common, with both oriented from $\tilde \be_1$ to $\tilde \be_2$.
\end{itemize}
Hereafter, we restrict attention to minimal graphs.  

Every zig-zag path $\alpha \in \zzG$ is a closed cycle in $\Gtor$ and therefore determines a homology class $[\alpha] \in H_1(\torus,\Z) \cong \Z^2$. The
unique convex lattice polygon $N \subset \R^2$, defined modulo translation by $\Z^2$, whose sides consist of the
vectors $\{[\alpha] \in \Z^2 \mid \alpha \in \zzG\}$ in counterclockwise cyclic order is called the \emph{Newton polygon}
of $\Gtor$.

For $\alpha \in \zzG$, let $E^*(\alpha)$ denote the side of $N$ containing $[\alpha]$. Let $\No$ (resp. $\Nbd$) denote the interior (resp. boundary of $N$). Let $\NZ:=N \cap \Z^2$ (resp. $\NZo:= N^\circ \cap \Z^2$) denote the lattice points (resp. interior lattice points) of $N$.


\begin{example}

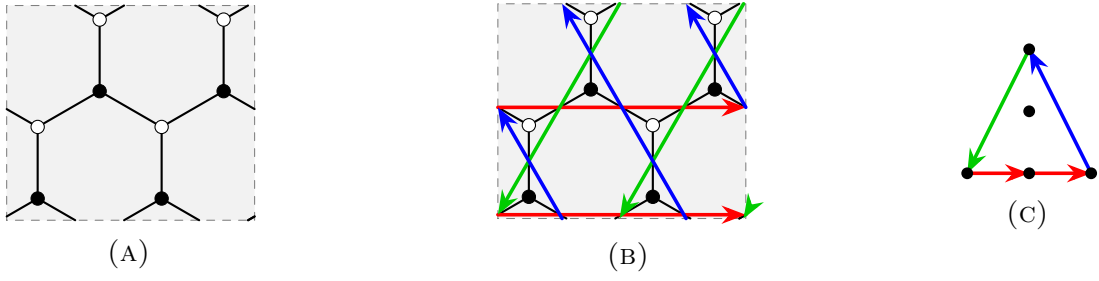
\begin{figure}
\centering

\begin{subfigure}[c]{0.36\textwidth}
\centering

\begin{tikzpicture}[
    scale=.82,
    line cap=round,
    line join=round,
    >=Stealth
]

\draw[black!50,dashed,fill=black!5]
    (0,0) rectangle (4,3.464);


\draw[line width=.8pt] (0.000,1.793) -- (0.500,1.505);
\draw[line width=.8pt] (0.500,0.350) -- (0.000,0.061);

\draw[line width=.8pt] (3.500,2.082) -- (4.000,1.793);
\draw[line width=.8pt] (4.000,0.061) -- (3.894,0.000);

\draw[line width=.8pt] (1.106,3.464) -- (1.500,3.237);
\draw[line width=.8pt] (1.894,3.464) -- (1.500,3.237);
\draw[line width=.8pt] (3.106,3.464) -- (3.500,3.237);
\draw[line width=.8pt] (3.894,3.464) -- (3.500,3.237);

\draw[line width=.8pt] (0.500,0.350) -- (1.106,0.000);
\draw[line width=.8pt] (2.500,0.350) -- (1.894,0.000);
\draw[line width=.8pt] (2.500,0.350) -- (3.106,0.000);
\draw[line width=.8pt] (4.000,0.061) -- (3.894,0.000);

\draw[line width=.8pt] (0.500,0.350) -- (0.500,1.505);
\draw[line width=.8pt] (2.500,0.350) -- (2.500,1.505);
\draw[line width=.8pt] (1.500,2.082) -- (1.500,3.237);
\draw[line width=.8pt] (3.500,2.082) -- (3.500,3.237);

\draw[line width=.8pt] (1.500,2.082) -- (0.500,1.505);
\draw[line width=.8pt] (1.500,2.082) -- (2.500,1.505);
\draw[line width=.8pt] (3.500,2.082) -- (2.500,1.505);

\node[bvert] at (0.500,0.350) {};
\node[bvert] at (2.500,0.350) {};
\node[bvert] at (1.500,2.082) {};
\node[bvert] at (3.500,2.082) {};

\node[wvert] at (0.500,1.505) {};
\node[wvert] at (2.500,1.505) {};
\node[wvert] at (1.500,3.237) {};
\node[wvert] at (3.500,3.237) {};

\end{tikzpicture}

\caption{}
\label{fig:genus-one-graph}
\end{subfigure}
\hfill
\begin{subfigure}[c]{0.36\textwidth}
\centering

\begin{tikzpicture}[
    scale=.82,
    line cap=round,
    line join=round,
    >=Stealth
]

\draw[black!50,dashed,fill=black!5]
    (0,0) rectangle (4,3.464);


\draw[line width=.8pt] (0.000,1.793) -- (0.500,1.505);
\draw[line width=.8pt] (0.500,0.350) -- (0.000,0.061);

\draw[line width=.8pt] (3.500,2.082) -- (4.000,1.793);
\draw[line width=.8pt] (4.000,0.061) -- (3.894,0.000);

\draw[line width=.8pt] (1.106,3.464) -- (1.500,3.237);
\draw[line width=.8pt] (1.894,3.464) -- (1.500,3.237);
\draw[line width=.8pt] (3.106,3.464) -- (3.500,3.237);
\draw[line width=.8pt] (3.894,3.464) -- (3.500,3.237);

\draw[line width=.8pt] (0.500,0.350) -- (1.106,0.000);
\draw[line width=.8pt] (2.500,0.350) -- (1.894,0.000);
\draw[line width=.8pt] (2.500,0.350) -- (3.106,0.000);
\draw[line width=.8pt] (4.000,0.061) -- (3.894,0.000);

\draw[line width=.8pt] (0.500,0.350) -- (0.500,1.505);
\draw[line width=.8pt] (2.500,0.350) -- (2.500,1.505);
\draw[line width=.8pt] (1.500,2.082) -- (1.500,3.237);
\draw[line width=.8pt] (3.500,2.082) -- (3.500,3.237);

\draw[line width=.8pt] (1.500,2.082) -- (0.500,1.505);
\draw[line width=.8pt] (1.500,2.082) -- (2.500,1.505);
\draw[line width=.8pt] (3.500,2.082) -- (2.500,1.505);


\draw[red,line width=1.35pt,<-]
    (4.000,1.793) -- (0.000,1.793);

\draw[red,line width=1.35pt,<-]
    (4.000,0.061) -- (0.000,0.061);

\draw[green!80!black,line width=1.35pt,<-]
    (0.000,0.061) -- (1.964,3.464);

\draw[green!80!black,line width=1.35pt,<-]
    (1.964,0.000) -- (3.964,3.464);

\draw[green!80!black,line width=1.35pt,<-]
    (3.964,0.000) -- (4.000,0.061);

\draw[blue,line width=1.35pt,<-]
    (0.000,1.793) -- (1.036,0.000);

\draw[blue,line width=1.35pt,<-]
    (1.036,3.464) -- (3.036,0.000);

\draw[blue,line width=1.35pt,<-]
    (3.036,3.464) -- (4.000,1.793);

\node[bvert] at (0.500,0.350) {};
\node[bvert] at (2.500,0.350) {};
\node[bvert] at (1.500,2.082) {};
\node[bvert] at (3.500,2.082) {};

\node[wvert] at (0.500,1.505) {};
\node[wvert] at (2.500,1.505) {};
\node[wvert] at (1.500,3.237) {};
\node[wvert] at (3.500,3.237) {};

\end{tikzpicture}

\caption{}
\label{fig:genus-one-zigzags}
\end{subfigure}
\hfill
\begin{subfigure}[c]{0.20\textwidth}
\centering

\begin{tikzpicture}[
    scale=.82,
    >=Stealth,
    line cap=round,
    line join=round
]


\draw[blue,line width=1.35pt,<-]
    (0,2) -- (1,0);

\draw[red,line width=1.35pt,<-]
    (1,0) -- (0,0);

\draw[red,line width=1.35pt,<-]
    (0,0) -- (-1,0);

\draw[green!80!black,line width=1.35pt,<-]
    (-1,0) -- (0,2);

\foreach \p in {
    (-1,0),
    (0,0),
    (1,0),
    (0,2),
    (0,1)
}{
    \fill[black] \p circle (2.7pt);
}

\end{tikzpicture}

\caption{}
\label{fig:genus-one-newton}
\end{subfigure}

\caption{
A periodic bipartite graph~(A), its zig-zag paths (drawn, as is customary, as paths in the medial graph)~(B), and its Newton polygon~(C).
}\label{fig:running-example-zigzags-genus-one}
\end{figure}

We will use as a running example the periodic bipartite graph in
\cref{fig:running-example-zigzags-genus-one}(a).
Its zig-zag paths are shown in
\cref{fig:running-example-zigzags-genus-one}(b), and its Newton polygon in \cref{fig:running-example-zigzags-genus-one}(c).
\end{example}

\subsection{Kasteleyn theory on the torus}

Choose a fundamental rectangle for the torus with horiztonal (resp. vertical) side $s_z$ (resp. $s_w$). Let $\be=\bb\bw$ be an edge of $\Gtor$ oriented from $\bb$ to $\bw$. Let
\[
 \hom[\be] := (\be \wedge s_w,-\be\wedge s_z)\in\Z^2, 
\]
where $\be \wedge s_w$ and $\be\wedge s_z$ are the signed intersection numbers of $\be$ with the  sides of the fundamental rectangle. For $\ptC = (z,w) \in \C^2$ and $\pt = (j,k) \in \Z^2$, we write 
\[
\ptC^{\pt} = z^j w^k
\]
for the corresponding monomial. Define the \emph{Kasteleyn matrix} $\bK(\ptC): \C^{\blackvertices(G)} \to 
\C^{\whitevertices(G)}$ by 
\[
\bK_{\bw \bb}(\ptC) :=
\sum_{\be=\bb\bw} \wt(\be)\kappa(\be)
\ptC^{\hom[\be]},
\]
where the sum is over all edges with endpoints $\bb$ and $\bw$, and $\kappa:\edges(\Gtor) \to \{\pm 1\}$ is a \emph{Kasteleyn sign}. The \emph{characteristic polynomial} is the Laurent polynomial in $z,w$ defined by
\[
P(\ptC) := \det \bK(\ptC).
\]
By {\cite[Theorem~3.12]{GK}}, its Newton polygon, i.e., the convex hull of the exponents of monomials with nonzero coefficients in $P(\ptC)$ agrees with the Newton polygon $N$ of $G$ defined in \cref{sec:NewtonPolygon}.

The \emph{open spectral curve} is the vanishing locus
\[
\mcurve^\circ := \{\ptC \in (\C^\times)^2 \mid P(\ptC) = 0\}.
\]
The open spectral curve is not compact. Toric geometry provides a standard way to compactify it; see e.g.~\cite{GKZ} for more details. The Newton polygon $N$ defines a compactification of $(\C^\times)^2$ given by
\[
\toricsurface := 
\overline{ \{ [\ptC^\pt]_{\pt \in \NZ} \mid \ptC \in (\C^\times)^2 \}} 
\subseteq \mathbb P^{|\NZ|-1},
\]
called the \emph{toric surface} associated with $N$.
Its boundary decomposes as
\[
\toricsurface \setminus (\C^\times)^2 
= \bigcup_{\text{sides $\Eside$ of $N$}} \partial V^{\mathrm{tor}}_\Eside,
\]
where for each side $\Eside$ of $N$, 
$\partial V^{\mathrm{tor}}_\Eside:=(\toricsurface\cap\Pbb^{|\Eside|_\Z}) \cong \C\mathbb P^1$
 is called a \emph{line at infinity}.
\noindent Here and below, we denote by $|\Eside|_\Z:=|\Eside\cap\Z^2|-1$ the \emph{integral length} of $\Eside$. 

The \emph{spectral curve} $\mcurve$ is defined as the closure of $\openspectralcurve$ inside $\toricsurface$. The points in $\spectralcurve \setminus \openspectralcurve$ are called 
\emph{points at infinity} of $\openspectralcurve$. The intersection $\spectralcurve \cap \partial V^{\mathrm{tor}}_\Eside$ contains $|\Eside|_\Z$ points, counted with multiplicity.

\begin{example}

\label{ex:running_graph_zigzags_genus_one}

\begin{figure}

\begin{tikzpicture}[
    line cap=round,
    line join=round,
    >=Stealth
]

%

\draw[black!50,dashed,fill=black!5]
    (0,0) rectangle (4,3.464);

%
%
%
%

\draw
    (0.000,1.793) -- (0.500,1.505);

\draw
    (3.500,2.082)
    -- node[right] {$\displaystyle \frac{\rho_z}{X_1X_2}$}
    (4.000,1.793);


\draw
    (0.500,0.350)
    -- node[left] {$\displaystyle
        \frac{X_1X_2X_3}{\rho_z\rho_w}$}
    (0.000,0.061);

\draw
    (4.000,0.061) -- (3.894,0.000);

\draw
    (3.894,3.464) -- (3.500,3.237);



\draw
    (0.500,0.350)
    -- node[below] {$\displaystyle \frac{1}{\rho_w}$}
    (1.106,0.000);

\draw
    (1.106,3.464) -- (1.500,3.237);


\draw
    (2.500,0.350)
    -- node[below] {$\displaystyle \frac{1}{\rho_w X_1}$}
    (1.894,0.000);

\draw
    (1.894,3.464) -- (1.500,3.237);


\draw
    (2.500,0.350)
    -- node[below right] {$\displaystyle
        \frac{X_2X_3}{\rho_w}$}
    (3.106,0.000);

\draw
    (3.106,3.464) -- (3.500,3.237);


\draw
    (0.500,0.350)
    --
    (0.500,1.505);

\draw
    (2.500,0.350)
    --
    (2.500,1.505);

\draw
    (1.500,2.082)
    --
    (1.500,3.237);

\draw
    (3.500,2.082)
    --
    (3.500,3.237);


\draw
    (1.500,2.082)
    --
    (0.500,1.505);

\draw
    (1.500,2.082)
    --
    (2.500,1.505);

\draw
    (3.500,2.082)
    --
    (2.500,1.505);


\coordinate[bvert,label=right:$\mathsf{b}_1$]
    (b1) at (0.500,0.350);

\coordinate[bvert,label=right:$\mathsf{b}_2$]
    (b2) at (2.500,0.350);

\coordinate[bvert,label=below:$\mathsf{b}_3$]
    (b3) at (1.500,2.082);

\coordinate[bvert,label=below:$\mathsf{b}_4$]
    (b4) at (3.500,2.082);


\coordinate[wvert,label=above:$\mathsf{w}_1$]
    (w1) at (0.500,1.505);

\coordinate[wvert,label=above:$\mathsf{w}_2$]
    (w2) at (2.500,1.505);

\coordinate[wvert,label=above:$\mathsf{w}_3$]
    (w3) at (1.500,3.237);

\coordinate[wvert,label=above:$\mathsf{w}_4$]
    (w4) at (3.500,3.237);


\node[blue] at (1.500,0.925) {$\bface_1$};

\node[blue] at (3.500,0.925) {$\bface_2$};

\node[blue] at (0.500,2.660) {$\bface_3$};

\node[blue] at (2.500,2.660) {$\bface_4$};

\end{tikzpicture}
\caption{A choice of edge weights for the graph in \cref{fig:running-example-zigzags-genus-one}.}\label{fig:running_graph_zigzags_genus_one_K}
\end{figure}
For our running example, let $X_i:=X_{\bface_i}$ and let $\rho_z$ and $\rho_w$ be as in \cref{sec:bip_graph_torus}. A choice of edge weights representing the gauge equivalence class is shown in \cref{fig:running_graph_zigzags_genus_one_K}. Since every face is a hexagon, the Kasteleyn sign may be chosen to be identically equal to $1$. With these choices, the Kasteleyn matrix is
\[
\bK(z,w)=
\begin{blockarray}{ccccc}
& \bb_1 & \bb_2 & \bb_3 & \bb_4\\
\begin{block}{c[cccc]}
\bw_1 &
1 & 0 & 1 &
{\rho_z z}/{X_1X_2}
\\
\bw_2 &
0 & 1 & 1 & 1
\\
\bw_3 &
1/{\rho_w w} &
1/{X_1\rho_w w} &
1 & 0
\\
\bw_4 &
{X_1X_2X_3}/{\rho_z\rho_w zw} &
{X_2X_3}/{\rho_w w} &
0 & 1
\\
\end{block}
\end{blockarray},
\]
and the characteristic polynomial is
\[
{
P(z,w)
=
1
-
\frac{
1+\frac{1}{X_1}+X_3+X_2X_3
}{
\rho_w w
}
+
\frac{
X_2X_3+\frac{X_3}{X_1}
}{
\rho_w^2 w^2
}
-
\frac{
X_2X_3
}{
\rho_z\rho_w^2 z w^2
}
-
\frac{
X_3\rho_z z
}{
X_1\rho_w^2 w^2
}.
}
\]
Its Newton polygon agrees with \cref{fig:running-example-zigzags-genus-one}(C).
\end{example}

\subsection{Amoebas and Harnack curves}\label{sec:Harnack}

The \emph{amoeba} of the open spectral curve $\openspectralcurve$ is the image
\[
\amoebaSig := 
\Log(\openspectralcurve),
\quad\text{where}\quad
\Log : (\C^\times)^2 \to \R^2,\quad 
(z,w) \mapsto (\log|z|, \log|w|).
\]

The points at infinity of $\openspectralcurve$ correspond to \emph{tentacles} of the amoeba. If $p \in \spectralcurve \cap \partial V^{\mathrm{tor}}_\Eside$, then the corresponding tentacle is directed along the outward-pointing normal to $\Eside$.  

Amoebas were introduced by Gelfand--Kapranov--Zelevinsky in the early 1990s \cite{GKZ} and have since become a central tool for connecting complex algebraic geometry with tropical and convex geometry; see for example the survey~\cite{MikhalkinAmoebas}. Their importance for the dimer model was recognized by Kenyon--Okounkov--Sheffield \cite{KOS}.

A real algebraic curve $\openspectralcurve$  defined by a real Laurent polynomial $P(\ptC)$ with Newton polygon $N$ is called a \emph{(simple) Harnack curve} if
\[
\Log : \openspectralcurve \rightarrow \amoebaSig
\]
is at most $2$--$1$. $\Log$ is then $2$--$1$ over the interior of $\amoebaSig$ and $1$--$1$ over the boundary which coincides with $\Log(\mcurve(\R))$. Harnack curves are classical objects; see~\cite{MikhalkinReal, MikhalkinRullgaard} for modern developments.

\begin{theorem}[Kenyon--Sheffield~\cite{KS}; see also~\cite{KOS,KenyonOkounkovHarnack}]
\label{thm:spectral_curve_is_harnack}
The spectral curve of a dimer model with positive edge weights is a Harnack curve.
\end{theorem}

\subsection{Abel map and Jacobian}\label{sec:jacobian}

Let $\mcurve$ be a compact Riemann surface of genus $g$. Let \[(\bm A=(A_1,\dots,A_g), \bm B=(B_1,\dots,B_g))\] be a symplectic basis of $H_1(\mcurve,\Z)$, i.e., a basis satisfying
\begin{equation}
A_j \wedge A_k = B_j \wedge  B_k = 0, \qquad  A_j \wedge B_k = \delta_{jk}, \label{eq:AwedgeB}
\end{equation}
where $\wedge$ denotes the intersection pairing. For $\bm C = (C_1,\dots,C_g)$ and $\bm D = (D_1,\dots, D_g)$ in $H_1(\mcurve,\Z)$, we introduce the vector notation for the $g \times g$ matrix
\[
\bm C \wedge \bm D = \left( C_j \wedge D_k \right)_{j,k=1}^g, 
\]
so that \eqref{eq:AwedgeB} can be written as
\[
\bm A \wedge \bm A = \bm B \wedge \bm B=0, \qquad \bm A \wedge \bm B = I.
\]

Let $\bm \omega=(\omega_1,\dots,\omega_g)
$ be the corresponding dual basis of holomorphic $1$-forms on
$\mcurve$, with the normalization
\[
\int_{\bm A} \bm \omega =
2\pi\imunit I,
\]
where, for a tuple of cycles $\bm C = (C_1,\dots,C_g)$ and $1$-forms $\bm \omega = (\omega_1,\dots,\omega_g)$, we introduce the vector notation for the $g \times g$ matrix
\[
\int_{\bm C} \bm \omega := \left(\int_{C_j}\omega_k \right)_{j,k=1}^g.
\]
The \emph{period matrix} $\Pi$ is the $g \times g$ symmetric matrix with negative-definite real part defined by
\[
\Pi := \int_{\bm B} \bm \omega.
\]
The \emph{Jacobian variety} of $\mcurve$ is the complex torus
\[
\jac(\mcurve) :=  \C^{g} \big/ \big(2\pi \imunit \Z^{g} \oplus \Pi \Z^{g}\big).
\]
We denote by $[\ptCg]$ the image of $\ptCg \in \C^{g}$ in $\jac(\mcurve)$.

A \emph{divisor} on $\mcurve$ is a finite formal sum of the form \[ \divisor = \sum_{j=1}^n c_j \divptgeom_j, \] where $\divptgeom_j \in \mcurve$ and $c_j \in \Z$. The integer $\sum_{j=1}^n c_j$ is called the \emph{degree} of $\divisor$. A divisor $\divisor$ is said to be \emph{effective} if $c_j \geq 0$ for all $j$. 

The \emph{symmetric product} of $\mcurve$ of degree $n$ is defined as \begin{equation}\label{eq:sym_dfn} \sym^n(\mcurve) := \mcurve^n / \mathfrak S_n, \end{equation} where the symmetric group $\mathfrak S_n$ acts by permuting the factors. Points of $\sym^n(\mcurve)$ correspond to effective divisors of degree $n$.

A \emph{differential of the third kind} on $\mcurve$ is a meromorphic $1$-form on $\mcurve$ with only simple poles. Given a degree-$0$ divisor $\divisor = \sum_{j=1}^n c_j \divptgeom_j$ on $\mcurve$, there exists a unique meromorphic differential $\omega_\divisor$ of the third kind with simple poles at $\divptgeom_1,\dots,\divptgeom_n$ and residues $c_1,\dots,c_n$, respectively, and vanishing $A$-periods:
\[	
\res_{\divptgeom_j}(\omega_\divisor) = c_j, \quad \text{for } 1\leq j\leq n,  \qquad \int_{\bm A} \omega_\divisor = 0; 
\]
see \cite[Chapter~I]{Fay}. We call $\omega_\divisor$ the \emph{$A$-normalized differential of the third kind with residue divisor $\divisor$}.

Fixing a basepoint $\basept \in \mcurve$, the \emph{Abel map}
\[
\mu_\basept : \mcurve \rightarrow \jac(\mcurve)
\]
is defined by
\[
\mu_\basept(\ptSig) :=  \int_{\basept}^{\ptSig} \bm \omega  = \left(\int_\basept^{\ptSig} \omega_j \right)_{j=1}^g \ \ \text{mod } 2\pi \imunit\Z^{g} \oplus \Pi \Z^{g}.
\]
The Abel map extends linearly to divisors.

\subsection{Theta function and Riemann's theorem} \label{sec:theta_fn}

The \emph{theta function} associated to $\mcurve$ is the holomorphic function
\[
\theta(\cdot \mid \Pi):\C^{g}\to\C,\qquad 
\theta(\ptCg \mid \Pi) := 
\sum_{\ptZ \in \Z^{g}} \exp(\tfrac{1}{2} \<\ptZ,\Pi \ptZ \> + \<\ptZ,\ptCg \>).
\]

\begin{theorem}[Riemann's theorem {\cite[Thm.~3.1]{Tata1}}] 
\label{thm:Riemann}
The Abel map
\[
\mu_\basept : \sym^g(\mcurve) \rightarrow \jac(\mcurve)
\]
is surjective and bimeromorphic.  
Its meromorphic inverse is described as follows: there exists a vector $\vecRC \in \C^{g}$, called the \emph{vector of Riemann constants}, depending only on the symplectic basis $(\bm A,\bm B)$ and basepoint $\basept$, such that for each $\vecptjac \in \C^{g}$, exactly one of the following holds:
\begin{enumerate}[label=(\arabic*)]
    \item\label{RT1} The function 
    $
    \ptSig \mapsto \theta(\mu_\basept(\ptSig) - \vecptjac + \vecRC \mid \Pi)
    $
    vanishes identically on $\mcurve$. In this case $[\ptCg]$ lies in a proper closed subset of $\jac(\mcurve)$.
    
    \item\label{RT2} The function 
    $
    \ptSig \mapsto \theta(\mu_\basept(\ptSig) - \vecptjac + \vecRC \mid \Pi)
    $
    has exactly $g$ zeroes $\ptSig_1,\dots,\ptSig_g \in \mcurve$ (counted with multiplicity), satisfying
    \[
    \mu_\basept(\ptSig_1 + \cdots + \ptSig_g) = [\ptCg].
    \]
\end{enumerate}
\end{theorem}

\begin{remark}\label{remark:theta_divisor}
If $g \ge 1$, the Abel map embeds $\mcurve$ into $\jac(\mcurve)$. The \emph{theta divisor} is
\[
\Theta := \{ \ptCg \in \C^{g} \mid \theta(\ptCg \mid \Pi)=0 \}.
\]
Although $\theta(\ptCg \mid \Pi)$ is multi-valued on $\jac(\mcurve)$, the zero set is periodic, i.e.,
\[
\Theta + (2\pi \imunit \Z^{g} \oplus \Pi \Z^{g}) = \Theta,
\]
so $\Theta$ descends to a well-defined subset of $\jac(\mcurve)$. Riemann’s theorem then asserts that for $\vecptjac \in \C^{g}$, either 
$\mu_\basept(\mcurve) \subseteq \Theta + [\vecptjac] - [\vecRC]$, 
or else $\mu_\basept(\mcurve) \cap (\Theta + [\vecptjac] - [\vecRC])$ consists of exactly $g$ points counted with multiplicity. 
\end{remark}

\subsection{Standard divisors} \label{sec:standard_divisors}

For Harnack curves (and more generally, M-curves), Boutillier--Cimasoni--de Tili\`ere~{\cite{BDdT}} prove stronger properties than the general statements of~\cref{sec:jacobian} and~\cref{sec:theta_fn}. These properties will be important for tropicalizing the theta function.

For a Harnack curve $\openspectralcurve$ of genus $g$, the connected components of the real locus 
$\spectralcurve(\R)$ are called \emph{ovals}. 
There are $g+1$ ovals in total. 
There is a unique \emph{outer oval} $B_0^{\mathrm{ov}}$ intersecting the lines at infinity, 
while the remaining ovals $\bm B=(B_1^{\mathrm{ov}},\dots,B_g^{\mathrm{ov}})$, called \emph{compact ovals}, 
are in bijection with the interior lattice points of $N$. The compact ovals are in bijection with the bounded connected components of the amoeba; we give them the orientation induced by the amoeba. In other words, the outer oval is oriented counterclockwise and the compact ovals are oriented clockwise.

Choosing cycles
$
\bm A^{\mathrm{ov}}
=
(A_1^{\mathrm{ov}},\dots,A_g^{\mathrm{ov}})
$
that are anti-invariant under complex conjugation, and
\[
\bm A^{\mathrm{ov}} \wedge \bm B^{\mathrm{ov}} = I,
\]
we obtain a symplectic basis
$
(\bm A^{\mathrm{ov}},\bm B^{\mathrm{ov}})
$
of $H_1(\mcurve,\Z)$; see \cite[Section 2.2.4]{KenyonOkounkovHarnack}. We call such a basis \emph{oval-adapted}. 

Hereafter, we denote objects associated to the symplectic basis $(\bm A^{\mathrm{ov}}, \bm B^{\mathrm{ov}})$ by a superscript $\mathrm{ov}$. We also fix a basepoint $\basept \in B_0^\mathrm{ov}$ throughout.

\begin{proposition} [Boutillier--Cimasoni--de Tili\`ere~{\cite{BDdT}}] \label{prop:omega_and_delta_pos}
		For a Harnack $\openspectralcurve$ with basepoint $\basept \in B_0^{\mathrm{ov}}$ and an oval-adapted symplectic basis $(\bm A^{\mathrm{ov}}, \bm B^{\mathrm{ov}})$, we have:
		\begin{enumerate}[label=(\arabic*)]
			\item\label{BDT1} If $\ptSig \in B_0^{\mathrm{ov}}$, then $\mu_\basept(\ptSig) \in  \R^{g}$.
			\item\label{BDT2} $\Pi^{\mathrm{ov}}_{jk} \in \R$ for all $1 \leq j,k \leq g$.
			\item\label{BDT3} $\vecRC^{\mathrm{ov}} \in \pi \imunit \onebf + \R^{g}$.
		\end{enumerate}
	\end{proposition}
\begin{proof}
	Part~\itemref{BDT1} is~\cite[Lemma~15(2)]{BDdT}, part~\itemref{BDT2} is \cite[Lemma~11(3)]{BDdT}, and part~\itemref{BDT3} is~\cite[Lemma~19]{BDdT}.\footnote{The statements in \cref{prop:omega_and_delta_pos} appear different from those in~\cite{BDdT} due to our choice of normalization $\int_{A_j} \omega_k = 2\pi \imunit \delta_{jk}$, following \cite{Fay}, rather than $\int_{A_j} \omega_k = \delta_{jk}$. Moreover, our convention for $A$- and $B$-cycles is also reversed compared to~\cite{BDdT}.}
\end{proof}

Following Kenyon--Okounkov~\cite{KenyonOkounkovHarnack}, we call a degree-$g$ effective divisor $\divisor$ in $\openspectralcurve(\R)$ a \emph{standard divisor} if it has exactly one point in each compact oval. We denote the set of standard divisors on $\mcurve$ by
\[
\sym^g_{>0} (\mcurve) := \prod_{j=1}^g B_j^{\mathrm{ov}}.
\]
We also denote
\[
\jacpos(\mcurve) := \pi \imunit \onebf + (\R^{g} / \Pi^{\mathrm{ov}} \Z^{g}),
\]
which is a connected component of the real part of the Jacobian $\jac(\mcurve)$ translated by $\pi \imunit \onebf$, where $\onebf:=(1,1,\dots,1) \in \C^g$.

We have the following refinement of \cref{thm:Riemann} for standard divisors.

\begin{proposition} [Boutillier--Cimasoni--de Tili\`ere~{\cite[Lemma~15]{BDdT}}] \label{prop:jacpos}
	For a Harnack curve $\mcurve$ with basepoint $\basept \in B_0^{\mathrm{ov}}$ and symplectic basis $(\bm A^{\mathrm{ov}}, \bm B^{\mathrm{ov}})$ as in~\cref{sec:Harnack}, the Abel map induces a homeomorphism \begin{align*}\mu_{\basept}:  \sym^g_{>0} (\mcurve) &\rightarrow \jacpos(\mcurve)\\
	\divisor &\mapsto \mu_\basept(\divisor).
	\end{align*} 
\end{proposition}
In particular, for standard divisors always correspond to Case~\itemref{RT2} of \cref{thm:Riemann}.

\subsection{Angle maps}

 An \emph{angle map} is a bijection 
\[
\nu:\zzG \to \mcurve \setminus \openspectralcurve
\]
between zig-zag paths and points at infinity such that for all $\alpha \in \zzG$, we have that \[\nu(\alpha) \in \mcurve \cap \partial V^{\mathrm{tor}}_{\Eside(\alpha)}.\] When viewed as a tuple, we write 
$
\bm \nu = (\nu(\alpha))_{\alpha \in \zzG}.
$ For brevity, we often write $\alpha$ for $\nu(\alpha)$.

\subsection{The spectral transform}

A \emph{spectral datum associated to $\Gtor$} is a triple $(\mcurve,\divisor,\nu)$ where:
\begin{enumerate}
  \item $\mcurve$ is a Harnack curve with Newton polygon $N$. \footnote{The space of Harnack curves with Newton polygon $N$ has four
components related by $(z,w)\mapsto(\pm z,\pm w)$. The choice of
Kasteleyn signs fixes the relevant component, which we use throughout.}
  \item $\divisor \in \sym^g_{>0} (\mcurve)$ is a standard divisor.
  \item $\nu$ is an angle map.
\end{enumerate}

Let $\bw$ be a fixed reference white vertex of $\Gtor$. Consider the $\bw$-column of the adjugate matrix 
$\bQ(\ptC)$ of $\bK(\ptC)$. By~\cite{KenyonOkounkovHarnack}, the points where the entries of this column simultaneously vanish, i.e.,
\[
\bQ(\ptC)_{\bb\bw} = 0, \quad \text{ for all }\bb \in \blackvertices(\Gtor),
\]
defines a standard divisor 
\[\divisor = \sum_{j=1}^g \divptgeom_j.\] 

Following \cite[Section~7.2]{GK} and~\cite[Section~3.1]{KenyonOkounkovHarnack}, define the \emph{Casimirs} \[\veccasimirs = (\casimir_\alpha)_{ \alpha \in \zzG }, \qquad
\casimir_\alpha := [\wt](\alpha). 
\]
Since $\sum_{\alpha \in \zzG} [\alpha]=0$, we have
$
\prod_{\alpha \in \zzG} \casimir_\alpha = 1.
$ 
Recall that the tentacles of $\amoebaSig$ are in bijection with the points at infinity of $\openspectralcurve$. For each $\alpha \in \zzG$, there is a point at infinity $\nu(\alpha) \in \mcurve \cap \partial V^{\mathrm{tor}}_{\Eside(\alpha)}$ determined by the property that the tentacle of $\amoebaSig$ corresponding to a zig-zag path $\alpha$ with $[\alpha]$ is asymptotic to the line \begin{equation}\label{eq:tentacles}
\< [\alpha], \ptR \> + \log |\casimir_\alpha| = 0.
\end{equation}
This determines an angle map
\[
\nu:\zzG \to \mcurve \cap \bigcup_{\text{sides $\Eside$ of $N$}} \partial V^{\mathrm{tor}}_\Eside.
\]

The map
\begin{align*}
 \lambda: \{\text{positive edge weights on $\Gtor$}\}/\text{gauge} & \to \{\text{spectral data associated to $\Gtor$}\}\\
[\wt]&\mapsto(\spectralcurve,\divisor,\nu)
\end{align*}
is called the \emph{spectral transform}.

\begin{theorem}\label{thm:spectral_transform_birational}
  The spectral transform is a bijection.
\end{theorem}

\cref{thm:spectral_transform_birational} was originally proved for positive spectral data on the hexagonal lattice by Kenyon--Okounkov~\cite{KenyonOkounkovHarnack} through an indirect argument. Explicit inverses were later constructed for complex spectral data by Fock~\cite{Fock} and George--Goncharov--Kenyon~\cite{GGK}. The positive counterpart of Fock’s inverse map stated above for general graphs is due to Boutillier--Cimasoni--de Tili\`ere~\cite{BDdT}.

\begin{example}
\label{ex:running_graph_spectral_transform_genus_one}

\begin{figure}
\centering

\begin{tikzpicture}

\node[
    anchor=south west,
    inner sep=0
] (amoeba) at (0,0) {
    \includegraphics[width=.55\textwidth]{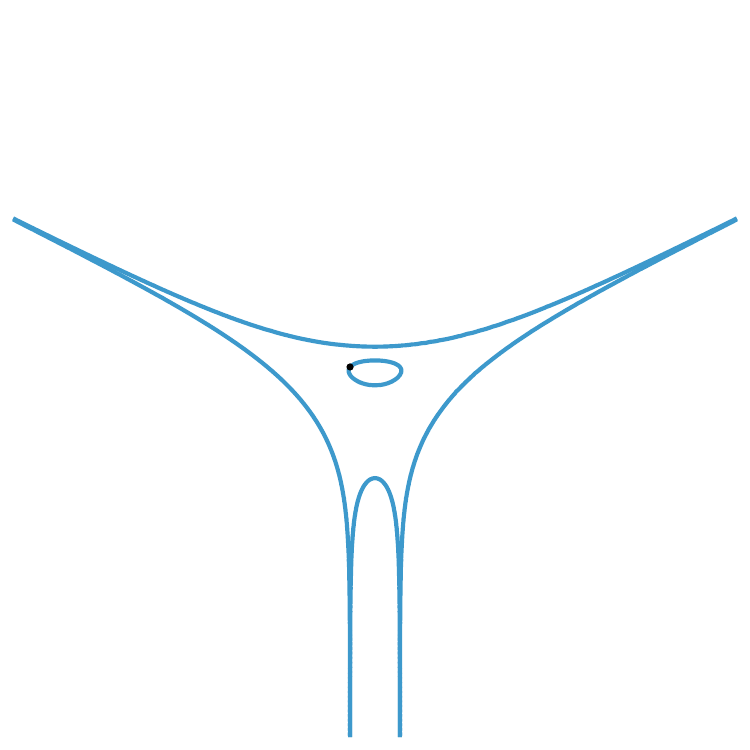}
};

\begin{scope}[
    x={(amoeba.south east)},
    y={(amoeba.north west)}
]


\node[
    green!60!black,
    anchor=east
] at (-0.015,0.715)
    {$\beta$};

\node[
    blue,
    anchor=west
] at (1.015,0.715)
    {$\gamma$};

%

\node[
    red,
    anchor=north east
] at (0.475,-0.015)
    {$\alpha_1$};

\node[
    red,
    anchor=north west
] at (0.535,-0.015)
    {$\alpha_2$};


\node[
    anchor=south west,
    inner sep=1pt
] at (0.405,0.555)
    {$\Log(D)$};

\end{scope}

\end{tikzpicture}

\caption{
The real locus of the amoeba of the spectral curve, together with
the standard divisor and angle map.
}
\label{fig:running-example-amoeba}

\end{figure}
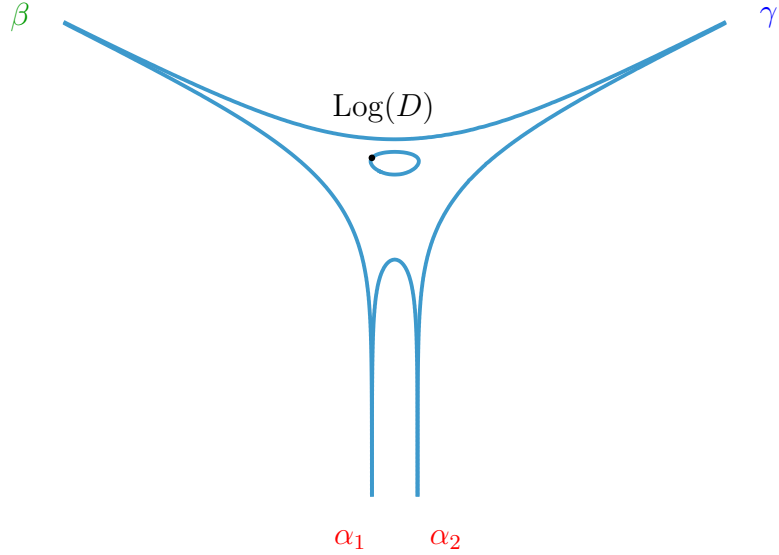

We continue with the running example and compute its spectral transform. Choose $\bw_1$ as the reference white vertex. The
$\bw_1$-column of $\bQ$ is
\[
\bQ_{\bullet,\bw_1}(z,w)
=
\begin{pmatrix}
\dfrac{
X_1\rho_w w-1-X_1X_2X_3
}{
X_1\rho_w w
}
\\[4mm]
\dfrac{
\rho_z z+X_1X_2X_3
}{
\rho_z\rho_w zw
}
\\[4mm]
\dfrac{
X_2X_3\rho_z z-X_2X_3-\rho_z\rho_w zw
}{
\rho_z\rho_w^2zw^2
}
\\[4mm]
-\dfrac{
X_2X_3
\left(
X_1\rho_w w+\rho_z z-1
\right)
}{
\rho_z\rho_w^2zw^2
}
\end{pmatrix},
\]
from which we get the standard divisor
\[
D=\left(
-\frac{X_1X_2X_3}{\rho_z},
\frac{1+X_1X_2X_3}{X_1\rho_w}\right).
\]

We next compute the angle map.  Let
$
\alpha_1,\alpha_2,\beta,\gamma\in\zzG
$
denote the four zig-zag paths, labelled so that
\[
[\alpha_1]=[\alpha_2]=(-1,0),
\qquad
[\beta]=(1,2),
\qquad
[\gamma]=(1,-2),
\]
and 
\[
\casimir_{\alpha_1}
=
{\rho_z},
\qquad
\casimir_{\alpha_2}
=
\frac{\rho_z}{X_1X_2},
\qquad
\casimir_{\beta}
=
\frac{X_2X_3}{\rho_z\rho_w^2},
\qquad
\casimir_{\gamma}
=
\frac{X_1\rho_w^2}{X_3\rho_z}.
\]
The angle map $\nu$ is only nontrivial for the horizontal side of $N$. The two tentacles are along $\log|w| \to -\infty$ with
\[
x + \log|C_{\alpha_1}|=0, \qquad x + \log|C_{\alpha_2}|=0.
\]
\cref{fig:running-example-amoeba} illustrates the spectral transform for
\[
X_1 = 2,\quad X_2 = 2,\quad X_3 = 1,\quad \rho_z = 2,\quad \rho_w = 2.
\]
\end{example}

\subsection{Fock's inverse spectral transform}
\label{sec:fock_weights}

Via $\nu$, we identify $\Z^{\zzG}$ with $\Z^{\mcurve \setminus \openspectralcurve}$, i.e., the group of divisors supported on $\mcurve \setminus \openspectralcurve \subset B_0^{\mathrm{ov}}$. The \emph{discrete Abel map} \[\dam : \blackvertices(\G) \sqcup \whitevertices(\G) \sqcup \faces(\G) \ra \Z^{ \zzG }\] assigns to each vertex and face of $\G$ a divisor supported on $\zzG$ according to the following rules:
\begin{enumerate}
	\item Normalize by setting $\dam(\tilde \bw)=0$ for a reference white vertex $\tilde \bw$.
	\item For a face $\tilde \bface$ incident to a black vertex $\tilde \bb$, set 
	\[
	\dam(\tilde\bb)-\dam(\tilde\bface)=\alpha,
	\]
	where $\alpha$ is the zig-zag path in $\Gtor$ whose lift to $\G$ separates $\tilde\bface$ and $\tilde\bb$.
	\item For a face $\tilde\bface$ incident to a white vertex $\tilde\bw$, set 
	\[
	\dam(\tilde\bw)-\dam(\tilde\bface)=-\alpha,
	\]
	where $\alpha$ is the zig-zag path in $\Gtor$ whose lift to $\G$ separates $ \tilde\bface$ and $ \tilde\bw$.
\end{enumerate}

Recall that we fix a basepoint $\basept \in B_0^{\mathrm{ov}}$. Let $\widetilde \mcurve$ denote the universal cover of $\mcurve$. We fix a lift $\widetilde \basept$ of $\basept$ and for each zig-zag path $\alpha \in \zzG$, we fix a lift $\widetilde \alpha \in \widetilde \mcurve$. Since $\widetilde \mcurve$ is simply connected, any two paths in
$\widetilde \mcurve$ joining $\widetilde\ptSig$ to $\widetilde\ptSig'$ are homotopic relative to their endpoints. For a $1$-form $\omega$, we
use the convention that
\[
  \int_{\ptSig}^{\ptSig'} \omega
\]
denotes the integral along any path in $\widetilde \mcurve$ joining the chosen lift of $\widetilde\ptSig$ to the chosen lift of $\widetilde\ptSig'$. For a holomorphic $1$-form, this integral is independent of the chosen path. For the meromorphic differentials of the third kind below, we also assume that the path avoids the poles. Although the integrals may depend on the path, we only consider their exponentials which are again path independent. Unless stated otherwise, all integrals in this section are understood in this sense.

A \emph{wedge} of $\Gtor$ is a $1$-chain
\[
\mathfrak w=\be_+-\be_-,
\]
where
\(
\be_-=\bb\bw_-,
\) and \(
\be_+=\bb\bw_+
\)
are two edges incident to a common black vertex $\bb$. Let
$\alpha,\beta_-,\beta_+\in\zzG$ be the zig-zag paths determined by
\[
\be_-\in\alpha\cap\beta_-,
\qquad
\be_+\in\alpha\cap\beta_+.
\]
Let $\bface_-$, respectively $\bface_+$, be the face separated from $\bb$
by $\beta_-$, respectively $\beta_+$; see \cref{fig:wedge}.

Let $(\spectralcurve,\divisor,\nu)$ be a spectral datum associated to $\Gtor$. We assign to a $\mathfrak w$ the \emph{wedge weight}
\begin{equation}
\label{eq:fock:wedge}
\wt(\mathfrak w)
:=
-\frac{
  \theta(\mu_\basept^{\mathrm{ov}}(\dam(\bface_-))-\mu^{\mathrm{ov}}_\basept(\divisor)+\vecRC^{\mathrm{ov}}\mid\Pi^{\mathrm{ov}})
}{
  \theta(\mu_\basept^{\mathrm{ov}}(\dam(\bface_+))-\mu^{\mathrm{ov}}_\basept(\divisor)+\vecRC^{\mathrm{ov}}\mid\Pi^{\mathrm{ov}})
}
\exp\left(
  \int_{\basept}^{\alpha}\omega_{\beta_+-\beta_-}^{\mathrm{ov}}
\right).
\end{equation}
Here, $\omega_{\beta_+-\beta_-}^{\mathrm{ov}}$ is the $A$-normalized differential of the third kind with residue divisor $\beta_+-\beta_-$, and the Abel maps and $
  \int_{\basept}^{\alpha}\omega_{\beta_+-\beta_-}^{\mathrm{ov}}$ use our convention for integrals above. 

Any cycle $\gamma$ in $\Gtor$ can be decomposed as a sum of wedges,
\[
\gamma=\mathfrak w_1+\cdots+\mathfrak w_r.
\] 
We define its cycle weight by 
\[
[\wt](\gamma):= [\kappa](\gamma)
\prod_{j=1}^r\wt(\mathfrak w_j),
\]
where $[\kappa]$ is the gauge equivalence class of the Kasteleyn signs. In particular, the face weights are given by
\[
X_{\bface}
:=
(-1)^{\frac{|\partial\bface|}{2}+1}
\prod_{j=1}^r\wt(\mathfrak w_j).
\]
The corresponding gauge-equivalence class of
edge weights on $G$ are called \emph{Fock weights}.

\begin{theorem}[Fock~\cite{Fock}]\label{thm:fock_inverse}
  The assignment of Fock weights to spectral data inverts the spectral transform.
\end{theorem}

\section{Tropical spectral curves}\label{sec:trp_spec_curve}

\subsection{Puiseux series}

Throughout, we use the max convention for tropicalization. The \emph{field of formal Puiseux series} is
\[
{\mathbb C\{\!\{ t^{-1}\}\!\}}=\bigcup_{n\geq1}\C(\!(t^{-1/n})\!).
\]
An element $c(t) \in {\mathbb C\{\!\{ t^{-1}\}\!\}}$ of this field has the form
\[
c(t)
=
c_1t^{a_1}+c_2t^{a_2}+\cdots,
\]
where $c_i\neq0$ and
$
\cdots<a_2<a_1\in\Q
$
have a common denominator. The coefficient $c_1$ is called the
\emph{leading coefficient} and is denoted by $\lc(c(t))$. The field
${\mathbb C\{\!\{ t^{-1}\}\!\}}$ carries the valuation
\[
\begin{aligned}
\val:{\mathbb C\{\!\{ t^{-1}\}\!\}}&\ra\Q\cup\{-\infty\},\\
c(t)&\mapsto a_1,
\end{aligned}
\]
with $\val(0):=-\infty$.

Let $\Puis \subset {\mathbb C\{\!\{ t^{-1}\}\!\}}$ be the subfield of \emph{convergent Puiseux series}, i.e., those Puiseux series that have a positive radius of convergence as $t \to \infty$
The semifield $\Puispos$ of \emph{positive convergent Puiseux series} consists of those
$c(t)\in\Puis$ that are convergent, have real coefficients, and whose leading coefficient is positive.

Throughout this paper, we will only consider convergent Puiseux series. Therefore, we will simply refer to them as Puiseux series and omit the qualifier ``convergent''.

\subsection{Tropicalization}

Let $[\wt](t)$ be $\Puispos$-valued edge weights modulo gauge
equivalence, and let
\[
P(\ptC;t)
=
\sum_{\pt \in\NZ}
c_{\pt}(t)\ptC^\pt
\]
be the corresponding family of characteristic polynomials. Set
$
c_{\pt}^\trop:=\val(c_{\pt}(t)).
$
The \emph{tropical characteristic polynomial} is
\begin{equation}
\label{eq:trop_char_poly_formula}
P^\trop(\ptR)
:=
\max_{\pt\in\NZ}
\left(
\langle \pt,\ptR \rangle+c_{\pt}^\trop
\right),
\end{equation}
and the \emph{embedded tropical spectral curve} is
\begin{equation} \label{eq:trop_spc_emb}
\tropmcurve_{\mathrm{emb}}
:=
\left\{
\ptR \in\R^2
\;\middle|\;
P^\trop(\ptR)
\text{ achieves its maximum at least twice}
\right\}.
\end{equation}
Let $\mcurve(t)$ be the corresonding family of spectral curves. By the fundamental theorem of tropical geometry~\cite[Theorem~4.6]{MikhalkinAmoebas},
\begin{equation}
\tropmcurve_{\mathrm{emb}}=\lim_{t\to\infty}\frac{1}{\log t}\Log(\mcurve(t)), 
\label{eq:def_tropicalization_of_hypersurface}
\end{equation} 
where the limit is taken in the Hausdorff metric on subsets of $\R^2$.

The function
\[
c^\trop:\NZ\ra\R
\]
induces a regular subdivision of $N$ in the sense of
\cite{GKZ}. Namely, consider the upper hull
\begin{equation}
\label{eq:graph}
\Conv
\{
(\pt;z)\in\NZ\times\R
|
z\leq c^\trop_{\pt}
\}.
\end{equation}
The cells of the subdivision are the projections to $N$ of the upper faces of
\eqref{eq:graph}. The subdivision of $\R^2$ into the domains of linearity of
$P^\trop$ and the regular subdivision are dual polyhedral
complexes; see
\cite[Proposition~3.1.6]{MaclaganSturmfels} and \cref{fig:embedded-abstract-dual-subdivision} for an example.

\begin{example}
\label{ex:running_graph_tropical_spectral_curve}

\begin{figure}
\centering

\begin{subfigure}[c]{0.58\textwidth}
\centering

\begin{tikzpicture}[
    scale=1.05,
    line cap=round,
    line join=round,
    vertex/.style={
        circle,
        fill=black,
        inner sep=0pt,
        minimum size=5pt
    }
]


\draw[curveblue,line width=1.2pt]
    (-2,1)
    -- node[black,midway,above=3pt] {$4$}
    (2,1);

\draw[curveblue,line width=1.2pt]
    (2,1)
    -- node[black,midway,right=3pt] {$1$}
    (1,0);

\draw[curveblue,line width=1.2pt]
    (1,0)
    -- node[black,midway,below=3pt] {$2$}
    (-1,0);

\draw[curveblue,line width=1.2pt]
    (-1,0)
    -- node[black,midway,left=3pt] {$1$}
    (-2,1);


\draw[curveblue,line width=1.2pt]
    (-2,1) -- (-4,2);

\draw[curveblue,line width=1.2pt]
    (2,1) -- (4,2);

\draw[curveblue,line width=1.2pt]
    (-1,0) -- (-1,-2);

\draw[curveblue,line width=1.2pt]
    (1,0) -- (1,-2);


\node[vertex] at (-2,1) {};
\node[vertex] at (2,1) {};
\node[vertex] at (-1,0) {};
\node[vertex] at (1,0) {};


\node[below left] at (-2,1)
    {$(-2,1)$};

\node[below right] at (2,1)
    {$(2,1)$};

\node[below left=3pt] at (-1,0)
    {$(-1,0)$};

\node[below right=3pt] at (1,0)
    {$(1,0)$};

\end{tikzpicture}

\caption{}
\label{fig:running-example-tropical-spectral-curve-embedded}
\end{subfigure}%
\hfill%
\begin{subfigure}[c]{0.30\textwidth}
\centering

\begin{tikzpicture}[
    scale=1.35,
    line cap=round,
    line join=round
]


\draw[black,line width=1pt]
    (0,0)
    -- (1,-2)
    -- (-1,-2)
    -- cycle;


\draw[black,line width=1pt]
    (0,-1) -- (0,0);

\draw[black,line width=1pt]
    (0,-1) -- (1,-2);

\draw[black,line width=1pt]
    (0,-1) -- (0,-2);

\draw[black,line width=1pt]
    (0,-1) -- (-1,-2);


\node[bvert] at (0,0) {};
\node[bvert] at (0,-1) {};
\node[bvert] at (0,-2) {};
\node[bvert] at (-1,-2) {};
\node[bvert] at (1,-2) {};

\end{tikzpicture}

\caption{}
\label{fig:running-example-newton-subdivision}
\end{subfigure}

\caption{
The embedded tropical spectral curve associated to
$
P^\trop(x,y)
=
\max\{0,1-y,1-2y,-x-2y,x-2y\}
$~(A) and the corresponding dual subdivision of its Newton polygon~(B).
The bounded edges of the tropical curve are labeled by their lattice
lengths.
}
\label{fig:running-example-tropical-spectral-curve}

\end{figure}
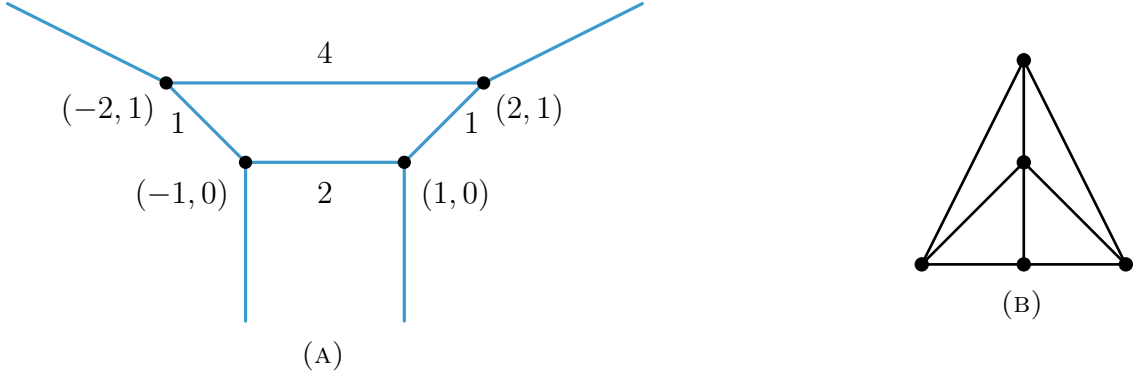

We return to the running example of
\cref{ex:running_graph_zigzags_genus_one} and choose
\[
X_1(t)=t,\qquad
X_2(t)=t,\qquad
X_3(t)=1,\qquad
\rho_z(t)=t,\qquad
\rho_w(t)=1.
\]
The limit \eqref{eq:def_tropicalization_of_hypersurface} is illustrated in \cref{fig:running-example-kapranov}. The tropical characteristic polynomial is
\[
P^\trop(x,y)
=
\max
\{
0,
1-y,
1-2y,
-x-2y,
x-2y
\}.
\]
Its tropical curve is shown in \cref{fig:running-example-tropical-spectral-curve}. 
\end{example}

\subsection{Abstract tropical curves}\label{sec:abst_trop_curve}

An \emph{abstract tropical curve} is a triple
$(\tropmcurve,\ell,\bm g)$, where:
\begin{itemize}
\item $\tropmcurve$ is a finite connected graph.
\item $\ell:\edges(\tropmcurve)\to\Rpos$ is a \emph{length} function.
\item $\bm g=(g_v)_{v\in\vertices(\tropmcurve)}, g_v \in \Z_{\geq 0}$, is a genus assigned to each vertex.
\end{itemize}

Set
\[
g^{\mathrm{comp}}
:=
\sum_{v\in\vertices(\tropmcurve)}g_v,
\qquad
g^\trop
:=
\operatorname{rank}H_1(\tropmcurve,\Z),
\qquad
g
:=
g^{\mathrm{comp}}+g^\trop.
\]
We call $g$ the \emph{genus} of the abstract tropical curve.

For a line segment $I=[a,b]\subset\R^2$ parallel to a primitive
vector $u\in\Z^2$, define its \emph{lattice length} $|I|_{\Z}$ by
$
b-a=|I|_{\Z}u.
$ We now associate an abstract tropical curve to
\eqref{eq:trop_spc_emb}. For each bounded edge $\bar e$ of
$\tropmcurve_{\mathrm{emb}}$, let $\bar e^*$ be its dual edge and set
\[
m_{\bar e}
:=
|\bar e^*|_{\Z}.
\]
The underlying graph $\tropmcurve$ is obtained from
$\tropmcurve_{\mathrm{emb}}$ by contracting all unbounded edges and
replacing each bounded edge $\bar e$ by $m_{\bar e}$ distinct edges
$
e_1,\dots,e_{m_{\bar e}}
$
with the same endpoints. Thus, there is a natural map
\[
\pi:
\tropmcurve
\rightarrow
\tropmcurve_{\mathrm{emb}}
\]
which is the identity on vertices and maps each edge
$e_i\in\pi^{-1}(\bar e)$ linearly onto $\bar e$. In particular,
$
\vertices(\tropmcurve)
=
\vertices(\tropmcurve_{\mathrm{emb}}).
$ For each edge $e_i$ lying over $\bar e$, define
$
\ell(e_i)
:=
|\bar e|_{\Z}.
$
For $v\in\vertices(\tropmcurve)$, let $v^*$ be the dual polygon and set
$
g_v
:=
\left|
(v^*)^\circ\cap\Z^2
\right|.
$
This defines the abstract tropical curve
$
(\tropmcurve,\ell,\bm g).
$

In our setting, $\tropmcurve$ comes with a natural planar embedding. For each bounded embedded edge
$\bar e$, draw the edges in $\pi^{-1}(\bar e)$ as parallel edges
in a neighborhood of $\bar e$. Their ordering will be identified in
\cref{sec:degenerate_curve} with the cyclic order of the corresponding
nodes in the special fiber of the family $(\mcurve(t))_{t \gg 1}$.

With respect to this planar embedding, $\tropmcurve$ has
$g^\trop+1$ faces. Let
$
f_0\in F(\tropmcurve)
$
denote the outer face and set
\[
F^\circ(\tropmcurve)
:=
F(\tropmcurve)\setminus\{f_0\}.
\]
We orient the boundary of $f_0$ counterclockwise and the boundaries of
all other faces clockwise, and define
$
B_f^{\mathrm{ov},\trop}
:=
\partial f.
$
The cycles
$
\bm B^{\mathrm{ov},\trop}
=
(
B_f^{\mathrm{ov},\trop}
)_{f\in F^\circ(\tropmcurve)}
$
form a basis of $H_1(\tropmcurve,\Z)$, and
$
\sum_{f\in F(\tropmcurve)}
B_f^{\mathrm{ov},\trop}
=
0.
$

\begin{example}
\label{ex:non-smooth-tropical-curve}

Consider the embedded tropical curve shown in
\cref{fig:embedded-abstract-dual-subdivision}(A). The bounded edge joining $(0,0)$ to $(3,0)$ has multiplicity two. Passing to the corresponding abstract tropical curve, we replace this
multiplicity-two edge by a bigon. The vertices all have genus $0$ and the lengths of the bounded edges are indicated in
\cref{fig:embedded-abstract-dual-subdivision}(B).

\begin{figure}
\centering


\begin{subfigure}[c]{0.36\textwidth}
\centering

\begin{tikzpicture}[
    scale=.32,
    line cap=round,
    line join=round,
    tropical/.style={
        draw=curveblue,
        line width=1.2pt
    },
    vertex/.style={
        circle,
        fill=black,
        inner sep=0pt,
        minimum size=5pt
    }
]

\draw[tropical] (0,.16) -- (3,.16);
\draw[tropical] (0,-.16) -- (3,-.16);

\draw[tropical] (3,0) -- (3,5);
\draw[tropical] (3,5) -- (-5,5);
\draw[tropical] (-5,5) -- (0,0);

\draw[tropical] (3,0) -- (7,-2);
\draw[tropical] (3,5) -- (7,9);
\draw[tropical] (-5,5) -- (-9,7);
\draw[tropical] (0,0) -- (-3,-3);

\node[vertex] at (0,0) {};
\node[vertex] at (3,0) {};
\node[vertex] at (3,5) {};
\node[vertex] at (-5,5) {};

\node[left] at (0,0) {$(0,0)$};
\node[right] at (3,0) {$(3,0)$};
\node[right] at (3,5) {$(3,5)$};
\node[below left] at (-5,5) {$(-5,5)$};

\end{tikzpicture}

\caption{}
\label{fig:embedded-tropical-curve}
\end{subfigure}%
\hfill%
\begin{subfigure}[c]{0.36\textwidth}
\centering

\begin{tikzpicture}[
    scale=.32,
    line cap=round,
    line join=round,
    tropical/.style={
        draw=curveblue,
        line width=1.2pt
    },
    vertex/.style={
        circle,
        fill=black,
        inner sep=0pt,
        minimum size=5pt
    }
]


\draw[tropical]
    (0,0)
    to[out=30,in=150]
    node[midway,above=4pt,text=black,font=\scriptsize]
    {$\ell(e_4)=3$}
    (3,0);

\draw[tropical]
    (0,0)
    to[out=-30,in=210]
    node[midway,below=4pt,text=black,font=\scriptsize]
    {$\ell(e_5)=3$}
    (3,0);


\draw[tropical]
    (3,0)
    -- node[midway,right=4pt,text=black,font=\scriptsize]
    {$\ell(e_1)=5$}
    (3,5);

\draw[tropical]
    (3,5)
    -- node[midway,above=4pt,text=black,font=\scriptsize]
    {$\ell(e_2)=8$}
    (-5,5);

\draw[tropical]
    (-5,5)
    -- node[midway,below left=2pt,text=black,font=\scriptsize]
    {$\ell(e_3)=5$}
    (0,0);


\node[vertex] at (0,0) {};
\node[vertex] at (3,0) {};
\node[vertex] at (3,5) {};
\node[vertex] at (-5,5) {};

\node[left] at (0,0) {$v_2$};
\node[right] at (3,0) {$v_1$};
\node[right] at (3,5) {$v_4$};
\node[left] at (-5,5) {$v_3$};

\end{tikzpicture}

\caption{}
\label{fig:abstract-tropical-curve}
\end{subfigure}%
\hfill%
\begin{subfigure}[c]{0.20\textwidth}
\centering

\begin{tikzpicture}[
    scale=.82,
    line cap=round,
    line join=round
]


\draw[black,line width=1pt]
    (0,-2)
    -- (1,0)
    -- (0,1)
    -- (-1,-1)
    -- cycle;

\draw[black,line width=1pt] (0,0) -- (0,-2);
\draw[black,line width=1pt] (0,0) -- (1,0);
\draw[black,line width=1pt] (0,0) -- (0,1);
\draw[black,line width=1pt] (0,0) -- (-1,-1);

\node[bvert] at (0,-2) {};
\node[bvert] at (1,0) {};
\node[bvert] at (0,1) {};
\node[bvert] at (-1,-1) {};
\node[bvert] at (0,-1) {};
\node[bvert] at (0,0) {};

\end{tikzpicture}

\caption{}
\label{fig:tropical-dual-subdivision}
\end{subfigure}

\caption{
An embedded tropical curve with an edge of multiplicity $2$~(A), the corresponding abstract tropical curve obtained by
replacing the multiplicity-two edge with a bigon and contracting unbounded edges~(B), and the
dual subdivision of the Newton polygon~(C).
}
\label{fig:embedded-abstract-dual-subdivision}
\end{figure}
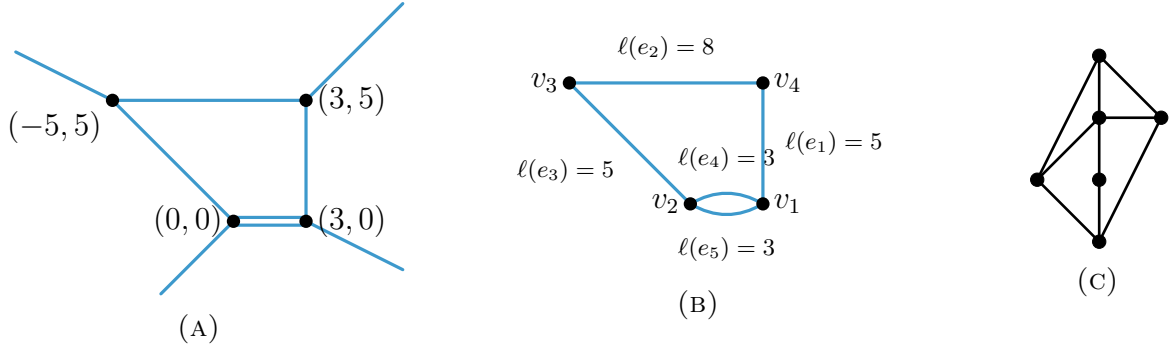

\end{example}

\subsection{Tropical differentials}\label{sec:trop_diff}
Let $(\tropmcurve,\ell,\bm g)$ be a abstract tropical curve of genus $g$. Let $H(\tropmcurve)$ denote the set of half-edges of
$\tropmcurve$. For each $h\in H(\tropmcurve)$, let $v(h)$ be its
incident vertex, let $e(h)$ be its underlying unoriented edge, and let
$\bar h$ be the opposite half-edge. Thus, $
e(h)=e(\bar h).$ For each vertex $v\in\vertices(\tropmcurve)$, set
\[
H_v
:=
\{
h\in H(\tropmcurve)
\mid
v(h)=v
\}.
\]

A \emph{meromorphic tropical $1$-form} on the abstract tropical curve
$(\tropmcurve,\ell, \bm g)$ is an antisymmetric function
\[
\omega:H(\tropmcurve)\rightarrow\Z,
\qquad
\omega(\bar h)=-\omega(h).
\]
The \emph{residue} of $\omega$ at a vertex
$v\in\vertices(\tropmcurve)$ is
\[
\res_v(\omega)
:=
\sum_{h\in H_v}\omega(h).
\]
A vertex $v$ is called a \emph{pole} of $\omega$ if
$
\res_v(\omega)\neq 0.
$
We say that $\omega$ is \emph{holomorphic} if
$
\res_v(\omega)=0
$
for every $v\in\vertices(\tropmcurve)$.

Let $I$ be an oriented edge segment contained in the edge $e(h)$ and
oriented in the direction of the half-edge $h$. We define the
\emph{tropical integral}
\[
\int_I\omega
:=
\omega(h)\ell(I),
\]
where $\ell(I)$ denotes the length of $I$ in the metric induced by
$\ell$. The tropical integral extends to $1$-chains by linearity.

\subsection{Tropical Abel map and Jacobian}\label{sec:tropical_jacobian}
Let $(\tropmcurve,\ell,\bm g)$ be a abstract tropical curve of genus $g$. Fix a basis $\bm C^\trop = (C_1^\trop,\dots,C_{g^\trop})$ of $H_1(\tropmcurve,\Z)$. For any $1 \leq j \leq g^\trop$, let $\omega_j^\trop$ to be the holomorphic tropical $1$-form
such that, for every $h\in H(\tropmcurve)$,
$\omega_j^{\trop}(h)$ is the coefficient of the edge $e(h)$, oriented in the direction of $h$, in the cycle $C_j^\trop$. In other words,
$\omega_j^{\trop}$ is the unit flow around $C_j^\trop$.


The \emph{tropical period matrix} is the positive-definite $g^\trop\times g^\trop$ matrix
\[
\Pi^{\trop}
:=
\int_{\bm C^{\trop}}\bm\omega^{\trop}.
\]

The \emph{tropical Jacobian} is the real
$g^\trop$-dimensional torus
\[
\jac(\tropmcurve)
:=
\R^{g^\trop}/\Pi^{\trop}\Z^{g^\trop}.
\]
For $\vecptjac^\trop\in\R^{g^\trop}$, we denote its image in $\jac(\tropmcurve)$ by
\(
[\vecptjac^\trop].
\)

The \emph{tropical Abel map} 
\[
\mu^{\trop}_\baseptT : \tropmcurve \rightarrow \jac(\tropmcurve)
\]
with basepoint $\baseptT \in \tropmcurve$ is defined by
\[
\mu_\baseptT^{\trop}(\ptSigT) := \int_\baseptT^{\ptSigT} \bm \omega^{\trop} \ \ \text{mod }  \Pi^{\trop} \Z^{g^\trop}.
\]
The tropical Abel map extends linearly to divisors.

A natural choice of basis is $\bm C^\trop = \bm B^{\mathrm{ov},\trop}$, in which case we denote the corresponding objects by an ov in the superscript. However, it will be more convenient to work with a different basis below.

\begin{example}\label{eg:tropical_period_matrix}

Consider the abstract tropical curve in
\cref{fig:embedded-abstract-dual-subdivision}(B).
It has two bounded faces. Let
$
\bm B^{\mathrm{ov},\trop}
=
(
B_1^{\mathrm{ov},\trop},
B_2^{\mathrm{ov},\trop}
),
$
where $B_1^{\mathrm{ov},\trop}$ is the clockwise boundary of the bigon
and $B_2^{\mathrm{ov},\trop}$ is the clockwise boundary of the other
bounded face. Then,
\[
\int_{B_1^{\mathrm{ov},\trop}}
\omega_1^{\mathrm{ov},\trop}
=
3+3=6,
\qquad
\int_{B_2^{\mathrm{ov},\trop}}
\omega_2^{\mathrm{ov},\trop}
=
3+5+8+5=21,
\]
whereas
\[
\int_{B_1^{\mathrm{ov},\trop}}
\omega_2^{\mathrm{ov},\trop}
=
\int_{B_2^{\mathrm{ov},\trop}}
\omega_1^{\mathrm{ov},\trop}
=
-3.
\]
Therefore, the tropical period matrix in the oval-adapted basis is
\[
\Pi^{\mathrm{ov},\trop}
=
\int_{\bm B^{\mathrm{ov},\trop}}
\bm\omega^{\mathrm{ov},\trop}
=
\begin{pmatrix}
6 & -3\\
-3 & 21
\end{pmatrix}.
\]

\end{example}

\section{Degeneration of the spectral curve}
\label{sec:plumbing}

Let $\mathcal M_{g,n}$ denote the moduli space of smooth genus-$g$ curves with $n$ marked points, i.e., pairs $(\mcurve,\bm \nu = (\nu_1,\dots,\nu_n))$ where $\mcurve$ is a smooth curve of genus $g$ and $\nu_1,\dots,\nu_n$ are distinct points of $\mcurve$. For brevity, when $g$ and $n$ are understood from context, we refer to $(\mcurve,\bm \nu)$ as a \emph{pointed curve}.

Its Deligne--Mumford compactification $\overline{\mathcal M}_{g,n}$ is obtained by allowing
stable curves. A \emph{stable pointed curve} is a connected compact curve $\mcurve$
whose only singularities are nodes, together with $n$ distinct marked
smooth points
$
\bm \nu,
$
satisfying a stability condition on each irreducible component.
More precisely, let $\mcurve_v$ be a component of the normalization,
let $g_v$ be its genus, and let $n_v$ be the number of special points
on $\mcurve_v$, consisting of the preimages of the nodes and the marked
points. The stability condition is
\[
2g_v-2+n_v>0
\]
for every component $\mcurve_v$. In other words, every genus-$0$ component has at least
three special points, and every genus-$1$ component has at least one special point.

The relevance of the Deligne--Mumford compactification for us is that
the family of pointed spectral curves
$
(\mcurve(t),\bm\nu(t))
$
has a limit in $\overline{\mathcal M}_{g,n}$. We will show that this
limit is a stable pointed curve
$
(\mcurve(\infty),\bm\nu(\infty)),
$
whose irreducible components are naturally indexed by the vertices of
the tropical spectral curve and whose nodes correspond to its bounded
edges. We then describe $\mcurve(t)$ near this boundary point of
$\overline{\mathcal M}_{g,n}$ using plumbing coordinates. This
description will allow us to study the asymptotics of the classical
objects on $\mcurve(t)$ that enter Fock's formula using results of \cite{HuNorton}.

\subsection{Toric degeneration}
\label{sec:degenerate_curve}

Let $[\wt](t)$ be a $\Puispos$-valued edge weight modulo gauge
equivalence satisfying a genericity condition that we will specify
later, and let
\[
(\spectralcurve(t))_{t\gg1}\subset\toricsurface
\]
be the corresponding family of spectral curves with embedded tropical
spectral curve $\tropmcurve_{\mathrm{emb}}$. Then
$\tropmcurve_{\mathrm{emb}}$ determines a degeneration of
$\toricsurface$ inside which the family $\spectralcurve(t)$ itself
degenerates to a nodal curve $\degeneratecurve$. In this section, we
describe this degeneration; see for example
\cite[Section~6.6]{MaclaganSturmfels} for details.

By replacing $t$ with $t^{\frac1M}$ for some large $M\in\Z$, we can
assume that $\hei_\pt\in\Z$ for all $\pt\in\NZo$. The tropical curve
$\tropmcurve_{\mathrm{emb}}$ determines a \emph{toric degeneration}
\[
p:\toricsurfacedegeneration\ra\C,
\]
where
\[
\toricsurfacedegeneration
:=
\overline{
\left\{
([t^{\hei_\pt}\ptC^\pt],t^{-1})_{\pt\in\NZ}
\mid
\ptC\in(\C^\times)^2
\right\}}
\subset
\C\mathbb P^{|\NZ|-1}\times\C
\]
and $p$ is projection to the second factor $\C$. Let
\[
\toricsurfacedegenerationfibert:=p^{-1}(t^{-1})
\]
denote the fibers. When $t\neq\infty$,
$\toricsurfacedegenerationfibert$ is isomorphic to $\toricsurface$.
When $t=\infty$ (i.e., when $t^{-1}=0$),
\[
\toricsurfacedegenerationfiberzero
=
\bigcup_{v\in\vertices(\tropmcurve)}
\toricsurfacesmall_{v^*},
\]
where the toric surfaces $\toricsurfacesmall_{v^*}$ intersect along
their lines at infinity according to the dual regular subdivision.

Inside $\C\mathbb P^{|\NZ|-1}$, define the hypersurface
\[
H(t)
:=
\left\{
[\ptC_\pt]_{\pt\in\NZ}
\middle|
\sum_{\pt\in\NZ}
t^{-\hei_\pt}c_\pt(t)\ptC_\pt=0
\right\}.
\]
For $t\neq\infty$, the spectral curve is the intersection
\[
\spectralcurve(t)
=
\toricsurfacedegenerationfibert\cap H(t).
\]
When $t=\infty$, intersecting with the hypersurface
\[
H(\infty)
:=
\left\{
[\ptC_\pt]_{\pt\in\NZ}
\middle|
\sum_{\pt\in\NZ}
\lc(c_\pt(t))\ptC_\pt=0
\right\}
\]
gives the special fiber
\[
\degeneratecurve
:=
\toricsurfacedegenerationfiberzero\cap H(\infty).
\]
The curve
\[
\degeneratecurve
=
\bigcup_{v\in\vertices(\tropmcurve)}
\spectralcurve_v,
\qquad
\spectralcurve_v
:=
\toricsurfacesmall_{v^*}\cap\degeneratecurve.
\]
Suppose $\ptR\in\R^2$ are the coordinates of $v$. Then,
$\spectralcurve_v$ is a Harnack curve defined inside
$\toricsurfacesmall_{v^*}$, cut out by the \emph{initial form}
\[
\init_v(P)(\ptC)
:=
\lim_{t\to\infty}
t^{-P^\trop(\ptR)}
P(t^\ptR\ptC;t).
\]

\begin{assumption}
\label{ass:generic-leading-coefficients}
We assume that $[\wt](t)$ is generic in the sense that each
$\spectralcurve_v$ is smooth and transverse to the toric boundary. This is a Zariski-open condition on the leading coefficients of
$[\wt](t)$.
\end{assumption}

Let $\bar e$ be a bounded edge of
$\tropmcurve_{\mathrm{emb}}$, with endpoints $u$ and $v$, and let
$\bar e^*$ be its dual edge. Recall that
\[
\pi^{-1}(\bar e)
=
\{e_1,\dots,e_{m_{\bar e}}\}.
\]
Without loss of generality, after a monomial change of coordinates,
assume that
\[
\bar e=[(-|\bar e|_{\Z},0),(0,0)],
\qquad
\bar e^*=[(0,0),(0,m_{\bar e})].
\]
The edge $\bar e$ defines an affine chart of
$\toricsurfacedegeneration$ of the form
\[
\mathcal U_{\bar e}
=
\left\{
(z_u,z_v,w,t^{-1})
\in
\C^2\times\C^\times\times\C
\;\middle|\;
z_uz_v-t^{-|\bar e|_{\Z}}=0
\right\},
\]
where
\[
z_u=z,
\qquad
z_v=\frac{t^{-|\bar e|_{\Z}}}{z},
\]
and
\begin{equation}
\mathcal U_{\bar e}
\cap
\left(
\toricsurfacesmall_{u^*}
\cap
\toricsurfacesmall_{v^*}
\right)
=
\left\{
(z_u,z_v,w,t^{-1})\in\mathcal U_{\bar e}
\;\middle|\;
z_u=z_v=t^{-1}=0
\right\}.
\label{eq:line_of_intersection}
\end{equation}

In these coordinates, the equation of $\spectralcurve(t)$ is of the
form
\[
F(z_u,z_v,w,t^{-1})
=
f_{\bar e}(w)+R(z_u,z_v,w,t^{-1}),
\]
where
\[
f_{\bar e}(w)
=
\sum_{k=0}^{m_{\bar e}}
\lc(c_{(0,k)}(t))w^k
\]
is a polynomial in $w$ of degree $m_{\bar e}$, and
$R(z_u,z_v,w,t^{-1})$ consists of terms that vanish on
\eqref{eq:line_of_intersection}. By~Assumption~\ref{ass:generic-leading-coefficients},
$f_{\bar e}(w)$ has $m_{\bar e}$ distinct roots
$
w_1,\dots,w_{m_{\bar e}}\in\C^\times.
$
Let $n_i :=(0,0,w_i,0)$. Then,
\[
\spectralcurve_u\cap\spectralcurve_v
=
\{
n_i
\mid
i=1,\dots,m_{\bar e}
\}.
\]
Since
\[
\frac{\partial F}{\partial w}(0,0,w_i,0)
=
f_{\bar e}'(w_i)
\neq0,
\]
the implicit function theorem allows us, in a neighborhood of each
$n_i$, to eliminate $w$, writing $w$ as a holomorphic function of
$z_u,z_v,t^{-1}$. Intersecting with $\mathcal U_{\bar e}$ leaves
\[
z_uz_v
=
t^{-|\bar e|_{\Z}},
\]
which is the standard local model of a nodal degeneration. Thus, the points $n_1,\dots,n_{m_{\bar e}}$ at which $\spectralcurve_u$ and $\spectralcurve_v$ meet are nodes.

We identify the nodes $n_1,\dots,n_{m_{\bar e}}$ with the edges $e_1,\dots,e_{m_{\bar e}}$ lying over $\bar e$ so that the cyclic order of the nodes along the real locus $\mcurve_u(\R)$ agrees with that of the edges around $u$.

We now regard this as a degeneration of pointed curves. For
\(t\gg 1\), let
\(
\bm\nu(t)=(\nu(t)(\alpha))_{\alpha\in\zzG}
\)
be the angle map associated with \([\wt](t)\). Thus,
$
(\spectralcurve(t),\bm\nu(t))
$
is a pointed curve, whose marked points are precisely the points at
infinity of \(\spectralcurve(t)\). Under the toric degeneration, we obtain a pointed nodal curve
$
(\degeneratecurve,\bm\nu(\infty))$, which we claim is stable. For a component
\(\spectralcurve_v\), the number of
special points 
$
n_v
=
\left|\partial v^*\cap\Z^2\right|
$
and the genus
$
g_v
=
\left|(v^*)^\circ\cap\Z^2\right|.
$ By Pick's theorem,
\[
2g_v-2+n_v
=
2\operatorname{Area}(v^*)
>
0.
\]
Thus, every component satisfies the stability condition and
$
(\degeneratecurve,\bm\nu(\infty))
$
is a stable pointed curve.

\subsection{Plumbing}

Let
$
(\degeneratecurve,\bm\nu(\infty))
$
be a stable pointed curve and let $(\tropmcurve,\ell,\bm g)$ be the abstract tropical curve. Plumbing provides a local chart around
$
(\degeneratecurve,\bm\nu(\infty))
\in
\overline{\mathcal M}_{g,n}
$. There are two types of local parameters.

The first type vary the components of the normalization while
preserving the nodes. For each $v\in\vertices(\tropmcurve)$, let
$
(\mcurve_v,\bm\nu_v)
\in
\mathcal M_{g_v,n_v}
$
denote the corresponding component, where $\bm\nu_v$ is the
tuple of special points on $\mcurve_v$ which consist of the preimages of the nodes together with the original marked
points lying on $\mcurve_v$. 
Let
$
U_v\subset\mathcal M_{g_v,n_v}
$
be a sufficiently small neighborhood of
$(\mcurve_v,\bm\nu_v)$, and set
$
U
:=
\prod_{v\in\vertices(\tropmcurve)}U_v.
$
We write
$
\bm u=(u_v)_{v\in\vertices(\tropmcurve)}\in U
$
for these parameters chosen so that
$u_v=0$ corresponds to $(\mcurve_v,\bm\nu_v)$. For $\bm u\in U$, denote the corresponding component pointed curves by
$
(\mcurve_v(\bm u),\bm\nu_v(\bm u)).
$
For each half-edge $h\in H_v$, we denote by
$
n_h(\bm u)
$
the marked point corresponding to the preimage of the node indexed by
$h$. 

The second type of parameters smooth the nodes. For each edge
$e\in\edges(\tropmcurve)$, let
$
q_e\in\C
$
be the corresponding \emph{plumbing parameter}, and write
$
\bm q=(q_e)_{e\in\edges(\tropmcurve)}.
$

We now recall the plumbing construction explicitly. After shrinking
$U$, for each half-edge $h\in H(\tropmcurve)$ we may choose a local
holomorphic coordinate
$
z_h(\bm u)
$
centered at $n_h(\bm u)$ which varies holomorphically with $\bm u$; see
\cite[Section~8.1]{HubbardKoch}. Choose $\rho>0$ sufficiently small so
that, for every $\bm u\in U$ and every
$v\in\vertices(\tropmcurve)$, the coordinate disks
\[
\{
|z_h(\bm u)|<\rho
\},
\qquad
h\in H_v,
\]
are pairwise disjoint and contain no other points of
$\bm\nu_v(\bm u)$.

Suppose first that $q_e\neq0$ for every
$e\in\edges(\tropmcurve)$. For each half-edge
$h\in H(\tropmcurve)$, remove the closed disk
\begin{equation}
\label{eq:general_plumbing_disk}
D_h(\bm u,\bm q)
:=
\left\{
|z_h(\bm u)|
\leq
\frac{|q_{e(h)}|}{\rho}
\right\}
\subset
\mcurve_{v(h)}(\bm u)
\end{equation}
and consider the \emph{collar}
\[
K_h(\bm u,\bm q)
:=
\left\{
\frac{|q_{e(h)}|}{\rho}
<
|z_h(\bm u)|
<
\rho
\right\}.
\]
For each edge $e=\{h,\bar h\}$, glue the collars
$K_h(\bm u,\bm q)$ and $K_{\bar h}(\bm u,\bm q)$ by
\begin{equation}
\label{eq:general_plumbing_map}
z_h(\bm u)
\mapsto
z_{\bar h}(\bm u)
=
\frac{q_e}{z_h(\bm u)}.
\end{equation}
We denote the resulting pointed curve by
$
(\mcurve(\bm u,\bm q),\bm\nu(\bm u,\bm q)),
$
where 
\[
\mcurve(\bm u,\bm q)
:=
\left(
\bigsqcup_{v\in\vertices(\tropmcurve)}
\left(
\mcurve_v(\bm u)
\setminus
\bigsqcup_{h\in H_v}D_h(\bm u,\bm q)
\right)
\right)
\bigg/
\sim,
\]
where $\sim$ is the equivalence relation generated by the gluing
maps~\eqref{eq:general_plumbing_map}, and the points in $\bm\nu(\bm u,\bm q)$ are the original marked points
among the tuples $\bm\nu_v(\bm u)$; the points $n_h(\bm u)$ are no
longer marked after the corresponding node is smoothed.

If $q_e=0$, the node corresponding to $e=\{h,\bar h\}$ is left
unsmoothed, so that
$
n_h(\bm u)\sim n_{\bar h}(\bm u).
$
In particular,
\[
\mcurve(\bm u,\bm 0)
=
\left(
\bigsqcup_{v\in\vertices(\tropmcurve)}
\mcurve_v(\bm u)
\right)
\bigg/
\left(
n_h(\bm u)\sim n_{\bar h}(\bm u)
\text{ for every }
e=\{h,\bar h\}
\right),
\]
and
\[
(\mcurve(\bm 0,\bm 0),
\bm\nu(\bm 0,\bm 0))
=
(\mcurve(\infty),\bm\nu(\infty)).
\]

For $\bm u$ and $\bm q$ sufficiently close to
$(\bm 0,\bm 0)$, the parameters $(\bm u,\bm q)$ give a local
chart for a neighborhood of
$(\mcurve(\infty),\bm\nu(\infty))$ in
$\overline{\mathcal M}_{g,n}$; see
\cite[Sections~8--10]{HubbardKoch}. Therefore, for \(t^{-1} \in [0,\epsilon)\) for some $\epsilon >0$, there is an
isomorphism of pointed curves
\[
(\spectralcurve(t),\bm\nu(t))
\cong 
(
\mcurve(\bm u(t),\bm q(t)),
\bm\nu(\bm u(t),\bm q(t))
)
\]
for some $(\bm u(t), \bm q(t)) \to (\bm 0,\bm 0)$ as $t \to \infty$. By~\cite[Definition~1.1.1 and the paragraph following it]{ACS}, the valuation of a smoothing parameter is independent of the local defining equation. Hence,
\[
\val(q_e(t))=-\ell(e)
\qquad
\text{for every }e\in\edges(\tropmcurve).
\]

Hereafter, since we restrict the plumbing construction to the one-parameter family $(\bm u(t),\bm q(t))$, we suppress the parameters $(\bm u(t),\bm q(t))$ from the notation and indicate the resulting dependence simply by $t$.

\subsection{Cores and degeneration-adapted paths}
\label{sec:tropicalization_of_paths}

Following~\cite[Proof of Theorem~3.6]{HubbardKoch}, for each vertex
$v\in\vertices(\tropmcurve)$, define
\[
\mcurve_v^{\mathrm{core}}(t)
:=
\mcurve_v(t)
\setminus
\bigsqcup_{h\in H_v}
\left\{
|z_h(t)|<\rho
\right\},
\qquad t^{-1}\in[0,\varepsilon).
\]
Let $\sigma(t)$ (resp. $\sigma_v$) denote complex conjugation on $\mcurve(t)$ (resp. $\mcurve_v$). Since the nodes lie on the real locus, we may choose the local coordinate $z_h(t)$ to be compatible with complex conjugation, i.e., so that $z_h(t) \circ \sigma(t) = \overline{z_h(t)}$. Complex conjugation makes 
\[
\bigsqcup_{t^{-1}\in[0,\varepsilon)}
\mcurve_v^{\mathrm{core}}(t)
\rightarrow
[0,\varepsilon)
\]
a $\Z/2\Z$-equivariant map, where the action on the base is trivial.

After shrinking $\varepsilon$ if necessary,
the equivariant form of Ehresmann's fibration theorem for manifolds
with corners~\cite[Lemma~2.34]{GoodwillieIgusaMalkiewichMerling} gives smooth diffeomorphisms
\begin{equation}\label{eq:trivialization}
\bm \Phi = (\Phi_v)_{v \in \vertices(\tropmcurve)},\qquad \Phi_v:
\mcurve_v^{\mathrm{core}}(\infty)\times[0,\varepsilon)
\rightarrow
\bigsqcup_{t^{-1}\in[0,\varepsilon)}
\mcurve_v^{\mathrm{core}}(t),
\qquad
\Phi_v(\cdot;0)=\operatorname{id},
\end{equation}
such that:
\begin{itemize}
\item For every $h\in H_v$, $\Phi_v$ is compatible with the boundary trivializations provided by $z_h(t)$, i.e., 
\[
z_h(t)\circ\Phi_{v}(\cdot;t^{-1})
=
z_h(\infty)
\qquad
\text{on }
\left\{
|z_h(\infty)|=\rho
\right\}.
\]
\item $\bm \Phi$ is $\Z/2\Z$-equivariant (compatible with complex conjugation), i.e., 
\[
\Phi_v(\sigma_v(\cdot);t^{-1})= \sigma(t) \circ \Phi_v(\cdot;t^{-1}).
\]
\end{itemize}
In other words, for $t\gg1$, $\bm \Phi$ provides a smooth trivialization of the family of cores $\bigsqcup_{t^{-1}\in[0,\varepsilon)} \mcurve_v^{\mathrm{core}}(t)$, allowing us to identify each $\mcurve_v^{\mathrm{core}}(t)$ with the fixed reference core $\mcurve_v^{\mathrm{core}}(\infty)$. Such a trivialization is, in general, not holomorphic. We only use the trivialization to pull back integrals of $1$-forms for which we only require smoothness.

A \emph{core path} is a family of paths
\[
\lambda^{\mathrm{core}}(t)
=
\Phi_{v}(\lambda^{\mathrm{core}};t^{-1})
\subset
\mcurve_v^{\mathrm{core}}(t),
\]
obtained by transporting a fixed path
\[
\lambda^{\mathrm{core}}
\subset
\mcurve_v^{\mathrm{core}}(\infty).
\]
Note that the boundary coordinate of an endpoint of a core path is independent of $t$.

A \emph{collar path} associated with an edge
$e=\{h,\bar h\}$ is a smooth family of paths
\[
\lambda^{\mathrm{collar}}(t)
\subset
K_h(t)\sim K_{\bar h}(t),
\qquad t\gg1.
\]
For each endpoint $\ptSig(t)$ of $\lambda^{\mathrm{collar}}(t)$, we require
that
\begin{equation}\label{eq:in_pt_trop}
\lim_{t\to\infty}
-\frac{\log |z_h(t)(\ptSig(t))|}{\log t}
\end{equation}
exists. We denote the point of $e$ at distance \eqref{eq:in_pt_trop} from $v(h)$ by $\ptSig^\trop$. Note that we put no restriction on the homotopy class of a collar path inside
the collar; in particular, it may wind around the collar.

A family of paths $(\gamma(t))_{t\gg1}$ is called
\emph{degeneration-adapted} if, for all sufficiently large $t$, it can
be written as a finite concatenation of core paths and collar paths.
Whenever a core path and a collar path are consecutive, their common
endpoint lies on the corresponding boundary circle of the core.
In particular, its boundary coordinate is independent of $t$.

The \emph{tropicalization} of a degeneration-adapted family
$(\gamma(t))_{t\gg1}$ is the oriented path
\[
\gamma^\trop\subset\tropmcurve
\]
obtained by contracting every core path to the corresponding vertex and
replacing every collar path by the oriented segment of the corresponding
edge joining the tropical limits of its endpoints. Conversely, every oriented path in $\tropmcurve$ admits a
degeneration-adapted lift.

Although \eqref{eq:in_pt_trop} may appear difficult to verify directly,
in our applications it is usually automatic. The endpoints of the paths
that we consider will be $\Puis$-valued points. After passing from
plumbing coordinates to the toric coordinates of
\cref{sec:degenerate_curve}, the limit in \eqref{eq:in_pt_trop} can
equivalently be computed as $\val(z_{v(h)}(\ptSig(t)))$. This determines the embedded tropical limit
$\pi(\ptSig^\trop)$, while the plumbing collar containing $\ptSig(t)$
determines the edge of the abstract tropical curve containing
$\ptSig^\trop$. Together, these two pieces of data determine
$\ptSig^\trop$.

\subsection{Degeneration-adapted symplectic basis}\label{sec:degen_adapted_basis}

We construct a symplectic basis of
$H_1(\mcurve(t),\Z)$ that is adapted to the degeneration
$
\mcurve(t)
\rightarrow
\mcurve(\infty).
$
For every vertex $v\in\vertices(\tropmcurve)$, let
\[
(\bm A_v^{\mathrm{ov}}
=
(
A_{v,1}^{\mathrm{ov}},
\dots,
A_{v,g_v}^{\mathrm{ov}}
),
\bm B_v^{\mathrm{ov}}
=
(
B_{v,1}^{\mathrm{ov}},
\dots,
B_{v,g_v}^{\mathrm{ov}}
))
\]
be a symplectic basis of
$H_1(\mcurve_v,\Z)$ as in \cref{sec:standard_divisors}, represented by cycles
contained in $\mcurve_v^{\mathrm{core}}$. For $t$ sufficiently large,
we regard these cycles as cycles on $\mcurve(t)$ via the trivialization $\bm \Phi$ and denote them by
$
A_{v,a}^{\mathrm{ov}}(t),
B_{v,a}^{\mathrm{ov}}(t),
1\leq a\leq g_v.
$
Since the trivialization is compatible with complex conjugation and since the ovals of a Harnack curve do not meet, $B_{v,a}^{\mathrm{ov}}(t)$ is just the corresponding oval of $\mcurve_v(t)$ and the $A$-cycles $A_{v,a}^{\mathrm{ov}}(t)$ are anti-invariant.

 Denote
\begin{equation}\label{eq:comp_cycles}
\bm A^{\mathrm{comp}}(t) = (
\bm A_{v}^{\mathrm{ov}}(t)
)_{
{
v\in\vertices(\tropmcurve)
}
}, \qquad \bm B^{\mathrm{comp}}(t) = (
\bm B_{v}^{\mathrm{ov}}(t)
)_{
{
v\in\vertices(\tropmcurve)
}
} 	
\end{equation}
Since 
\[
A_{v,a}^{\mathrm{ov}}\wedge B_{v,0}^{\mathrm{ov}}=-1,
\]
we can assume they intersect transversely exactly once.

To construct the remaining cycles, we fix a spanning tree $T\subset\tropmcurve$. For each edge
$
e\notin\edges(T),
$
let
$
B_e^\trop\subset\tropmcurve
$
be the clockwise-oriented cycle consisting of $e$ together with the unique path
in $T$ joining the two endpoints of $e$. The cycles
\[
\bm B^\trop := (B_e^\trop
)_{e\notin \edges(T)}
\]
form a basis of $H_1(\tropmcurve,\Z)$. For each $e\notin \edges(T)$, let $F_e$ be the set of faces in the
component of the $T^*\setminus\{e^*\}$ not containing the outer face $f_0$, where $T^*$ is the planar dual tree and $e^*$ is the edge dual to $e$. Then,
$
B_e^\trop
=
\sum_{f\in F_e} B_f^{\mathrm{ov},\trop}.
$
We therefore obtain a degeneration-adapted lift by
\begin{equation}\label{eq:ovals_to_deg}
\widetilde B_e^{\mathrm{trop}}(t)
:=
\sum_{f\in F_e} B_f^{\mathrm{ov}}(t),
\end{equation}
where $B_f^{\mathrm{ov}}(t) \subset \mcurve(t)$ is the oval tropicalizing to $f$. Since the cycles $B_f^{\mathrm{ov}}(t)$ and
$B_{v,a}^{\mathrm{ov}}(t)$ are ovals, they are pairwise disjoint. 

For $e=\{h,\bar h\}\notin T$, define the \emph{seam cycle}
\[
A_e^{\mathrm{seam}}(t)
:=
\{
|z_h(t)|
=
\sqrt{|q_e(t)|}
\} =\{|z_{\bar h}(t)|
=
\sqrt{|q_e(t)|} \}.
\]
in the plumbing collar corresponding to $e$. Orient $A_e^{\mathrm{seam}}(t)$ so that
\[
A_e^{\mathrm{seam}}(t)
\wedge
\widetilde B_e^{\mathrm{trop}}(t)
=
1.
\]
Let 
\[
\bm A^{\mathrm{seam}}(t) = (A_e^{\mathrm{seam}})_{e \notin \edges(T)}, \qquad \widetilde{\bm{B}}^{\mathrm{trop}}(t) = (\widetilde B_e^{\mathrm{trop}})_{e \notin \edges(T)}.
\]
Since $B_e^\trop$ contains no edge of
$\edges(\tropmcurve)\setminus\edges(T)$ other than $e$, we have
\[
\bm A^{\mathrm{seam}}(t) \wedge \widetilde{\bm{B}}^{\mathrm{trop}}(t) = I.
\]
Moreover, the seam cycles are mutually disjoint, and they are disjoint
from all component cycles.

It remains only to make the intersection of the tropical $B$-cycles with the
component $A$-cycles equal to $0$. To that end, set
\[
{\bm{B}}^{\mathrm{trop}}(t) = (B_e^{\mathrm{trop}})_{e \notin \edges(T)}, \qquad B_e^{\mathrm{trop}}(t)
:=
\widetilde B_e^{\mathrm{trop}}(t)
-
\sum_{\substack{v\in\vertices(\tropmcurve)\\1\leq a\leq g_v}}
(
A_{v,a}^{\mathrm{ov}}(t)
\wedge \widetilde B_e^{\mathrm{trop}}(t)
)
B_{v,a}^{\mathrm{ov}}(t).
\]
The cycles
\[
\bm A^{\mathrm{deg}}(t)
:=
(
\bm A^{\mathrm{comp}}(t),
\bm A^{\mathrm{seam}}(t)
),
\qquad
\bm B^{\mathrm{deg}}(t)
:=
(
\bm B^{\mathrm{comp}}(t),
\bm B^{\mathrm{trop}}(t)
)
\]
form a symplectic basis of
$
H_1(\mcurve(t),\Z).
$
We call
$
(
\bm A^{\mathrm{deg}}(t),
\bm B^{\mathrm{deg}}(t)
)
$
a \emph{degeneration-adapted symplectic basis}.

\section{Asymptotics under degeneration}\label{sec:asymptotics}

\subsection{Stable differentials}
\label{sec:stable_differentials}

A \emph{stable differential} on $\mcurve(\infty)$ is a collection
\[
\Omega_\bullet
=
(
\Omega_v
)_{v\in\vertices(\tropmcurve)}
\]
such that, for every vertex $v$, the differential $\Omega_v$ is
meromorphic on $\mcurve_v$, holomorphic away from the points
$
\{
n_h
\mid
h\in H_v
\},
$
and has at worst simple poles at these points. Moreover, for every edge
$e=\{h,\bar h\}$, the residues satisfy
\begin{equation}\label{eq:residue_matching}
\res_{n_h}\Omega_{v(h)}
+
\res_{n_{\bar h}}\Omega_{v(\bar h)}
=
0.
\end{equation}
For each $e\notin\edges(T)$, choose the half-edge
$h_e$ so that the seam cycle $A_e^{\mathrm{seam}}(t)$ is positively
oriented in the $z_{h_e}(t)$-coordinate. Define
\[
\bm a(\Omega_\bullet)
:=
\left(
\int_{A_{v,a}}\Omega_v
\right)_{
\substack{
v\in\vertices(\tropmcurve)\\
1\leq a\leq g_v
}},
\qquad
\bm r(\Omega_\bullet)
:=
\left(
\res_{n_{h_e}}\Omega_{v(h_e)}
\right)_{
e\notin\edges(T)
}.
\]

The \emph{tropical $1$-form associated to $\Omega_\bullet$} is defined by
\[
\Omega_\bullet^\trop(h)
:=
\res_{n_h}\Omega_{v(h)},
\qquad
h\in H(\tropmcurve).
\]
The condition~\eqref{eq:residue_matching} gives
$
\Omega_\bullet^\trop(\bar h)
=
-\Omega_\bullet^\trop(h),
$
while the residue theorem on each component gives
\[
\sum_{h\in H_v}
\Omega_\bullet^\trop(h)
=
0.
\]
Thus, $\Omega_\bullet^\trop$ is a tropical $1$-form on
$\tropmcurve$.

The following two results are stated in \cite{HuNorton} under the
assumption that the moduli $\bm u$ are independent of $t$. However,
\cite[Remark~2.4]{HuNorton} explains that these results extend to the
more general setting stated below.

\begin{theorem}[{\cite[Theorem~1.1]{HuNorton}}]
\label{thm:HuNorton:main1}
Let
$
\Omega_\bullet
=
(\Omega_v)_{v\in\vertices(\tropmcurve)}
$
be a stable differential on $\mcurve(\infty)$. Then, for $t\gg1$, there
exists a unique holomorphic differential
$\Omega(t)$ on $\mcurve(t)$ satisfying
\[
\int_{\bm A^{\mathrm{deg}}(t)}
\Omega(t)
=
\begin{pmatrix}
\bm a(\Omega_\bullet)\\
2\pi\imunit\,\bm r(\Omega_\bullet)
\end{pmatrix}.
\]
Moreover, for every vertex $v$ and every compact subset
$
K
\subset
\mcurve_v
\setminus
\{
n_h\mid h\in H_v
\},
$
we have
\[
\Phi_v(\cdot;t^{-1})^*\Omega(t)\to\Omega_v
\]
uniformly on $K$.
\end{theorem}

\begin{theorem}[{\cite[Theorem~1.2]{HuNorton}}]
\label{thm:tropical_integration}
For $t\gg1$, let $\gamma(t)$ be a degeneration-adapted family of paths
with tropicalization $\gamma^\trop$. Let
$
\Omega_\bullet
=
(\Omega_v)_{v\in\vertices(\tropmcurve)}
$
be a stable differential on $\mcurve(\infty)$, with associated tropical
$1$-form $\Omega_\bullet^\trop$, and let $\Omega(t)$ be the family of
holomorphic differentials provided by
\cref{thm:HuNorton:main1}. Then, as $t\to\infty$,
\[
\int_{\gamma(t)}
\Omega(t)
=
-\left(
\int_{\gamma^\trop}
\Omega_\bullet^\trop
\right)
\log t
+
O(1).
\]
\end{theorem}

\subsection{The period matrix and the Abel map}

The period-matrix asymptotics in this section are essentially due to Iwao~\cite[Theorem~4.3.1]{Iwao}, and his analysis also gives the corresponding Abel-map asymptotics in the plane-curve setting. We rederive these statements using plumbing since it allows us to remove certain technical assumptions that he uses. The same method will also be used for normalized differentials of the third kind where Iwao's results do not apply. We first apply \cref{thm:HuNorton:main1} to construct the basis of holomorphic differentials on $\mcurve(t)$ dual to $\bm A^{\mathrm{deg}}(t)$ and then apply \cref{thm:tropical_integration} to this basis. 

For every vertex $v\in\vertices(\tropmcurve)$, let $\bm \omega^{\mathrm{ov}}_v = (\omega_{v,1}^{\mathrm{ov}},\dots,\omega_{v,g_v}^{\mathrm{ov}})
$ be the normalized basis of holomorphic differentials on $\mcurve_v$, so that
\[
\int_{\bm A_{v}^{\mathrm{ov}}}
\bm \omega_{v}^{\mathrm{ov}}
=
2\pi\imunit I.
\]
For each vertex $v$ and each $1\leq a\leq g_v$, define a stable
differential $\Omega_{v,a;\bullet}
=
(
\Omega_{v,a;u}
)_{u\in\vertices(\tropmcurve)}
$ by
\[
\Omega_{v,a;u}
:=
\begin{cases}
\omega_{v,a}^{\mathrm{ov}},
&
u=v,
\\
0,
&
u\neq v.
\end{cases}
\]
Then,
\[
\bm a(\Omega_{v,a;\bullet})
=
2\pi\imunit \bm e_{v,a},
\qquad
\bm r(\Omega_{v,a;\bullet})
=
0,
\]
where $\bm e_{v,a} \in \C^{g^{\mathrm{comp}}}$ denotes the standard basis
vector corresponding to $(v,a)$. By
\cref{thm:HuNorton:main1}, there exists a unique holomorphic differential $\omega_{v,a}^{\mathrm{deg}}(t)$ on $\mcurve(t)$ satisfying
\[
\int_{\bm A^{\mathrm{deg}}(t)}
\omega_{v,a}^{\mathrm{deg}}(t)
=
2\pi\imunit
\begin{pmatrix}
\bm e_{v,a}\\
0
\end{pmatrix}.
\]

We next construct the differentials corresponding to the seam cycles.
Let
$
\omega_e^{\mathrm{deg},\trop}
$
be the tropical $1$-form given by unit flow along $B_e^\trop$. The forms
\[
\bm\omega^{\mathrm{deg},\trop}
:=
(
\omega_e^{\mathrm{deg},\trop}
)_{e \notin \edges(T)}
\]
form a basis of tropical $1$-forms on $\tropmcurve$.

For each
$
e \notin \edges(T)
$
and each
$
v\in\vertices(\tropmcurve),
$
let $\Omega_{e;v}^{\mathrm{deg}}$ be the unique $A$-normalized
differential of the third kind on $\mcurve_v$ satisfying
\[
\res_{n_h}\Omega_{e;v}^{\mathrm{deg}}
=
\omega_e^{\mathrm{deg},\trop}(h),
\qquad
h\in H_v.
\]
Then,
\[
\Omega_{e;\bullet}^{\mathrm{deg}}
:=
(
\Omega_{e;v}^{\mathrm{deg}}
)_{v\in\vertices(\tropmcurve)}
\]
is a stable differential on $\mcurve(\infty)$, whose associated
tropical $1$-form is
$
\omega_e^{\mathrm{deg},\trop}.
$
Moreover,
\[
\bm a(\Omega_{e;\bullet}^{\mathrm{deg}})
=
0,
\qquad
\bm r(\Omega_{e;\bullet}^{\mathrm{deg}})
=
\bm e_e,
\]
where $\bm e_e\in\C^{g^\trop}$ is the standard basis vector
corresponding to $e$. By
\cref{thm:HuNorton:main1}, there exists a unique holomorphic differential
$
\omega_e^{\mathrm{deg}}(t)
$
on $\mcurve(t)$ satisfying
\[
\int_{\bm A^{\mathrm{deg}}(t)}
\omega_e^{\mathrm{deg}}(t)
=
2\pi\imunit
\begin{pmatrix}
0\\
\bm e_e
\end{pmatrix}.
\]

Therefore,
\[
\bm\omega^{\mathrm{deg}}(t)
:=
\left(
(
\omega_{v,a}^{\mathrm{deg}}(t)
)_{
\substack{
v\in\vertices(\tropmcurve)\\
1\leq a\leq g_v
}},
(
\omega_e^{\mathrm{deg}}(t)
)_{e\notin\edges(T)}
\right)
\]
is the basis of holomorphic differentials on $\mcurve(t)$ normalized
with respect to $\bm A^{\mathrm{deg}}(t)$.

Define the \emph{degeneration-adapted tropical period matrix} by
\[
\Pi^{\mathrm{deg},\trop}
:=
\left(
\int_{B_{e'}^\trop}
\omega_e^{\mathrm{deg},\trop}
\right)_{
e,e'\notin\edges(T)
}.
\]

By \cref{thm:tropical_integration}, we obtain:

\begin{proposition}
\label{cor:periodtropical}
As \(t\to\infty\), the period matrix
\begin{equation}
\Pi^{\mathrm{deg}}(t)
=
-
\begin{pmatrix}
0&0\\
0&\Pi^{\mathrm{deg},\trop}
\end{pmatrix}
\log t
+
O(1).
\label{eq:period_asympotics_ov}
\end{equation}
\end{proposition}

Similarly, for the Abel map, we have:

\begin{proposition}
\label{cor:abelmap}
Let $(\gamma(t))_{t\gg1}$ be a degeneration-adapted family of paths
with tropicalization $\gamma^\trop$. Then, as $t\to\infty$,
\[
\int_{\gamma(t)}
\bm\omega^{\mathrm{deg}}(t)
=
-
\begin{pmatrix}
0\\
\displaystyle
\int_{\gamma^\trop}
\bm\omega^{\mathrm{deg},\trop}
\end{pmatrix}
\log t
+
O(1).
\]
In particular, suppose that $\gamma(t)$ joins the basepoint
$\basept(t)$ to a family of points $\ptSig(t)\in\mcurve(t)$, and let
$\ptSig^\trop\in\tropmcurve$ be the endpoint of $\gamma^\trop$.
Then, the lift of the Abel map determined by $\gamma(t)$ satisfies
\[
\mu_{\basept(t)}^{\mathrm{deg}}(\ptSig(t))
=
-
\begin{pmatrix}
0\\
\displaystyle
\int_{\baseptT}^{\ptSig^\trop}
\bm\omega^{\mathrm{deg},\trop}
\end{pmatrix}
\log t
+
O(1).
\]
\end{proposition}

\subsection{Differentials of the third kind}
\label{sec:third_kind_asymptotics}

The goal of this section is to obtain asymptotics for integrals of
$A$-normalized differentials of the third kind of the form
$\omega_{\beta(t)-\alpha(t)}^{\mathrm{deg}}(t)$, where $\alpha(t)$ and
$\beta(t)$ are families of points on $\mcurve(t)$. The main idea is to
regard this differential as the limit of a holomorphic differential
after introducing an additional node identifying $\alpha(t)$ and
$\beta(t)$. We then apply the multivariable versions of
\cref{thm:HuNorton:main1} and \cref{thm:tropical_integration} proved
in~\cite{HuNorton}.

Related asymptotics for imaginary-normalized differentials of the third kind using Schottky-uniformized degenerations were obtained by Ichikawa~\cite[Theorem~4.7]{IchikawaHarmonicAmoebas}.

\begin{proposition}
\label{prop:trop_differential_third_kind}
Let
$
\alpha\in\mcurve_{v_\alpha}^{\mathrm{core}}(\infty)
$
and
$
\beta\in\mcurve_{v_\beta}^{\mathrm{core}}(\infty),
$
and let 
\[
\alpha(t) =\Phi_{v_\alpha}(\alpha;t^{-1}),\qquad \beta(t) =\Phi_{v_\beta}(\beta;t^{-1})
\]
be the corresponding points of
$\mcurve(t)$. Set
$
\alpha^\trop:=v_\alpha$ and $\beta^\trop:=v_\beta.
$
Let $\delta^\trop$ be the unique path in the spanning tree $T$ from
$\alpha^\trop$ to $\beta^\trop$, and define the tropical
third-kind differential $\omega_{\beta^\trop-\alpha^\trop}^{\mathrm{deg},\trop}$ as the tropical $1$-form given by unit flow along $\delta^\trop$.

Let $(\gamma(t))_{t\gg1}$ be a degeneration-adapted family of paths
such that $\gamma(t)$ is disjoint from $\alpha(t)$ and $\beta(t)$, and
let $\gamma^\trop$ be its tropicalization. Then, as $t\to\infty$,
\[
\int_{\gamma(t)}
\omega_{\beta(t)-\alpha(t)}^{\mathrm{deg}}(t)
=
-\left(
\int_{\gamma^\trop}
\omega_{\beta^\trop-\alpha^\trop}^{\mathrm{deg},\trop}
\right)
\log t
+
O(1).
\]
\end{proposition}

\begin{proof}

Form a new nodal curve by identifying $\alpha$ and $\beta$, and let
$\widehat{\tropmcurve}$ be the corresponding abstract tropical curve obtained
from $\tropmcurve$ by adding an edge $e_*$ between
$\alpha^\trop$ and $\beta^\trop$.

Choose local holomorphic coordinates $z_\alpha(t)$ and $z_\beta(t)$ centered
at $\alpha$ and $\beta$, respectively, whose coordinate disks are
contained in the corresponding reference cores. Introduce an
independent plumbing parameter $q_*$ at the new node, with plumbing
equation
\[
z_\alpha(t) z_\beta(t)=q_*.
\]
Together with the original plumbing parameters $\bm q(t)$, this gives a
two-parameter family which we denote by
$
\widehat\mcurve(q_*,t).
$
For $q_*\neq0$, the additional node is smoothed, while at $q_*=0$ its
normalization is $\mcurve(t)$ with the two preimages of the node equal
to $\alpha(t)$ and $\beta(t)$.

After choosing the coordinate disks sufficiently small, the complement
of these disks is canonically identified with the corresponding subset
of $\mcurve(t)$:
\begin{equation}\label{eq:two-paramter-core}
\widehat\mcurve(q_*,t)
\setminus
\left(
\{|z_\alpha(t)|<\rho\}
\cup
\{|z_\beta(t)|<\rho\}
\right)
=
\mcurve(t)
\setminus
\left(
\{|z_\alpha(t)|<\rho\}
\cup
\{|z_\beta(t)|<\rho\}
\right).
\end{equation}

Let
$
\widehat\omega_*^\trop
$
be the tropical $1$-form on $\widehat\tropmcurve$ whose restriction to
$\tropmcurve$ is
$
\omega_{\beta^\trop-\alpha^\trop}^\trop
$
and whose value on $e_*$ is chosen so that the balancing condition
holds at $\alpha^\trop$ and $\beta^\trop$. Let
$
\widehat\Omega_{\bullet;*}
=
(\widehat\Omega_{v;*})_v
$
be the stable differential associated to
$\widehat\omega_*^\trop$. Thus, its residues at the old nodes are given
by
\[
\res_{n_h}\widehat\Omega_{v(h);*}
=
\widehat\omega_*^\trop(h),
\]
while at the two branches of the new node they are
\[
\res_{\beta}\widehat\Omega_{v_\beta;*}=1,
\qquad
\res_{\alpha}\widehat\Omega_{v_\alpha;*}=-1.
\]
On each component, $\widehat\Omega_{v;*}$ is normalized to have
vanishing periods along the component $A$-cycles.

By the multivariable version of \cref{thm:HuNorton:main1}, there is a
family of holomorphic differentials
$
\widehat\omega_*(q_*,t)
$
on $\widehat\mcurve(q_*,t)$ with the prescribed periods which converges to
$\widehat\Omega_{\bullet;*}$ on compact subsets away from the nodes.

Now fix $t$ and let $q_*\to0$. On the normalization $\mcurve(t)$, the
limiting differential has simple poles at $\alpha(t)$ and $\beta(t)$,
with residues $-1$ and $1$, respectively, and has vanishing periods
along $\bm A^{\mathrm{deg}}(t)$. Hence, by uniqueness,
\[
\lim_{q_*\to0}
\widehat\omega_*(q_*,t)
=
\omega_{\beta(t)-\alpha(t)}^{\mathrm{deg}}(t)
\]
uniformly on compact subsets of
$
\mcurve(t)\setminus\{\alpha(t),\beta(t)\}.
$

We now apply the multivariable version of
\cref{thm:tropical_integration}. Since $\gamma(t)$ is
degeneration-adapted and disjoint from $\alpha(t)$ and $\beta(t)$, the
coordinate disks may be chosen so that $\gamma(t)$ lies in the common
region~\eqref{eq:two-paramter-core}. In particular, $\gamma(t)$ does
not cross the new plumbing collar, so the asymptotic expansion
of
$
\int_{\gamma(t)}\widehat\omega_*(q_*,t)
$
has no logarithmic contribution from $q_*$. Letting $q_* \to 0$ therefore
gives
\[
\int_{\gamma(t)}
\omega_{\beta(t)-\alpha(t)}^{\mathrm{deg}}(t)
=
-
\left(
\int_{\gamma^\trop}
\widehat\omega_*^\trop
\right)
\log t
+
O(1).
\]
Finally, since $\gamma^\trop$ is contained in the original graph
$\tropmcurve$ and
$
\widehat\omega_*^\trop
$
restricts to
$
\omega_{\beta^\trop-\alpha^\trop}^\trop,
$
we have
\[
\int_{\gamma^\trop}
\widehat\omega_*^\trop
=
\int_{\gamma^\trop}
\omega_{\beta^\trop-\alpha^\trop}^\trop.
\]
This proves the proposition.
\end{proof}

\subsection{The vector of Riemann constants}
\label{sec:trop_riemann_constants}

In this section, we compute the asymptotics of the vector of Riemann
constants.

Recall that we fixed a spanning tree
$
T\subset\tropmcurve.
$
For each
$
e\notin\edges(T),
$
recall that the half-edge $h_e$ was chosen so that the seam cycle
$A_e^{\mathrm{seam}}(t)$ is positively oriented in the
$z_{h_e}(t)$-coordinate. With our orientation conventions,
$B_e^\trop$ traverses $e$ from $v(h_e)$ to $v(\bar h_e)$.

Let $s_e$ be the midpoint of $e$. We define the \emph{cut tropical curve}
$\tropmcurve_{\mathrm{cut}}$ by cutting every edge
$
e\notin\edges(T)
$
at $s_e$. Thus, $s_e$ is replaced by two points
$
s_e^-
$
and
$
s_e^+,
$
where $s_e^-$ belongs to the half-edge incident to $v(h_e)$ and
$s_e^+$ belongs to the half-edge incident to $v(\bar h_e)$. In
particular, the oriented cycle $B_e^\trop$ traverses the seam from the
$s_e^-$-side to the $s_e^+$-side. 

Since $T$ is a spanning tree,
$\tropmcurve_{\mathrm{cut}}$ is itself a tree. In particular, any two points of
$\tropmcurve_{\mathrm{cut}}$ are joined by a unique path. All tropical
integrals in this section are taken along these unique paths.

\begin{proposition}
\label{prop:riemann_constants_trop}
Let
$
\vecRC^{\mathrm{deg}}(t)
=
(\RC_j^{\mathrm{deg}}(t))_{j=1}^g
\in\C^g
$
denote the vector of Riemann constants of $\mcurve(t)$ with respect to
the symplectic basis
$
(
\bm A^{\mathrm{deg}}(t),
\bm B^{\mathrm{deg}}(t)
)
$
and basepoint $\basept(t)$. Then,
\[
\vecRC^{\mathrm{deg}}(t)
=
-
\begin{pmatrix}
0\\
\vecRC^{\mathrm{deg},\trop}
\end{pmatrix}
\log t
+
O(1),
\]
where
$
\vecRC^{\mathrm{deg},\trop}
=
(
\RC_e^{\mathrm{deg},\trop}
)_{e\notin\edges(T)}
\in
\R^{g^\trop}
$
is given by
\begin{equation}
\label{eq:RC_deg_trop}
\RC_e^{\mathrm{deg},\trop}
=
-\frac{1}{2}\Pi_{ee}^{\mathrm{deg},\trop}
+
\sum_{v\in\vertices(\tropmcurve)}
g_v
\int_{\baseptT}^{v}
\omega_e^{\mathrm{deg},\trop}
+
\sum_{\substack{
e'\neq e\\
e'\notin\edges(T)
}}
\int_{\baseptT}^{s_{e'}^+}
\omega_e^{\mathrm{deg},\trop}.
\end{equation}
\end{proposition}

\begin{example}
\label{ex:riemann_constants_non_smooth}

Consider again the abstract tropical curve in
\cref{fig:embedded-abstract-dual-subdivision}(B), and choose the
spanning tree
\[
T=\{e_1,e_2,e_3\}.
\]
Thus, the edges not contained in $T$ are $e_4$ and $e_5$.
We orient the edges by
\[
e_1:v_4\to v_1,\qquad
e_2:v_3\to v_4,\qquad
e_3:v_2\to v_3,
\qquad
e_4,e_5:v_1\to v_2.
\]
With our clockwise orientation convention for the bounded faces,
\[
B_1^{\mathrm{ov},\trop}
=
e_5-e_4,
\qquad
B_2^{\mathrm{ov},\trop}
=
e_1+e_2+e_3+e_4.
\]
The degeneration-adapted tropical cycles are given by
\[
B_{e_4}^\trop
=
e_1+e_2+e_3+e_4=B_2^{\mathrm{ov},\trop},
\qquad
B_{e_5}^\trop
=
e_1+e_2+e_3+e_5=B_1^{\mathrm{ov},\trop}
+
B_2^{\mathrm{ov},\trop}.
\]
The degeneration-adapted period matrix is
\[
\Pi^{\mathrm{deg},\trop}
=
\begin{pmatrix}
21&18\\
18&21
\end{pmatrix}.
\]
We choose the tropical basepoint
$
\baseptT=v_1.
$
Since $B_{e_4}^\trop$ (resp. $B_{e_5}^\trop$) traverses $e_4$ (resp. $e_5$) from $v_1$ to $v_2$, the point
$
s_{e_4}^+
$
(resp.
$
s_{e_5}^+
$
) lies on the half-edge of $\tropmcurve_{\mathrm{cut}}$ incident to
$v_2$.

First consider $e=e_4$. Since
$\omega_{e_4}^{\mathrm{deg},\trop}$ vanishes on $e_5$, we have
\[
\int_{v_1}^{s_{e_5}^+}
\omega_{e_4}^{\mathrm{deg},\trop}
=
-(5+8+5)
=
-18.
\]
Using
$
\Pi_{e_4e_4}^{\mathrm{deg},\trop}=21$ and evaluating~\eqref{eq:RC_deg_trop}, we get
\[
\RC_{e_4}^{\mathrm{deg},\trop}
=
-\frac12\cdot21-18
=
-\frac{57}{2}.
\]
By symmetry under exhcanging $e_4$ and $e_5$, we get
\[
\vecRC^{\mathrm{deg},\trop}
=
\begin{pmatrix}
-{57}/{2}\\
-{57}/{2}
\end{pmatrix}.
\]
\end{example}

The rest of this section is devoted to the proof of
\cref{prop:riemann_constants_trop}.

Let $\mcurve$ be a compact Riemann surface of genus $g$, let
$\basept\in\mcurve$ be a basepoint, and let $(\bm A,\bm B)$ be a
symplectic basis represented by a canonical system of loops based at
$\basept$. Let $\bm\omega$ be the corresponding normalized basis of
holomorphic differentials, let $\Pi$ be the period matrix, and let
$\vecRC$ be the vector of Riemann constants.

Dissecting $\mcurve$ along $(\bm A,\bm B)$ gives a polygon
$\mathfrak P$. Each $A$-cycle $A_j$ has two lifts to $\mathfrak P$;
label them $A_j^-$ and $A_j^+$ so that the corresponding $B$-cycle
crosses from $A_j^-$ to $A_j^+$. Then, Fay~\cite[Eq.~(13)]{Fay} gives
\begin{equation}
\label{eq:Fay_cRS}
\RC_j
=
\pi\imunit
-\frac{1}{2}\Pi_{jj}
+
\frac{1}{2\pi\imunit}
\sum_{k\neq j}
\int_{\zeta\in A_k^+}
\omega_k(\zeta)
\int_{\basept}^{\zeta}
\omega_j,
\qquad
1\leq j\leq g,
\end{equation}
where the inner integral is taken along the unique path from $\basept$
to $\zeta$ contained in $\mathfrak P$.

In order to apply \eqref{eq:Fay_cRS} to $\mcurve(t)$, we first choose
representatives of the degeneration-adapted symplectic basis that pass
through the basepoint $\basept(t)$. We use the component $A$-cycles
$A_{v,a}^{\mathrm{ov}}(t)$ and the seam cycles
$A_e^{\mathrm{seam}}(t)$ themselves, without homotoping the latter out
of the plumbing collars.

For each $A_k^{\mathrm{deg}}(t)$, choose a point
$
\eta_k(t)\in A_k^{\mathrm{deg}}(t).
$
Define
\[
\tau_k^\trop
:=
\begin{cases}
v,
&
k=(v,a),
\\
s_e^+,
&
k=e\notin\edges(T).
\end{cases}
\]
Let $\lambda_k^\trop$ be the unique path in
$\tropmcurve_{\mathrm{cut}}$ from $\baseptT$ to
$\tau_k^\trop$, and choose a degeneration-adapted lift $\lambda_k(t)$
joining $\basept(t)$ to $\eta_k(t)$. In the case
$
k=e\notin\edges(T),
$
we choose $\lambda_k(t)$ to approach the seam from the side
corresponding to $s_e^+$. We may then replace
$A_k^{\mathrm{deg}}(t)$ by
\[
\widehat A_k^{\mathrm{deg}}(t)
=
\lambda_k(t)
\cdot
A_k^{\mathrm{deg}}(t)
\cdot
\lambda_k(t)^{-1}.
\]
After a small perturbation, we may replace
$\widehat A_k^{\mathrm{deg}}(t)$ by a simple loop based at
$\basept(t)$. Similarly, we choose representatives
$\widehat B_k^{\mathrm{deg}}(t)$ of the $B$-cycles passing through
$\basept(t)$ so that
\[
(
\widehat{\bm A}^{\mathrm{deg}}(t),
\widehat{\bm B}^{\mathrm{deg}}(t)
)
\]
forms a canonical system of loops based at $\basept(t)$. Dissecting
$\mcurve(t)$ along these loops, we obtain a polygon
$\mathfrak P_{\mcurve(t)}^{\mathrm{deg}}$. Under tropicalization, the
two copies of $A_e^{\mathrm{seam}}(t)$ on the boundary of this polygon
correspond to $s_e^-$ and $s_e^+$, with
$(A_e^{\mathrm{seam}}(t))^+$ corresponding to $s_e^+$.

Index the degeneration-adapted $A$-cycles by $\{1,\dots,g\}$ so that
the indices corresponding to
$
\{
(v,a)
\mid
v\in\vertices(\tropmcurve),
\ 1\leq a\leq g_v
\}
$
appear first and followed by the indices corresponding to
$
\edges(\tropmcurve)\setminus\edges(T)
$. Applying \eqref{eq:Fay_cRS}, we obtain
\[
\RC_j^{\mathrm{deg}}(t)
=
\pi\imunit
-
\frac{1}{2}\Pi_{jj}^{\mathrm{deg}}(t)
+
\frac{1}{2\pi\imunit}
\sum_{k\neq j}
\int_{\zeta\in
(\widehat A_k^{\mathrm{deg}}(t))^+}
\omega_k^{\mathrm{deg}}(t)(\zeta)
\int_{\basept(t)}^\zeta
\omega_j^{\mathrm{deg}}(t).
\]
We now show that we may replace
$\widehat A_k^{\mathrm{deg}}(t)$ by $A_k^{\mathrm{deg}}(t)$ in this
formula.

\begin{lemma}
\label{lem:Fay_RC_deg}
For $1\leq j\leq g$, we have
\begin{equation}
\label{eq:Fay_RC_deg}
\RC_j^{\mathrm{deg}}(t)
=
\pi\imunit
-
\frac{1}{2}\Pi_{jj}^{\mathrm{deg}}(t)
+
\frac{1}{2\pi\imunit}
\sum_{k\neq j}
\int_{\zeta\in(A_k^{\mathrm{deg}}(t))^+}
\omega_k^{\mathrm{deg}}(t)(\zeta)
\int_{\basept(t)}^\zeta
\omega_j^{\mathrm{deg}}(t).
\end{equation}
\end{lemma}

\begin{proof}
By $A$-normalization of $\bm\omega^{\mathrm{deg}}(t)$, for $k\neq j$,
after traversing $A_k^{\mathrm{deg}}(t)$ the function
\[
\zeta
\mapsto
\int_{\basept(t)}^\zeta
\omega_j^{\mathrm{deg}}(t)
\]
changes by
\[
\int_{A_k^{\mathrm{deg}}(t)}
\omega_j^{\mathrm{deg}}(t)
=
2\pi\imunit\,\delta_{kj}
=
0.
\]
Therefore, this function has the same value on the two copies of
$\lambda_k(t)$, and their contributions cancel.
\end{proof}

The following is obtained by summing \cite[Eq.~(3.11)]{HuNorton} over all $k$ and using \cite[Lemma~3.2]{HuNorton}.

\begin{lemma}
\label{lem:estimate_on_seams}
Let $\Omega_\bullet=(\Omega_v)_{v\in\vertices(\tropmcurve)}$ be a stable
differential on $\mcurve(\infty)$, with associated tropical $1$-form
$\Omega_\bullet^\trop$, and let $\Omega(t)$ be the family of holomorphic
differentials on $\mcurve(t)$ provided by
\cref{thm:HuNorton:main1}. For
$
e\in\edges(\tropmcurve)\setminus\edges(T),
$
we have
\[
\Omega(t)
=
\left(
\frac{\Omega_\bullet^\trop(h_e)}{z_{h_e}}
+
O(1)
\right)
dz_{h_e}
\]
uniformly on $A_e^{\mathrm{seam}}(t)$.
\end{lemma}

Define
\[
\omega_j^{\mathrm{deg},\trop}
:=
\begin{cases}
0,
&
j=(v,a),
\\
\omega_e^{\mathrm{deg},\trop},
&
j=e\notin\edges(T).
\end{cases}
\]

\begin{lemma}
\label{lem:uniform_A_cycle_asymptotics}
Let $1\leq j,k\leq g$ with $j\neq k$. Then, uniformly for
$\zeta\in(A_k^{\mathrm{deg}}(t))^+$,
\begin{equation}
\label{eq:uniform_A_cycle_asymptotics_first}
\int_{\basept(t)}^\zeta
\omega_j^{\mathrm{deg}}(t)
=
-
\left(
\int_{\baseptT}^{\tau_k^\trop}
\omega_j^{\mathrm{deg},\trop}
\right)
\log t
+
O(1).
\end{equation}
Moreover,
\[
\int_{(A_k^{\mathrm{deg}}(t))^+}
\left|
\omega_k^{\mathrm{deg}}(t)
\right|
=
O(1).
\]
Consequently,
\[
\frac{1}{2\pi\imunit}
\int_{\zeta\in(A_k^{\mathrm{deg}}(t))^+}
\omega_k^{\mathrm{deg}}(t)(\zeta)
\int_{\basept(t)}^\zeta
\omega_j^{\mathrm{deg}}(t)
=
-
\left(
\int_{\baseptT}^{\tau_k^\trop}
\omega_j^{\mathrm{deg},\trop}
\right)
\log t
+
O(1).
\]
\end{lemma}

\begin{proof}
Choose $\eta_k(t)\in(A_k^{\mathrm{deg}}(t))^+$ as above and choose a
path in $\mathfrak P_{\mcurve(t)}^{\mathrm{deg}}$ from $\basept(t)$ to
$\eta_k(t)$ that is a degeneration-adapted lift of the unique path in
$\tropmcurve_{\mathrm{cut}}$ from $\baseptT$ to
$\tau_k^\trop$. By
\cref{thm:tropical_integration},
\begin{equation}
\label{eq:prof_1}
\int_{\basept(t)}^{\eta_k(t)}
\omega_j^{\mathrm{deg}}(t)
=
-
\left(
\int_{\baseptT}^{\tau_k^\trop}
\omega_j^{\mathrm{deg},\trop}
\right)
\log t
+
O(1).
\end{equation}

Suppose first that $k=(v,a)$. Then
$A_k^{\mathrm{deg}}(t)=A_{v,a}^{\mathrm{ov}}(t)$ is contained in a
component core, so for
$\zeta\in(A_k^{\mathrm{deg}}(t))^+$, the path from $\eta_k(t)$ to
$\zeta$ may be chosen inside the same component core. By
\cref{thm:HuNorton:main1}, after identifying the cores using the
trivialization $\bm\Phi$, the differentials
$\omega_j^{\mathrm{deg}}(t)$ converge uniformly on a compact
neighborhood of $A_k^{\mathrm{deg}}(t)$. Therefore,
\begin{equation}
\label{eq:prof_2}
\int_{\eta_k(t)}^\zeta
\omega_j^{\mathrm{deg}}(t)
=
O(1)
\end{equation}
uniformly in $\zeta$.

Now suppose that $k=e\notin\edges(T)$. The points
$\eta_k(t)$ and $\zeta$ lie on the same copy
$(A_e^{\mathrm{seam}}(t))^+$ of the seam. By
\cref{lem:estimate_on_seams}, the integral of
$\omega_j^{\mathrm{deg}}(t)$ along any subarc of the seam is uniformly
bounded. Hence, again,
\begin{equation}
\int_{\eta_k(t)}^\zeta
\omega_j^{\mathrm{deg}}(t)
=
O(1)
\end{equation}
uniformly in $\zeta$. Combining this with \eqref{eq:prof_1} gives
\eqref{eq:uniform_A_cycle_asymptotics_first}.

For $k=(v,a)$, the same uniform convergence on the component core,
applied to $\omega_k^{\mathrm{deg}}(t)$, gives
\[
\int_{(A_k^{\mathrm{deg}}(t))^+}
\left|
\omega_k^{\mathrm{deg}}(t)
\right|
=
O(1).
\]
For $k=e\notin\edges(T)$, the same estimate follows from
\cref{lem:estimate_on_seams}. This proves the first two assertions.

By the first two assertions, there exist constants $M_1,M_2>0$ such
that, for all sufficiently large $t$,
\[
\frac{1}{2\pi}
\int_{(A_k^{\mathrm{deg}}(t))^+}
\left|
\omega_k^{\mathrm{deg}}(t)
\right|
\leq
M_1,
\qquad
\sup_{\zeta\in(A_k^{\mathrm{deg}}(t))^+}
\left|
\int_{\basept(t)}^\zeta
\omega_j^{\mathrm{deg}}(t)
-
c\log t
\right|
\leq
M_2,
\]
where
$
c
:=
-
\int_{\baseptT}^{\tau_k^\trop}
\omega_j^{\mathrm{deg},\trop}.
$
Using the normalization
\[
\int_{(A_k^{\mathrm{deg}}(t))^+}
\omega_k^{\mathrm{deg}}(t)
=
2\pi\imunit,
\]
we obtain
\begin{align*}
&
\left|
\frac{1}{2\pi\imunit}
\int_{\zeta\in(A_k^{\mathrm{deg}}(t))^+}
\omega_k^{\mathrm{deg}}(t)(\zeta)
\int_{\basept(t)}^\zeta
\omega_j^{\mathrm{deg}}(t)
-
c\log t
\right|
\\
&\qquad=
\left|
\frac{1}{2\pi\imunit}
\int_{\zeta\in(A_k^{\mathrm{deg}}(t))^+}
\omega_k^{\mathrm{deg}}(t)(\zeta)
\left(
\int_{\basept(t)}^\zeta
\omega_j^{\mathrm{deg}}(t)
-
c\log t
\right)
\right|
\\
&\qquad\leq
\frac{1}{2\pi}
\int_{(A_k^{\mathrm{deg}}(t))^+}
\left|
\omega_k^{\mathrm{deg}}(t)
\right|
\sup_{\zeta\in(A_k^{\mathrm{deg}}(t))^+}
\left|
\int_{\basept(t)}^\zeta
\omega_j^{\mathrm{deg}}(t)
-
c\log t
\right|
\\
&\qquad\leq
M_1M_2.
\end{align*}
Therefore,
\[
\frac{1}{2\pi\imunit}
\int_{\zeta\in(A_k^{\mathrm{deg}}(t))^+}
\omega_k^{\mathrm{deg}}(t)(\zeta)
\int_{\basept(t)}^\zeta
\omega_j^{\mathrm{deg}}(t)
=
-
\left(
\int_{\baseptT}^{\tau_k^\trop}
\omega_j^{\mathrm{deg},\trop}
\right)
\log t
+
O(1).
\]
\end{proof}

\begin{proof}[Proof of \cref{prop:riemann_constants_trop}]
Suppose first that
$
j=(v,a).
$
The stable differential associated to
$\omega_{v,a}^{\mathrm{deg}}(t)$ has zero associated tropical $1$-form.
Hence, by \cref{lem:uniform_A_cycle_asymptotics}, for every
$k\neq(v,a)$,
\[
\frac{1}{2\pi\imunit}
\int_{\zeta\in(A_k^{\mathrm{deg}}(t))^+}
\omega_k^{\mathrm{deg}}(t)(\zeta)
\int_{\basept(t)}^\zeta
\omega_{v,a}^{\mathrm{deg}}(t)
=
O(1).
\]
Moreover, by \cref{cor:periodtropical},
$
\Pi_{(v,a),(v,a)}^{\mathrm{deg}}(t)
=
O(1),
$
and therefore,
\[
\RC_{v,a}^{\mathrm{deg}}(t)
=
O(1).
\]
Now let
$
j=e\notin\edges(T).
$
By \cref{cor:periodtropical},
\[
\Pi_{ee}^{\mathrm{deg}}(t)
=
-
\Pi_{ee}^{\mathrm{deg},\trop}
\log t
+
O(1).
\]

For a component index $k=(v,a)$,
\cref{lem:uniform_A_cycle_asymptotics} gives
\[
\frac{1}{2\pi\imunit}
\int_{\zeta\in(A_{v,a}^{\mathrm{ov}}(t))^+}
\omega_{v,a}^{\mathrm{deg}}(t)(\zeta)
\int_{\basept(t)}^\zeta
\omega_e^{\mathrm{deg}}(t)
=
-
\left(
\int_{\baseptT}^{v}
\omega_e^{\mathrm{deg},\trop}
\right)
\log t
+
O(1).
\]
Summing over $1\leq a\leq g_v$ and then over
$v\in\vertices(\tropmcurve)$ gives
\[
-
\sum_{v\in\vertices(\tropmcurve)}
g_v
\left(
\int_{\baseptT}^{v}
\omega_e^{\mathrm{deg},\trop}
\right)
\log t
+
O(1).
\]

Next suppose that
$
k=e'\neq e$ with $e'\notin\edges(T).
$
Again by \cref{lem:uniform_A_cycle_asymptotics},
\[
\frac{1}{2\pi\imunit}
\int_{\zeta\in(A_{e'}^{\mathrm{seam}}(t))^+}
\omega_{e'}^{\mathrm{deg}}(t)(\zeta)
\int_{\basept(t)}^\zeta
\omega_e^{\mathrm{deg}}(t)
=
-
\left(
\int_{\baseptT}^{s_{e'}^+}
\omega_e^{\mathrm{deg},\trop}
\right)
\log t
+
O(1).
\]
Summing over $e'\neq e$ and substituting these asymptotics into
\eqref{eq:Fay_RC_deg}, we obtain
\[
\RC_e^{\mathrm{deg}}(t)
=
\left(
\frac12\Pi_{ee}^{\mathrm{deg},\trop}
-
\sum_{v\in\vertices(\tropmcurve)}
g_v
\int_{\baseptT}^{v}
\omega_e^{\mathrm{deg},\trop}
-
\sum_{\substack{
e'\neq e\\
e'\notin\edges(T)
}}
\int_{\baseptT}^{s_{e'}^+}
\omega_e^{\mathrm{deg},\trop}
\right)
\log t
+
O(1).
\]
Therefore,
\[
\vecRC^{\mathrm{deg}}(t)
=
-
\begin{pmatrix}
0\\
\vecRC^{\mathrm{deg},\trop}
\end{pmatrix}
\log t
+
O(1).
\]
\end{proof}

\subsection{The theta function}
\label{sec:break_divisor}

Following Mikhalkin--Zharkov~\cite[Section~5.2]{MZ}, but using the
max-plus convention, we define the \emph{tropical theta function}
associated to the tropical period matrix $\Pi^{\mathrm{deg},\trop}$ by
\[
\theta^\trop(\cdot \mid \Pi^{\mathrm{deg},\trop}):\R^{g^\trop}\to\R,\qquad
\theta^\trop
( \ptCg^\trop \mid\Pi^{\mathrm{deg},\trop} )
:=
\max_{\bm{l}\in\Z^{g^\trop}}
\left(
-\frac12
\<\bm{l},\Pi^{\mathrm{deg},\trop}\bm{l}\>
-
\<\bm{l},\ptCg^\trop \>
\right).
\]

\begin{proposition}
\label{lem:theta_tropicalization}
Suppose that $\ptCg(t)\in\R^g$ and that
\begin{equation} \label{eq:asymp_z_t}
\ptCg(t)
=
-
\begin{pmatrix}
0\\
\ptCg^\trop
\end{pmatrix}
\log t
+
O(1)
\end{equation}
for some $\ptCg^\trop\in\R^{g^\trop}$. Then,
\[
\log
\theta
(
\ptCg(t)\mid\Pi^{\mathrm{deg}}(t)
)
=
\theta^\trop
(
\ptCg^\trop\mid\Pi^{\mathrm{deg},\trop}
)
\log t
+
O(1).
\]
\end{proposition}

\begin{proof}
Write
\[
\ptZ
=
\begin{pmatrix}
\bm{k}\\
\bm{l}
\end{pmatrix}
\in
\Z^{g^\mathrm{comp}}\times\Z^{g^\trop}.
\]
By definition,
\begin{equation}\label{eq:theta_defn_trop_proof}
\theta
\left(
\ptCg(t)\mid\Pi^{\mathrm{deg}}(t)
\right)
=
\sum_{\ptZ\in\Z^g}
\exp\left(
\frac12
\<\ptZ,\Pi^{\mathrm{deg}}(t)\ptZ\>
+
\<\ptZ,\ptCg(t)\>
\right).
\end{equation}
Using~\eqref{eq:period_asympotics_ov} and~\eqref{eq:asymp_z_t}, we obtain
\begin{align*}
\frac12
\<\ptZ,\Pi^{\mathrm{deg}}(t)\ptZ\>
+
\<\ptZ,\ptCg(t)\>
&=
\left(
-\frac12
\<\bm{l},
\Pi^{\mathrm{deg},\trop}\bm{l}\>
-
\<\bm{l},\ptCg^\trop\>
\right)
\log t
+
O(1).
\end{align*}
Since $\vecptjac(t)$ and $\Pi^{\mathrm{deg}}(t)$ are real, all summands in~\eqref{eq:theta_defn_trop_proof} are positive.
In particular, there is no cancellation of leading terms. Therefore,
\begin{align*}
\log
\theta
\left(
\ptCg(t)\mid\Pi^{\mathrm{deg}}(t)
\right)
&=
\max_{\bm{l}\in\Z^{g^\trop}}
\left(
-\frac12
\<\bm{l},
\Pi^{\mathrm{deg},\trop}\bm{l}\>
-
\<\bm{l},\ptCg^\trop\>
\right)
\log t
+
O(1)
\\
&=
\theta^\trop
\left(
\ptCg^\trop\mid\Pi^{\mathrm{deg},\trop}
\right)
\log t
+
O(1).
\end{align*}
\end{proof}

\begin{figure}
\centering

\begin{tikzpicture}[
    x={(0.50cm,-0.29cm)},
    y={(-0.29cm,0.50cm)},
    line cap=round,
    line join=round,
    >=Stealth,
    every node/.style={font=\scriptsize}
]

\tikzset{
    thetadiv/.style={
        red!75!black,
        line width=1.15pt
    },
    abelimage/.style={
        black,
        line width=1.2pt
    },
    jacboundary/.style={
        black,
        line width=.9pt
    },
    intersection/.style={
        draw=brown,
        circle,
        fill=brown,
        minimum size=6pt,
        inner sep=0pt
    },
    thetavertex/.style={
        draw=red!75!black,
        circle,
        fill=red!75!black,
        minimum size=5pt,
        inner sep=0pt
    }
}


\coordinate (O)   at (-18,-18);
\coordinate (P1)  at (3,0);
\coordinate (P2)  at (0,3);
\coordinate (P12) at (21,21);

\fill[black!5]
    (O)--(P1)--(P12)--(P2)--cycle;

\draw[jacboundary]
    (O)--(P1)--(P12)--(P2)--cycle;


\coordinate (TL) at ({-12/7},1);
\coordinate (TB) at (1,{-12/7});

\coordinate (T1) at (1,1);
\coordinate (T2) at (19,19);

\coordinate (TT) at (19,{135/7});
\coordinate (TR) at ({135/7},19);

\draw[thetadiv] (TL)--(T1);
\draw[thetadiv] (TB)--(T1);
\draw[thetadiv] (T1)--(T2);
\draw[thetadiv] (T2)--(TT);
\draw[thetadiv] (T2)--(TR);


\node[thetavertex] at (T1) {};
\node[thetavertex] at (T2) {};

\node[
    text=red!75!black,
    above =3pt,
    font=\tiny
]
at (T1)
{$(1,1)$};

\node[
    text=red!75!black,
     left=3pt,
    font=\tiny
]
at (T2)
{$(19,19)$};


\node[
    text=red!75!black,
    left=3pt,
    font=\tiny
]
at (TL)
{$\left(-\frac{12}{7},1\right)$};

\node[
    text=red!75!black,
    below=3pt,
    font=\tiny
]
at (TB)
{$\left(1,-\frac{12}{7}\right)$};

\node[
    text=red!75!black,
    above=3pt,
    font=\tiny
]
at (TT)
{$\left(19,\frac{135}{7}\right)$};

\node[
    text=red!75!black,
    right=3pt,
    font=\tiny
]
at (TR)
{$\left(\frac{135}{7},19\right)$};


\coordinate (Vone)   at (0,0);
\coordinate (Vthree) at (-13,-13);
\coordinate (Vfour)  at (-5,-5);


\draw[abelimage]
    (O)
    -- node[midway,below right] {$e_3$}
    (Vthree);


\draw[abelimage]
    (Vthree)
    -- node[midway,below right] {$e_2$}
    (Vfour);


\draw[abelimage]
    (Vfour)
    -- node[midway,below right] {$e_1$}
    (Vone);


\draw[abelimage]
    (Vone)
    -- node[pos=.78,below left] {$e_4$}
    (P1);


\draw[abelimage]
    (Vone)
    -- node[pos=.78,right] {$e_5$}
    (P2);


\node[
    bvert,
    label={[font=\scriptsize]left:{$v_1=(0,0)$}}
]
at (Vone) {};

\node[
    bvert,
    label={[font=\scriptsize]left:{$v_3=(-13,-13)$}}
]
at (Vthree) {};

\node[
    bvert,
    label={[font=\scriptsize]left:{$v_4=(-5,-5)$}}
]
at (Vfour) {};


\node[
    bvert,
    label={[font=\scriptsize]below left:{$v_2=(-18,-18)$}}
]
at (O) {};

\node[
    bvert,
    label={[font=\scriptsize]below right:{$v_2=(3,0)$}}
]
at (P1) {};

\node[
    bvert,
    label={[font=\scriptsize]above left:{$v_2=(0,3)$}}
]
at (P2) {};


\node[intersection] at (1,0) {};
\node[intersection] at (0,1) {};

\node[
    right=5pt,
    text=brown,
    font=\scriptsize
]
at (1,0)
{$(1,0)$};

\node[
    left=5pt,
    text=brown,
    font=\scriptsize
]
at (0,1)
{$(0,1)$};

\end{tikzpicture}

\caption{
Tropical Riemann's theorem for the abstract tropical curve in
\cref{fig:embedded-abstract-dual-subdivision}(B).
We take $\baseptT=v_1$ and
$\vecptjac^\trop=(1,1)$.
The parallelogram is a translated fundamental domain for
$\jac(\tropmcurve)$ whose side vectors are the columns of
$\Pi^{\mathrm{deg},\trop}$. For clarity, the picture has been linearly rescaled. The black graph is the Abel image
$\mu_{v_1}^{\mathrm{deg},\trop}(\tropmcurve)$
and the red graph is the translated tropical theta divisor
$\Theta^\trop+\vecptjac^\trop-\vecRC^{\mathrm{deg},\trop}$.
Their intersection consists of the two brown points $(1,0)$ and $(0,1)$.
}
\label{fig:tropical-riemann-non-smooth}

\end{figure}
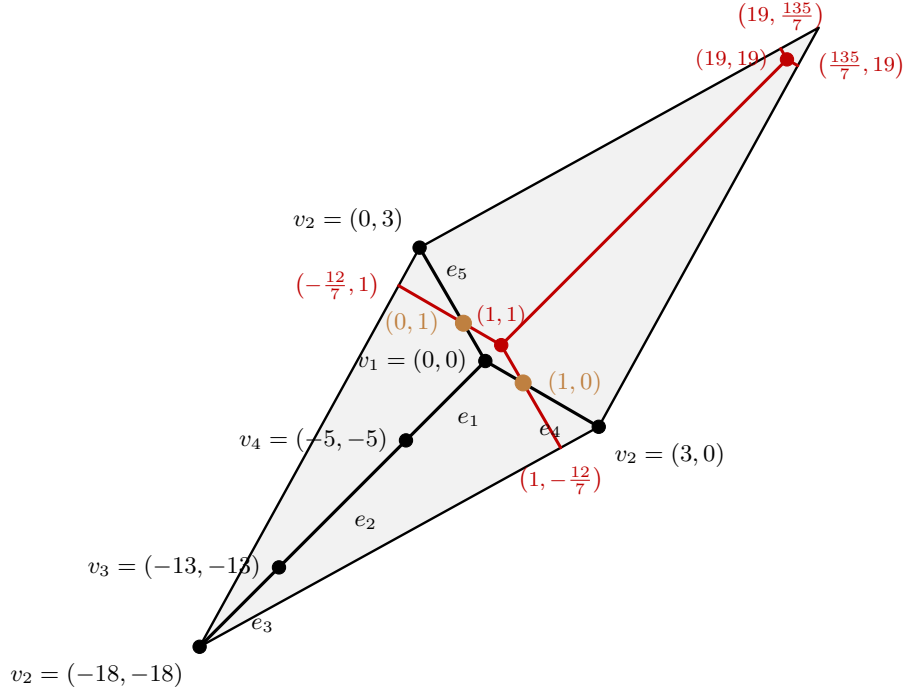

\begin{example}\label{ex:MZequalsours}

Let us verify for the abstract tropical curve in
\cref{fig:embedded-abstract-dual-subdivision}(B) that
$\vecRC^{\mathrm{deg},\trop}$ agrees with the tropical vector of
Riemann constants appearing in the following tropical version of
Riemann's theorem due to
Mikhalkin--Zharkov~\cite[Theorem~6.5]{MZ}, formulated as in \cref{remark:theta_divisor}.
Let
\[
\Theta^\trop
:=
\{
\ptR\in\jac(\tropmcurve)
\mid
\theta^\trop(\ptR\mid\Pi^{\mathrm{deg},\trop})
\text{ achieves its maximum at least twice}
\}
\]
be the tropical theta divisor. Then, there is a unique constant
$[\vecRC^{\mathrm{MZ}}] \in \jac(\tropmcurve)$ such that for any
$[\vecptjac^\trop]\in\jac(\tropmcurve)$,
\[
\Divtr
=
\sum_{v\in\vertices(\tropmcurve)}
g_v\cdot v
+
\mu^{\mathrm{deg},\trop}_{\baseptT}(\tropmcurve)
\cap_{\mathrm{stable}}
\left(
\Theta^\trop
+
[\vecptjac^\trop]
-
[\vecRC^{\mathrm{MZ}}]
\right)
\]
is a divisor such that
\[
\mu^{\mathrm{deg},\trop}_{\baseptT}(\Divtr)
=
[\vecptjac^\trop].
\]
Here, stable intersection means the limit of intersections with generic
perturbations. We check that
\[
[\vecRC^{\mathrm{deg},\trop}]
=
[\vecRC^{\mathrm{MZ}}].
\]

Since
\[
\Pi^{\mathrm{deg},\trop}
=
\begin{pmatrix}
21&18\\
18&21
\end{pmatrix},
\]
we represent the tropical Jacobian by drawing a translated fundamental
parallelogram whose sides are given by the columns of
$\Pi^{\mathrm{deg},\trop}$.
Take
$
\vecptjac^\trop=(1,1)
$
and recall from
\cref{ex:riemann_constants_non_smooth} that the fixed basepoint is
$\baseptT=v_1$ and
$
\vecRC^{\mathrm{deg},\trop}
=
\begin{pmatrix}
-{57}/{2}\\
-{57}/{2}
\end{pmatrix}.
$
The translate appearing in tropical Riemann's theorem is
\[
\Theta^\trop
+
\vecptjac^\trop
-
\vecRC^{\mathrm{deg},\trop},
\]
which is the red graph in
\cref{fig:tropical-riemann-non-smooth}.
The Abel image
$\mu_{v_1}^{\mathrm{deg},\trop}(\tropmcurve)$
is the black graph in
\cref{fig:tropical-riemann-non-smooth}, from which we get that the
intersection, which is transverse and thus agrees with the stable
intersection, is
\[
\mu_{v_1}^{\mathrm{deg},\trop}(\tropmcurve)
\cap
\left(
\Theta^\trop
+
\vecptjac^\trop
-
\vecRC^{\mathrm{deg},\trop}
\right)
=
\{
(1,0),(0,1)
\}
\]
Since all vertices have genus zero, $\Divtr$ is the preimage of these two points. Finally,
\[
\mu^{\mathrm{deg},\trop}_{v_1}(\Divtr)
=
(1,0)+(0,1)
=
(1,1)
=
\vecptjac^\trop.
\]
Therefore,
\[
[\vecRC^{\mathrm{deg},\trop}]
=
[\vecRC^{\mathrm{MZ}}].
\]

\end{example}

\section{Tropicalization of Fock's inverse spectral transform}
\label{sec:tropical_fock_weights}

\subsection{Tropical dimer model}

Let
\[
\wt^\trop:\edges(\Gtor)\rightarrow\Q
\]
be a \emph{tropical edge weight}.  Two tropical edge weights are
\emph{tropically gauge equivalent} if there is a function
\[
f:\blackvertices(\Gtor)\sqcup\whitevertices(\Gtor)\rightarrow\Q
\]
such that
\[
\wt_2^\trop(\bb\bw)
=
\wt_1^\trop(\bb\bw)-f(\bb)+f(\bw)
\qquad
\text{for every }\bb\bw\in\edges(\Gtor).
\]
We denote the tropical gauge-equivalence class of $\wt^\trop$ by
$[\wt^\trop]$.

If
$
\gamma
=
(\bb_1\xrightarrow{\be_1}\bw_1
\xrightarrow{\be_2}\bb_2
\xrightarrow{\be_3}\cdots
\xrightarrow{\be_{2k-1}}\bw_k
\xrightarrow{\be_{2k}}\bb_1)
$
is a cycle in $\Gtor$, define its \emph{tropical cycle weight} by
\[
[\wt^\trop](\gamma)
:=
\sum_{j=1}^k
\left(
\wt^\trop(\be_{2j-1})
-
\wt^\trop(\be_{2j})
\right).
\]
This depends only on the tropical gauge-equivalence class.

\subsection{The tropical inverse spectral transform}

We now define a tropical version of Fock's inverse spectral
transform. A \emph{tropical spectral datum associated to $\Gtor$} is a triple
$
(
\tropmcurve,\divisor^\trop,\nu^\trop
),
$
where:
\begin{enumerate}
\item
$(\tropmcurve,\ell,\bm g)$ is the abstract tropical curve
associated to an embedded tropical spectral curve
$\tropmcurve_{\mathrm{emb}}$ with Newton polygon $N$.

\item
$\divisor^\trop$ is a tropical effective
divisor of degree $g$.

\item
$\nu^\trop$ is the tropical angle map, assigning to each zig-zag path
$\alpha\in\zzG$ a vertex
$\alpha^\trop\in \vertices(\tropmcurve)$ such that $\alpha^\trop$ is the endpoint of an unbounded edge dual to the side $\Eside(\alpha)$ of $N$.
\end{enumerate}

Fix a spanning tree $T\subset\tropmcurve$ and a basepoint
$\baseptT\in\tropmcurve$, and use the corresponding
degeneration-adapted tropical period matrix
$\Pi^{\mathrm{deg},\trop}$, Abel map
$\mu_{\baseptT}^{\mathrm{deg},\trop}$, and vector of Riemann constants
$\vecRC^{\mathrm{deg},\trop}$. Extending $\nu^\trop$ linearly to divisors supported on $\zzG$ gives
the \emph{tropical discrete Abel map}
$
\dam^\trop.
$

For a wedge $\mathfrak w$ with associated
$\alpha,\beta_-,\beta_+$ and adjacent faces
$\bface_-,\bface_+$, define its \emph{tropical Fock wedge weight} by
\begin{equation}
\label{eq:tropical_fock_weight}
\begin{aligned}
\wt_{\mathrm{Fock}}^\trop(\mathfrak w)
:={}&
\theta^\trop\left(
\mu_{\baseptT}^{\mathrm{deg},\trop}
    (\dam^\trop(\bface_-))
-\mu_{\baseptT}^{\mathrm{deg},\trop}
    (\divisor^\trop)
+\vecRC^{\mathrm{deg},\trop}
\middle|
\Pi^{\mathrm{deg},\trop}
\right)
\\
&-
\theta^\trop\left(
\mu_{\baseptT}^{\mathrm{deg},\trop}
    (\dam^\trop(\bface_+))
-\mu_{\baseptT}^{\mathrm{deg},\trop}
    (\divisor^\trop)
+\vecRC^{\mathrm{deg},\trop}
\middle|
\Pi^{\mathrm{deg},\trop}
\right)
\\
&-
\int_{\baseptT}^{\alpha^\trop}
\omega_{\beta_+^\trop-\beta_-^\trop}^\trop.
\end{aligned}
\end{equation}
The integrals from $\baseptT$ to a point
of $\tropmcurve$ are taken along the unique path in the cut tropical $\tropmcurve_{\mathrm{cut}}$ joining the two
points.

If
$
\gamma
=
\mathfrak w_1+\cdots+\mathfrak w_r,
$
define the corresponding \emph{tropical Fock cycle weight} by
\begin{equation}
\label{eq:tropical_face_weight}
[\wt_{\mathrm{Fock}}^\trop](\gamma)
:=
\sum_{j=1}^r
\wt_{\mathrm{Fock}}^\trop(\mathfrak w_j).
\end{equation}

Thus, we have defined a map
\begin{align*}
 (\lambda^{-1}_{\mathrm{Fock}})^\trop:  \{\text{tropical spectral data associated to $\Gtor$}\}& \to \{\text{tropical edge weights on $\Gtor$}\}/\text{gauge}\\
((\tropmcurve,\ell,\bm g),\divisor^\trop,\nu^\trop)&\mapsto[\wt_{\mathrm{Fock}}^\trop]
\end{align*}

\subsection{Main theorem}

Let $[\wt](t)$ be $\Puispos$-valued edge weights modulo gauge
equivalence satisfying Assumption~\ref{ass:generic-leading-coefficients}, and let
\[
[\wt^\trop]:=\val([\wt](t))
\]
be its tropicalization. For $t\gg1$, apply the spectral transform:
\begin{equation}
\label{eq:spectral_transform_family}
\lambda([\wt](t))
=
(
\mcurve(t),\divisor(t),\nu(t)
).
\end{equation}

Recall that the family of spectral curves $\mcurve(t)$ degenerates to
a nodal curve $\mcurve(\infty)$, and that the corresponding embedded
tropical spectral curve $\tropmcurve_{\mathrm{emb}}$ determines an
abstract tropical curve
$
(\tropmcurve,\ell,\bm g).
$
Fix a spanning tree $T\subset\tropmcurve$ and a basepoint
$\basept(t)\in\mcurve(t)$ lying in the core of a component $\mcurve_v$,
with tropical limit $\baseptT=v$. The standard divisor
\[
\divisor(t)=\sum_{j=1}^g\divptgeom_j(t)
\]
tropicalizes to an embedded tropical divisor
\[
\divisor_{\mathrm{emb}}^\trop
:=
\sum_{j=1}^g \val(\divptgeom_j(t))
\]
on $\tropmcurve_{\mathrm{emb}}$. For each $j$, the component core or
plumbing collar containing $\divptgeom_j(t)$ determines the
corresponding vertex or edge of $\tropmcurve$. Thus, the degenerating
family $\divisor(t)$ determines a lift
\[
\divisor^\trop\in\sym^g(\tropmcurve)
\]
satisfying
$
\pi(\divisor^\trop)=\divisor_{\mathrm{emb}}^\trop.
$

For $\alpha\in\zzG$, write
$
\alpha(t):=\nu(t)(\alpha).
$
The tentacle of $\openspectralcurve(t)$ corresponding to
$\alpha(t)$ tropicalizes to the unbounded edge of
$\openspectralcurve^\trop_{\mathrm{emb}}$ contained in the line
\[
\langle[\alpha],\ptR\rangle+\Ctr_\alpha=0,
\]
where
\[
\Ctr_\alpha:=\val(\casimir_\alpha(t))
\]
is the \emph{tropical Casimir}. Let $\alpha^\trop$ be the vertex of
$\tropmcurve_{\mathrm{emb}}$ incident to this unbounded edge, which we
identify with the corresponding vertex of $\tropmcurve$. This defines
the \emph{tropical angle map}
\[
\nu^\trop:\zzG\rightarrow\vertices(\tropmcurve),
\qquad
\alpha\mapsto\alpha^\trop.
\]

We thus obtain a tropical spectral datum
$(
(\tropmcurve,\ell,\bm g),
\divisor^\trop,
\nu^\trop)
$ associated to the family
\eqref{eq:spectral_transform_family}.

\begin{theorem}
\label{thm:fock_weight_tropicalization}
Let $[\wt](t)$ be $\Puispos$-valued edge weights modulo gauge
equivalence satisfying Assumption~\ref{ass:generic-leading-coefficients}, and let
$
\lambda([\wt](t))
=
(
\mcurve(t),\divisor(t),\nu(t)
).
$
Let
$
(
(\tropmcurve,\ell,\bm g),
\divisor^\trop,
\nu^\trop
)
$
be the corresponding tropical spectral datum. Then,
\[
(\lambda^{-1}_{\mathrm{Fock}})^\trop
(
(\tropmcurve,\ell,\bm g),
\divisor^\trop,
\nu^\trop
)
=
[\wt^\trop].
\]
In other words, tropicalization commutes with Fock's inverse spectral
transform.
\end{theorem}

\begin{proof}

We first rewrite Fock's inverse spectral transform in
degeneration-adapted coordinates.  Recall from \cref{sec:degen_adapted_basis} that
\[
\bm B^{\mathrm{deg}}(t)
=
M_T\bm B^{\mathrm{ov}}(t), \qquad M_T \in \operatorname{GL}_{g}(\Z).
\]
We may choose the oval-adapted $A$-cycles to be $\bm A^{\mathrm{ov}}(t)
=
M_T^\top\bm A^{\mathrm{deg}}(t)$. 

Standard change-of-basis formulas (see, for example, \cite{Fay}) give that the corresponding normalized holomorphic differentials and period
matrices satisfy
\[
\bm\omega^{\mathrm{ov}}(t)
=
M_T^{-1}\bm\omega^{\mathrm{deg}}(t),
\qquad
\Pi^{\mathrm{ov}}(t)
=
M_T^{-1}
\Pi^{\mathrm{deg}}(t)
(M_T^{-1})^{\top},
\]
and that
\[
\mu_{\basept(t)}^{\mathrm{ov}}
=
M_T^{-1}
\mu_{\basept(t)}^{\mathrm{deg}},
\qquad
\vecptjac^{\mathrm{ov}}(t)
=
M_T^{-1}
\vecptjac^{\mathrm{deg}}(t),
\qquad
\vecRC^{\mathrm{ov}}(t)
=
M_T^{-1}
\vecRC^{\mathrm{deg}}(t).
\]
Since $M_T\in\operatorname{GL}_g(\Z)$, reindexing the series expansion
of the theta function gives
\[
\theta
\left(
M_T^{-1}\ptCg
\middle|
M_T^{-1}\Pi^{\mathrm{deg}}(M_T^{-1})^\top
\right)
=
\theta
\left(
\ptCg
\middle|
\Pi^{\mathrm{deg}}
\right).
\]
Moreover, the normalized differential of the third kind is unchanged:
vanishing of its periods along $\bm A^{\mathrm{ov}}(t)$ is equivalent to vanishing of its periods along
$\bm A^{\mathrm{deg}}(t)$.

Let $\mathfrak w$ be a wedge with associated zig-zag paths
$\alpha,\beta_-,\beta_+$ and adjacent faces
$\bface_-,\bface_+$, as in \cref{sec:fock_weights}.  Fock's wedge
weight therefore takes the same form in degeneration-adapted
coordinates:
\begin{equation}
\label{eq:fock_wedge_deg}
\begin{aligned}
\wt_{\mathrm{Fock}}(t)(\mathfrak w)
={}&
-\frac{
\theta\left(
\mu_{\basept(t)}^{\mathrm{deg}}(\dam(\bface_-))
-\mu^{\mathrm{deg}}_{\basept(t)}(\divisor(t))
+\vecRC^{\mathrm{deg}}(t)
\middle|
\Pi^{\mathrm{deg}}(t)
\right)
}{
\theta\left(
\mu_{\basept(t)}^{\mathrm{deg}}(\dam(\bface_+))
-\mu^{\mathrm{deg}}_{\basept(t)}(\divisor(t))
+\vecRC^{\mathrm{deg}}(t)
\middle|
\Pi^{\mathrm{deg}}(t)
\right)
}
\\
&\qquad\times
\exp\left(
\int_{\basept(t)}^{\alpha(t)}
\omega^{\mathrm{deg}}_{\beta_+(t)-\beta_-(t)}(t)
\right).
\end{aligned}
\end{equation}
The integrals from $\basept(t)$ to any point $\ptSig(t)$ are taken along degeneration-adapted lifts of the unique path in $\tropmcurve_{\mathrm{cut}}$ from $\baseptT$ to $\ptSig^\trop$; in particular, these paths tropicalize to the corresponding paths in \eqref{eq:tropical_fock_weight}.

By \cref{cor:abelmap}, 
\begin{align*}
\mu_{\basept(t)}^{\mathrm{deg}}(\divisor(t))
&=
-
\begin{pmatrix}
0\\
\mu_{\baseptT}^{\mathrm{deg},\trop}(\divisor^\trop)
\end{pmatrix}
\log t
+
O(1),\\
\mu_{\basept(t)}^{\mathrm{deg}}(\dam(\bface))
&=
-
\begin{pmatrix}
0\\
\mu_{\baseptT}^{\mathrm{deg},\trop}
    (\dam^\trop(\bface))
\end{pmatrix}
\log t
+
O(1).
\end{align*}
Together with \cref{prop:riemann_constants_trop}, we obtain
\begin{equation}
\label{eq:fock_theta_argument_asymptotic}
\begin{aligned}
&\mu_{\basept(t)}^{\mathrm{deg}}(\dam(\bface))
-\mu_{\basept(t)}^{\mathrm{deg}}(\divisor(t))
+\vecRC^{\mathrm{deg}}(t)
\\
&\qquad=
-
\begin{pmatrix}
0\\
\mu_{\baseptT}^{\mathrm{deg},\trop}
    (\dam^\trop(\bface))
-\mu_{\baseptT}^{\mathrm{deg},\trop}
    (\divisor^\trop)
+\vecRC^{\mathrm{deg},\trop}
\end{pmatrix}
\log t
+
O(1).
\end{aligned}
\end{equation}

By \cref{prop:omega_and_delta_pos},
\[
\mu_{\basept(t)}^{\mathrm{ov}}(\dam(\bface))
\in\R^g,
\qquad
\mu_{\basept(t)}^{\mathrm{ov}}(\divisor(t))
\in
\pi\imunit\onebf+\R^g,
\qquad
\vecRC^{\mathrm{ov}}(t)
\in
\pi\imunit\onebf+\R^g,
\]and therefore,
\[
\mu_{\basept(t)}^{\mathrm{ov}}(\dam(\bface))
-
\mu_{\basept(t)}^{\mathrm{ov}}(\divisor(t))
+
\vecRC^{\mathrm{ov}}(t)
\in\R^g.
\]
Since the corresponding degeneration-adapted theta argument is obtained
by multiplication by $M_T\in\operatorname{GL}_g(\Z)$, it also lies in
$\R^g$. Therefore, \cref{lem:theta_tropicalization} and
\eqref{eq:fock_theta_argument_asymptotic} give
\begin{equation}
\label{eq:theta_fock_asymptotic}
\begin{aligned}
&\log\theta\left(
\mu_{\basept(t)}^{\mathrm{deg}}(\dam(\bface))
-\mu_{\basept(t)}^{\mathrm{deg}}(\divisor(t))
+\vecRC^{\mathrm{deg}}(t)
\middle|
\Pi^{\mathrm{deg}}(t)
\right)
\\
&\qquad=
\theta^\trop\left(
\mu_{\baseptT}^{\mathrm{deg},\trop}
    (\dam^\trop(\bface))
-\mu_{\baseptT}^{\mathrm{deg},\trop}
    (\divisor^\trop)
+\vecRC^{\mathrm{deg},\trop}
\middle|
\Pi^{\mathrm{deg},\trop}
\right)
\log t
+
O(1).
\end{aligned}
\end{equation}

On the other hand, \cref{prop:trop_differential_third_kind} gives
\begin{equation}
\label{eq:third_kind_fock_asymptotic}
\int_{\basept(t)}^{\alpha(t)}
\omega_{\beta_+(t)-\beta_-(t)}^{\mathrm{deg}}(t)
=
-
\left(
\int_{\baseptT}^{\alpha^\trop}
\omega_{\beta_+^\trop-\beta_-^\trop}^\trop
\right)
\log t
+
O(1),
\end{equation}
where the tropical integral is taken along the unique path in $\tropmcurve_{\mathrm{cut}}$ from
$\baseptT$ to $\alpha^\trop$.

Taking the logarithm of the absolute value of
\eqref{eq:fock_wedge_deg} and substituting
\eqref{eq:theta_fock_asymptotic} and
\eqref{eq:third_kind_fock_asymptotic}, we obtain
\begin{equation}
\label{eq:wedge_asymp}
\log
\left|
\wt_{\mathrm{Fock}}(t)(\mathfrak w)
\right|
=
\wt_{\mathrm{Fock}}^\trop(\mathfrak w)\log t
+
O(1).
\end{equation}

Now let
$
\gamma
=
\mathfrak w_1+\cdots+\mathfrak w_r
$
be a cycle in $\Gtor$. By \cref{thm:spectral_transform_birational},
\[
[\wt](t)(\gamma)
=
[\wt_{\mathrm{Fock}}](t)(\gamma)
=
[\kappa](\gamma)
\prod_{j=1}^r
\wt_{\mathrm{Fock}}(t)(\mathfrak w_j).
\]
Since $|[\kappa](\gamma)|=1$, taking logarithms of absolute values and
using \eqref{eq:wedge_asymp} gives
\[
\log|[\wt](t)(\gamma)|
=
\left(
\sum_{j=1}^r
\wt_{\mathrm{Fock}}^\trop(\mathfrak w_j)
\right)
\log t
+
O(1).
\]
Since the weights are positive for $t \gg 1$, the aboslute values may be dropped, and therefore
\[
[\wt^\trop]
=
[\wt_{\mathrm{Fock}}^\trop],
\]
which proves the theorem.
\end{proof}

\begin{example}
We verify \cref{thm:fock_weight_tropicalization} for our running example. Recall the positive Puiseux weights
\[
X_1(t)=t, \qquad X_2(t)=t,\qquad
X_3(t)=1, \qquad \rho_z(t)=t,
\qquad\rho_w(t)=1
\] 
from \cref{ex:running_graph_tropical_spectral_curve} with tropical weights 
\[
X_1^\trop=1, \qquad X_2^\trop=1,
\qquad
X_3^\trop=0,
\qquad
\rho_z^\trop=1,
\qquad
\rho_w^\trop=0.
\]
The tropical spectral curve is shown in \cref{fig:running-example-tropical-spectral-curve}. From \cref{ex:running_graph_spectral_transform_genus_one}, the standard divisor is
\[
\divisor(t)
=
\left(
-\frac{X_1X_2X_3}{\rho_z},
\frac{1+X_1X_2X_3}{X_1\rho_w}
\right)
=
(-t,t+t^{-1}),
\]
which tropicalizes to the point $D^\trop=(1,1)$. Next, we compute the tropical angle map. Again from
\cref{ex:running_graph_spectral_transform_genus_one}, we have
\[
C_{\alpha_1}=\rho_z=t,
\qquad
C_{\alpha_2}
=
\frac{\rho_z}{X_1X_2}
=
t^{-1}.
\]
Since
$
[\alpha_1]=[\alpha_2]=(-1,0),
$
the lines containing the unbounded edges are
\[
x+C_{\alpha_i}^\trop=0.
\]
Since 
$
C_{\alpha_1}^\trop=1,
C_{\alpha_2}^\trop=-1,
$
we get
\[
\alpha_1^\trop=(-1,0),
\qquad
\alpha_2^\trop=(1,0).
\]
Thus, the tropical angle map is
\[
\alpha_1^\trop=(-1,0),
\qquad
\alpha_2^\trop=(1,0),
\qquad
\beta^\trop=(-2,1),
\qquad
\gamma^\trop=(2,1).
\]

We choose as spanning tree $T$ the complement of the edge $\beta^\trop \gamma^\trop$, and take the basepoint to be
$
\baseptT=\beta^\trop.
$
The degeneration-adapted tropical $B$-cycle is the bounded face, and so the period matrix is
\[
\Pi^{\mathrm{deg},\trop}
=
4+1+2+1
=
8,
\]
and the tropical theta function is
\[
\theta^\trop(x\mid 8)
=
\max_{m\in\Z}
\left\{
-4m^2-mx
\right\}.
\]
The values needed below are
\[
\theta^\trop(-3\mid8)=0,
\qquad
\theta^\trop(-2\mid8)=0,
\qquad
\theta^\trop(2\mid8)=0,
\qquad
\theta^\trop(5\mid8)=1;
\]
see \cref{fig:tropical-theta-example}.

\begin{figure}
\centering

\begin{tikzpicture}[
    x=0.55cm,
    y=0.45cm,
    line cap=round,
    line join=round,
    >=Stealth
]


\draw[->,line width=0.8pt]
    (-7.3,0) -- (7.3,0)
    node[right] {$x$};

\draw[->,line width=0.8pt]
    (0,-0.8) -- (0,3.7)
    node[above] {$y=\theta^\trop(x\mid 8)$};

%

\draw[curveblue,line width=2pt]
    (-7,3)
    -- (-4,0)
    -- (4,0)
    -- (7,3);






\node[circle,fill=black,inner sep=1.7pt]
    at (-3,0) {};

\node[circle,fill=black,inner sep=1.7pt]
    at (-2,0) {};

\node[circle,fill=black,inner sep=1.7pt]
    at (2,0) {};

\node[circle,fill=black,inner sep=1.7pt]
    at (5,1) {};

\node[above] at (-3,0)
    {$(-3,0)$};

\node[below] at (-2,0)
    {$(-2,0)$};

\node[below] at (2,0)
    {$(2,0)$};

\node[above=3pt] at (5,1)
    {$(5,1)$};

\end{tikzpicture}

\caption{
The tropical theta function
$
\theta^\trop(x\mid 8)
=
\max_{m\in\Z}\{-4m^2-mx\}.
$
}
\label{fig:tropical-theta-example}

\end{figure}
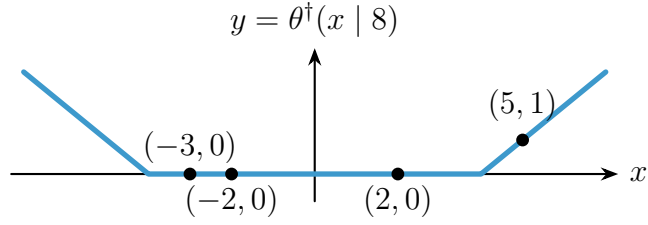

We have
$
\mu^{\mathrm{deg},\trop}_\baseptT(\divisor^\trop)=-5,
$
and by \cref{prop:riemann_constants_trop},
\[
\vecRC^{\mathrm{deg},\trop}
=
-\frac12\Pi^{\mathrm{deg},\trop}
=
-4.
\]

The boundary of $\bface_1$ decomposes into three wedges
$
\partial\bface_1
=
\mathfrak w_1+\mathfrak w_2+\mathfrak w_3,
$
whose associated zig-zag paths are
\[
\begin{array}{c|ccc}
&\alpha&\beta_-&\beta_+\\ \hline
\mathfrak w_1&\gamma&\beta&\alpha_1\\
\mathfrak w_2&\beta&\alpha_1&\gamma\\
\mathfrak w_3&\alpha_2&\gamma&\beta
\end{array}
\]
We next compute the tropical discrete Abel map for the faces appearing
around $\bface_1$. Recall that we normalize by
$\dam^\trop(\bw_1)=0$. We have
\(
\dam^\trop(\bface_1)=\beta^\trop,
\)
and hence
\(
\mu_{\baseptT}^{\mathrm{deg},\trop}
(\dam^\trop(\bface_1))
=
0.
\)

For the wedges, the relevant values of the tropical discrete Abel map are as
follows:
\[
\begin{array}{c|c|c|c|c}
\text{wedge}
&
\bface_-
&
\dam^\trop(\bface_-)
&
\bface_+
&
\dam^\trop(\bface_+)
\\ \hline
\mathfrak w_1
&
\bface_2 + (-1,0)
&
\gamma^\trop
&
\bface_3 + (0,-1)
&
\beta^\trop+\gamma^\trop-\alpha_1^\trop
\\[2pt]
\mathfrak w_2
&
\bface_4+(-1,0)
&
2\beta^\trop-\alpha_1^\trop
&
\bface_2 
&
2\beta^\trop-\gamma^\trop
\\[2pt]
\mathfrak w_3
&
\bface_4
&
\beta^\trop+\alpha_2^\trop-\gamma^\trop
&
\bface_3
&
\alpha_2^\trop
\end{array}
\]
The pairs of integers denotes which $\Z^2$-translate of the face in the fundamental domain is being considered. Applying the tropical Abel map gives
\[
\mu_{\baseptT}^{\mathrm{deg},\trop}(\alpha_1^\trop) = -1, \qquad \mu_{\baseptT}^{\mathrm{deg},\trop}(\alpha_2^\trop) = -3, \qquad \mu_{\baseptT}^{\mathrm{deg},\trop}(\beta^\trop) = 0, \qquad \mu_{\baseptT}^{\mathrm{deg},\trop}(\gamma^\trop) = -4,
\]
and therefore,
\[
\begin{array}{c|c|c}
\text{wedge}
&
\mu_{\baseptT}^{\mathrm{deg},\trop}
  (\dam^\trop(\bface_-))
&
\mu_{\baseptT}^{\mathrm{deg},\trop}
  (\dam^\trop(\bface_+))
\\ \hline
\mathfrak w_1 & -4 & -3\\
\mathfrak w_2 &  1 &  4\\
\mathfrak w_3 &  1 & -3
\end{array}
\]
We now compute the three wedge weights. For $\mathfrak w_1$, the adjacent faces are
\(
\bface_-=\bface_2,
\bface_+=\bface_3,
\)
and the associated zig-zag paths are
\(
(\alpha,\beta_-,\beta_+)
=
(\gamma,\beta,\alpha_1).
\)
Hence, the arguments of the theta function are
\[
\begin{aligned}
\mu_{\baseptT}^{\mathrm{deg},\trop}
(\dam^\trop(\bface_2))
-\mu^{\mathrm{deg},\trop}_\baseptT(\divisor^\trop)
+\vecRC^{\mathrm{deg},\trop}
&=
-4-(-5)-4
=
-3,
\\
\mu_{\baseptT}^{\mathrm{deg},\trop}
(\dam^\trop(\bface_3))
-\mu^{\mathrm{deg},\trop}_\baseptT(\divisor^\trop)
+\vecRC^{\mathrm{deg},\trop}
&=
-3-(-5)-4
=
-2.
\end{aligned}
\]
The differential
$
\omega_{\alpha_1^\trop-\beta^\trop}^\trop
$
is the unit flow along the path
$
(\beta^\trop
\rightarrow
\alpha_1^\trop)
$
and the integral $\int_{\baseptT}^{\gamma^\trop}\omega_{\alpha_1^\trop-\beta^\trop}^\trop$ is along the path 
\(
(\beta^\trop
\rightarrow
\alpha_1^\trop
\rightarrow
\alpha_2^\trop
\rightarrow
\gamma^\trop),
\)
so
\[
\int_\baseptT^{\gamma^\trop}
\omega_{\alpha_1^\trop-\beta^\trop}^\trop
=
1.
\]
Therefore,
\[
\begin{aligned}
\wt_{\mathrm{Fock}}^\trop(\mathfrak w_1)
&=
\theta^\trop(-3\mid8)
-
\theta^\trop(-2\mid8)
-
\int_\baseptT^{\gamma^\trop}
\omega_{\alpha_1^\trop-\beta^\trop}^\trop
\\
&=
0-0-1
\\
&=
-1.
\end{aligned}
\]
For $\mathfrak w_2$, the adjacent faces are
\(
\bface_-=\bface_4,
\bface_+=\bface_2,
\)
and the associated zig-zag paths are
\(
(\alpha,\beta_-,\beta_+)
=
(\beta,\alpha_1,\gamma).
\)
Hence, the arguments of the theta function are
\[
\begin{aligned}
\mu_{\baseptT}^{\mathrm{deg},\trop}
(\dam^\trop(\bface_4))
-\mu^{\mathrm{deg},\trop}_\baseptT(\divisor^\trop)
+\vecRC^{\mathrm{deg},\trop}
&=
1-(-5)-4
=
2,
\\
\mu_{\baseptT}^{\mathrm{deg},\trop}
(\dam^\trop(\bface_2))
-\mu^{\mathrm{deg},\trop}_\baseptT(\divisor^\trop)
+\vecRC^{\mathrm{deg},\trop}
&=
4-(-5)-4
=
5.
\end{aligned}
\]
Since $\alpha^\trop=\beta^\trop=\baseptT$, the integration path is constant, and
therefore
\[
\int_{\baseptT}^{\beta^\trop}
\omega_{\gamma^\trop-\alpha_1^\trop}^\trop
=
0.
\]
Thus,
\[
\begin{aligned}
\wt_{\mathrm{Fock}}^\trop(\mathfrak w_2)
&=
\theta^\trop(2\mid8)
-
\theta^\trop(5\mid8)
-
\int_{\baseptT}^{\beta^\trop}
\omega_{\gamma^\trop-\alpha_1^\trop}^\trop
\\
&=
0-1-0
\\
&=
-1.
\end{aligned}
\]

For $\mathfrak w_3$, the adjacent faces are
\(
\bface_-=\bface_4,
\bface_+=\bface_3,
\)
and the associated zig-zag paths are
\(
(\alpha,\beta_-,\beta_+)
=
(\alpha_2,\gamma,\beta).
\)
Hence, the arguments of the theta function are
\[
\begin{aligned}
\mu_{\baseptT}^{\mathrm{deg},\trop}
(\dam^\trop(\bface_4))
-\mu^{\mathrm{deg},\trop}_\baseptT(\divisor^\trop)
+\vecRC^{\mathrm{deg},\trop}
&=
1-(-5)-4
=
2,
\\
\mu_{\baseptT}^{\mathrm{deg},\trop}
(\dam^\trop(\bface_3))
-\mu^{\mathrm{deg},\trop}_\baseptT(\divisor^\trop)
+\vecRC^{\mathrm{deg},\trop}
&=
-3-(-5)-4
=
-2.
\end{aligned}
\]
The differential
\(
\omega_{\beta^\trop-\gamma^\trop}^\trop
\)
is the unit flow along the path
\(
(\gamma^\trop
\rightarrow
\alpha_2^\trop
\rightarrow
\alpha_1^\trop
\rightarrow
\beta^\trop),
\)
while the integration path from $\baseptT=\beta^\trop$ to
$\alpha_2^\trop$ is the path
\(
(\beta^\trop
\rightarrow
\alpha_1^\trop
\rightarrow
\alpha_2^\trop).
\)
Hence,
\[
\int_{\baseptT}^{\alpha_2^\trop}
\omega_{\beta^\trop-\gamma^\trop}^\trop
=
-3.
\]
Therefore,
\[
\begin{aligned}
\wt_{\mathrm{Fock}}^\trop(\mathfrak w_3)
&=
\theta^\trop(2\mid8)
-
\theta^\trop(-2\mid8)
-
\int_{\baseptT}^{\alpha_2^\trop}
\omega_{\beta^\trop-\gamma^\trop}^\trop
\\
&=
0-0-(-3)
\\
&=
3.
\end{aligned}
\]
The sum of the wedge weights is
\[
[\wt_{\mathrm{Fock}}^\trop](\partial\bface_1)
=
-1-1+3
=
1
=
X_1^\trop.
\]
Thus, the tropical Fock formula recovers the original tropical face
weight around $\bface_1$.
\end{example}

\section{A tropical Fay's trisecant identity}
\label{sec:tropical_fay}

In this section, we apply our asymptotic results to prove a tropical
version of Fay's trisecant identity. Tropical versions of Fay's identity
were first obtained by Inoue--Takenawa~\cite{InoueTakenawaFay} for hyperelliptic tropical curves and
subsequently generalized by Inoue--Iwao~\cite{InoueIwao2011} to general smooth
planar tropical curves. We first rewrite the classical identity in a
form involving normalized differentials of the third kind.

\begin{proposition}[{Fay's trisecant identity~\cite[p.~34, Eq.~(45)]{Fay}}]
\label{prop:fay_third_kind}
Let
$
\ptSig_1,\ptSig_2,\ptSig_3,\ptSig_4
$
be four distinct points of a compact Riemann surface $\mcurve$, and choose
lifts of these points to the universal cover. Let $\vecptjac \in\C^g$ and let
\[
\begin{aligned}
F_1&:=
\theta(\vecptjac+\bm \eta_{13}\mid\Pi)
\theta(\vecptjac+\bm \eta_{24}\mid\Pi)
\exp\left(
\int_{\ptSig_1}^{\ptSig_3}
\omega_{\ptSig_2-\ptSig_4}
\right)
\\
F_2&:=
\theta(\vecptjac+\bm \eta_{23}\mid\Pi)
\theta(\vecptjac+\bm \eta_{14}\mid\Pi)
\exp\left(
\int_{\ptSig_2}^{\ptSig_3}
\omega_{\ptSig_1-\ptSig_4}
\right)
\\
F_3&:=
\theta(\vecptjac\mid\Pi)
\theta(
\vecptjac+\bm \eta_{13}+\bm \eta_{24}
\mid\Pi),
\end{aligned}
\]
where $\bm \eta_{ij}
:=
\mu_{\basept}(\ptSig_j-\ptSig_i).$ Then, \[F_3=F_1+F_2.\]
\end{proposition}

We now tropicalize \cref{prop:fay_third_kind}. 

\begin{theorem}[Tropical Fay's trisecant identity]
\label{thm:tropical_fay}
Let $(\tropmcurve,\ell,\bm g)$ be an abstract tropical curve arising from
a degeneration as in \cref{sec:abst_trop_curve}, and let
$
\ptSig^\trop_1,\ptSig^\trop_2,\ptSig^\trop_3,\ptSig^\trop_4
\in
\vertices(\tropmcurve)\cap\partial f_0
$
be four vertices on the boundary of the outer face, labeled so that they occur weakly in the cyclic order
\[
\ptSig^\trop_1 \leq \ptSig^\trop_3 \leq \ptSig^\trop_2 \leq \ptSig^\trop_4 \leq \ptSig^\trop_1
\]
along $\partial f_0$. By weakly, we mean that the points are allowed to coincide. Choose an edge $e_0$ on the arc of $\partial f_0$ between $\ptSig_4^\trop$ and $\ptSig_1^\trop$ containing none of the other points, and choose the spanning tree $T$ so that it
contains every edge of $\partial f_0$ except $e_0$. Let $\vecptjac^\trop \in \R^{g^\trop}$ and let 
\[
\begin{aligned}
F_1^\trop
:={}&
\theta^\trop(
\vecptjac^\trop+\bm \eta_{13}^\trop
\mid
\Pi^{\mathrm{deg},\trop}
)
+
\theta^\trop(
\vecptjac^\trop+\bm \eta_{24}^\trop
\mid
\Pi^{\mathrm{deg},\trop}
)
-
\int_{\ptSig_1^\trop}^{\ptSig_3^\trop}
\omega_{\ptSig_2^\trop-\ptSig_4^\trop}^{\mathrm{deg},\trop},
\\
F_2^\trop
:={}&
\theta^\trop(
\vecptjac^\trop+\bm \eta_{23}^\trop
\mid
\Pi^{\mathrm{deg},\trop}
)
+
\theta^\trop(
\vecptjac^\trop+\bm \eta_{14}^\trop
\mid
\Pi^{\mathrm{deg},\trop}
)
-
\int_{\ptSig_2^\trop}^{\ptSig_3^\trop}
\omega_{\ptSig_1^\trop-\ptSig_4^\trop}^{\mathrm{deg},\trop},
\\
F_3^\trop
:={}&
\theta^\trop(
\vecptjac^\trop
\mid
\Pi^{\mathrm{deg},\trop}
)
+
\theta^\trop(
\vecptjac^\trop
+
\bm \eta_{13}^\trop
+
\bm \eta_{24}^\trop
\mid
\Pi^{\mathrm{deg},\trop}
),
\end{aligned}
\]
where $\bm \eta_{ij}^\trop
:=
\mu_{\baseptT}^{\trop}(\ptSig_j^\trop-\ptSig_i^\trop)$, and all tropical integrals are taken along the corresponding unique paths in the cut tropical curve $\tropmcurve_{\mathrm{cut}}$. Then,
\[
F_3^\trop
=
\max
\{
F_1^\trop,
F_2^\trop
\}.
\]
\end{theorem}

\begin{proof}

Let
$
\mcurve(t)\rightarrow\mcurve(\infty)
$
be the degeneration as in \cref{sec:degenerate_curve}. For $1\leq i\leq4$,
choose distinct real points
\[
\ptSig_i\in\mcurve_{v_i}^{\mathrm{core}}(\infty)
\]
such that the points
$
\ptSig_i(t)
=
\Phi_{v_i}(\ptSig_i;t^{-1})
$
lie on the outer oval $B_0^{\mathrm{ov}}(t)$ in the same cyclic order, and set
$
\ptSig_i^\trop:=v_i.
$
Set
\[
\bm \eta_{ij}(t)
:=
\int_{\ptSig_i(t)}^{\ptSig_j(t)}
\bm\omega^{\mathrm{deg}}(t).
\]

After cutting at $e_0$, the unique path in $T$ between any two of the
$\ptSig_i^\trop$ is the corresponding arc of
$\partial f_0\setminus e_0$. We choose its degeneration-adapted lift to
be the corresponding arc of $B_0^{\mathrm{ov}}(t)$, and take the
classical integral along this arc. By the cyclic ordering of the marked
points, the arc from $\ptSig_1(t)$ to $\ptSig_3(t)$ is disjoint from
$\ptSig_2(t)$ and $\ptSig_4(t)$, while the arc from $\ptSig_2(t)$ to
$\ptSig_3(t)$ is disjoint from $\ptSig_1(t)$ and $\ptSig_4(t)$.

By
Part~\itemref{BDT1} of \cref{prop:omega_and_delta_pos}, together with the
change-of-basis from the oval-adapted to the degeneration-adapted basis, we have
\[
\bm \eta_{ij}(t)\in\R^g.
\]

Let
$
\vecptjac(t)\in\R^g
$
be a family satisfying
\[
\vecptjac(t)
=
-
\begin{pmatrix}
0\\
\vecptjac^\trop
\end{pmatrix}
\log t
+
O(1)
\]
for some
$
\vecptjac^\trop\in\R^{g^\trop}.
$
Then all the arguments of the theta functions appearing in
\cref{prop:fay_third_kind} are real for $t\gg1$. By \cref{cor:abelmap},
\[
\bm \eta_{ij}(t)
=
-
\begin{pmatrix}
0\\
\bm \eta_{ij}^\trop
\end{pmatrix}
\log t
+
O(1).
\]
Therefore, by \cref{lem:theta_tropicalization},
\[
\begin{aligned}
\log
\theta(
\vecptjac(t)+\bm \eta_{ij}(t)
\mid
\Pi^{\mathrm{deg}}(t)
)
&=
\theta^\trop(
\vecptjac^\trop+\bm \eta_{ij}^\trop
\mid
\Pi^{\mathrm{deg},\trop}
)
\log t
+
O(1),\\
\log
\theta(
\vecptjac(t)
\mid
\Pi^{\mathrm{deg}}(t)
)
&=
\theta^\trop(
\vecptjac^\trop
\mid
\Pi^{\mathrm{deg},\trop}
)
\log t
+
O(1),\\
\log
\theta(
\vecptjac(t)+\bm \eta_{13}(t)+\bm \eta_{24}(t)
\mid
\Pi^{\mathrm{deg}}(t)
)
&=
\theta^\trop(
\vecptjac^\trop
+
\bm \eta_{13}^\trop
+
\bm \eta_{24}^\trop
\mid
\Pi^{\mathrm{deg},\trop}
)
\log t
+
O(1).
\end{aligned}
\]
By \cref{prop:trop_differential_third_kind},
\begin{align*}
\int_{\ptSig_1(t)}^{\ptSig_3(t)}
\omega_{\ptSig_2(t)-\ptSig_4(t)}^{\mathrm{deg}}(t)
&=
-
\left(
\int_{\ptSig_1^\trop}^{\ptSig_3^\trop}
\omega_{\ptSig_2^\trop-\ptSig_4^\trop}^{\mathrm{deg},\trop}
\right)
\log t
+
O(1),\\
\int_{\ptSig_2(t)}^{\ptSig_3(t)}
\omega_{\ptSig_1(t)-\ptSig_4(t)}^{\mathrm{deg}}(t)
&=
-
\left(
\int_{\ptSig_2^\trop}^{\ptSig_3^\trop}
\omega_{\ptSig_1^\trop-\ptSig_4^\trop}^{\mathrm{deg},\trop}
\right)
\log t
+
O(1).
\end{align*}

Let $F_i(t)$, $1\leq i\leq3$, denote the three terms in
\cref{prop:fay_third_kind}. Since the arguments of the theta functions are real and the period matrix is real by Part~\itemref{BDT2} of \cref{prop:omega_and_delta_pos} together with change-of-basis, every term in the series expansions of the theta functions are positive, and therefore the theta functions are also positive. Moreover, the exponential factors are also positive since the third kind differentials are real and we are integrating along real paths. Therefore, $F_1(t),F_2(t),F_3(t)>0$. Combining the asymptotics above gives
\[
\log F_i(t)
=
F_i^\trop\log t
+
O(1),
\qquad
1\leq i\leq3.
\]
Since each $F_i(t)$ is positive, there is no cancellation of leading terms and so
\[
F_3^\trop
=
\max
\{
F_1^\trop,
F_2^\trop
\}.
\]
\end{proof}

	\bibliographystyle{alpha}
	\bibliography{biblio}

@misc{IchikawaDegenerating,
  author = {Ichikawa, Takashi},
  title  = {Dimers for degenerating families of {M}-curves},
  year   = {2026},
  note   = {\arxiv{2601.18093}}
}

@misc{IchikawaHarmonicAmoebas,
  author = {Ichikawa, Takashi},
  title  = {General tropical convergence of harmonic amoebas},
  year   = {2026},
  note   = {\arxiv{2601.18180}}
}

@article {AFMS,
    AUTHOR = {Agostini, Daniele and Fevola, Claudia and Mandelshtam, Yelena
              and Sturmfels, Bernd},
     TITLE = {K{P} solitons from tropical limits},
   JOURNAL = {J. Symbolic Comput.},
  FJOURNAL = {Journal of Symbolic Computation},
    VOLUME = {114},
      YEAR = {2023},
     PAGES = {282--301},
      ISSN = {0747-7171,1095-855X},
   MRCLASS = {37K10 (14H42 35C08)},
  MRNUMBER = {4421050},
MRREVIEWER = {Arthemy\ V.\ Kiselev},
       DOI = {10.1016/j.jsc.2022.04.009},
       URL = {https://doi.org/10.1016/j.jsc.2022.04.009},
}

@article {thetasurfaces,
    AUTHOR = {Agostini, Daniele and \c{C}elik, T\"{u}rk\"{u} \"{O}zl\"{u}m
              and Struwe, Julia and Sturmfels, Bernd},
     TITLE = {Theta surfaces},
   JOURNAL = {Vietnam J. Math.},
  FJOURNAL = {Vietnam Journal of Mathematics},
    VOLUME = {49},
      YEAR = {2021},
    NUMBER = {2},
     PAGES = {319--347},
      ISSN = {2305-221X,2305-2228},
   MRCLASS = {14H42 (14K25 30F99)},
  MRNUMBER = {4287917},
MRREVIEWER = {Sanjay\ Kumar\ Singh},
       DOI = {10.1007/s10013-020-00443-x},
       URL = {https://doi.org/10.1007/s10013-020-00443-x},
}

@article {ichikwaCMP,
    AUTHOR = {Ichikawa, Takashi},
     TITLE = {Periods of tropical curves and associated {KP} solutions},
   JOURNAL = {Comm. Math. Phys.},
  FJOURNAL = {Communications in Mathematical Physics},
    VOLUME = {402},
      YEAR = {2023},
    NUMBER = {2},
     PAGES = {1707--1723},
      ISSN = {0010-3616,1432-0916},
   MRCLASS = {14T20 (14H42 30F99 37K10)},
  MRNUMBER = {4627329},
MRREVIEWER = {Simonetta\ Abenda},
       DOI = {10.1007/s00220-023-04757-y},
       URL = {https://doi.org/10.1007/s00220-023-04757-y},
}

@article {HubbardKoch,
    AUTHOR = {Hubbard, John H. and Koch, Sarah},
     TITLE = {An analytic construction of the {D}eligne-{M}umford
              compactification of the moduli space of curves},
   JOURNAL = {J. Differential Geom.},
  FJOURNAL = {Journal of Differential Geometry},
    VOLUME = {98},
      YEAR = {2014},
    NUMBER = {2},
     PAGES = {261--313},
      ISSN = {0022-040X,1945-743X},
   MRCLASS = {32G15 (14H10)},
  MRNUMBER = {3263519},
MRREVIEWER = {Milagros\ Izquierdo},
       URL = {http://projecteuclid.org/euclid.jdg/1406552251},
}

@article {ACS,
    AUTHOR = {Abramovich, Dan and Caporaso, Lucia and Payne, Sam},
     TITLE = {The tropicalization of the moduli space of curves},
   JOURNAL = {Ann. Sci. \'{E}c. Norm. Sup\'{e}r. (4)},
  FJOURNAL = {Annales Scientifiques de l'\'{E}cole Normale Sup\'{e}rieure.
              Quatri\`eme S\'{e}rie},
    VOLUME = {48},
      YEAR = {2015},
    NUMBER = {4},
     PAGES = {765--809},
      ISSN = {0012-9593,1873-2151},
   MRCLASS = {14T05 (14D20 14G22 14H10)},
  MRNUMBER = {3377065},
MRREVIEWER = {Simon\ Hampe},
       DOI = {10.24033/asens.2258},
       URL = {https://doi.org/10.24033/asens.2258},
}

@book {MaclaganSturmfels,
    AUTHOR = {Maclagan, Diane and Sturmfels, Bernd},
     TITLE = {Introduction to tropical geometry},
    SERIES = {Graduate Studies in Mathematics},
    VOLUME = {161},
 PUBLISHER = {American Mathematical Society, Providence, RI},
      YEAR = {2015},
     PAGES = {xii+363},
      ISBN = {978-0-8218-5198-2},
   MRCLASS = {14T05 (05B35 14M25 15A80 52B70)},
  MRNUMBER = {3287221},
MRREVIEWER = {Patrick Popescu-Pampu},
       DOI = {10.1090/gsm/161},
       URL = {https://doi.org/10.1090/gsm/161},
}

@book {GKZ,
    AUTHOR = {Gelfand, I. M. and Kapranov, M. M. and Zelevinsky, A. V.},
     TITLE = {Discriminants, resultants and multidimensional determinants},
    SERIES = {Modern Birkh\"{a}user Classics},
      NOTE = {Reprint of the 1994 edition},
 PUBLISHER = {Birkh\"{a}user Boston, Inc., Boston, MA},
      YEAR = {2008},
     PAGES = {x+523},
      ISBN = {978-0-8176-4770-4},
   MRCLASS = {14N05 (13D25 14M25 15A15 52B20)},
  MRNUMBER = {2394437},
}

@incollection {MikhalkinAmoebas,
    AUTHOR = {Mikhalkin, Grigory},
     TITLE = {Amoebas of algebraic varieties and tropical geometry},
 BOOKTITLE = {Different faces of geometry},
    SERIES = {Int. Math. Ser. (N. Y.)},
    VOLUME = {3},
     PAGES = {257--300},
 PUBLISHER = {Kluwer/Plenum, New York},
      YEAR = {2004},
   MRCLASS = {14P25 (14N10 14N35)},
  MRNUMBER = {2102998},
MRREVIEWER = {Jean-Yves Welschinger},
       DOI = {10.1007/0-306-48658-X\_6},
       URL = {https://doi.org/10.1007/0-306-48658-X_6},
}

@article {MikhalkinReal,
    AUTHOR = {Mikhalkin, G.},
     TITLE = {Real algebraic curves, the moment map and amoebas},
   JOURNAL = {Ann. of Math. (2)},
  FJOURNAL = {Annals of Mathematics. Second Series},
    VOLUME = {151},
      YEAR = {2000},
    NUMBER = {1},
     PAGES = {309--326},
      ISSN = {0003-486X},
   MRCLASS = {14P05},
  MRNUMBER = {1745011},
MRREVIEWER = {A. Tognoli},
       DOI = {10.2307/121119},
       URL = {https://doi.org/10.2307/121119},
}

@article {MikhalkinRullgaard,
    AUTHOR = {Mikhalkin, Grigory and Rullg\aa rd, Hans},
     TITLE = {Amoebas of maximal area},
   JOURNAL = {Internat. Math. Res. Notices},
  FJOURNAL = {International Mathematics Research Notices},
      YEAR = {2001},
    NUMBER = {9},
     PAGES = {441--451},
      ISSN = {1073-7928},
   MRCLASS = {14P99 (52A40)},
  MRNUMBER = {1829380},
MRREVIEWER = {Alexander Zvonkin},
       DOI = {10.1155/S107379280100023X},
       URL = {https://doi.org/10.1155/S107379280100023X},
}

@article {KenyonOkounkovHarnack,
    AUTHOR = {Kenyon, Richard and Okounkov, Andrei},
     TITLE = {Planar dimers and {H}arnack curves},
   JOURNAL = {Duke Math. J.},
  FJOURNAL = {Duke Mathematical Journal},
    VOLUME = {131},
      YEAR = {2006},
    NUMBER = {3},
     PAGES = {499--524},
      ISSN = {0012-7094},
   MRCLASS = {14H81 (14H50 82B23)},
  MRNUMBER = {2219249},
MRREVIEWER = {Scott Sheffield},
       DOI = {10.1215/S0012-7094-06-13134-4},
       URL = {https://doi.org/10.1215/S0012-7094-06-13134-4},
}

@article {HuNorton,
    AUTHOR = {Hu, Xuntao and Norton, Chaya},
     TITLE = {General variational formulas for {A}belian differentials},
   JOURNAL = {Int. Math. Res. Not. IMRN},
  FJOURNAL = {International Mathematics Research Notices. IMRN},
      YEAR = {2020},
    NUMBER = {12},
     PAGES = {3540--3581},
      ISSN = {1073-7928},
   MRCLASS = {32G20 (14H15 30F30)},
  MRNUMBER = {4120304},
       DOI = {10.1093/imrn/rny106},
       URL = {https://doi.org/10.1093/imrn/rny106},
}

@book {Tata1,
    AUTHOR = {Mumford, David},
     TITLE = {Tata lectures on theta. {I}},
    SERIES = {Modern Birkh\"{a}user Classics},
      NOTE = {With the collaboration of C. Musili, M. Nori, E. Previato and
              M. Stillman,
              Reprint of the 1983 edition},
 PUBLISHER = {Birkh\"{a}user Boston, Inc., Boston, MA},
      YEAR = {2007},
     PAGES = {xiv+235},
      ISBN = {978-0-8176-4572-4; 0-8176-4572-1},
   MRCLASS = {14K25 (11E45 11G10 14C30)},
  MRNUMBER = {2352717},
       DOI = {10.1007/978-0-8176-4578-6},
       URL = {https://doi.org/10.1007/978-0-8176-4578-6},
}

@article {KOS,
    AUTHOR = {Kenyon, Richard and Okounkov, Andrei and Sheffield, Scott},
     TITLE = {Dimers and amoebae},
   JOURNAL = {Ann. of Math. (2)},
  FJOURNAL = {Annals of Mathematics. Second Series},
    VOLUME = {163},
      YEAR = {2006},
    NUMBER = {3},
     PAGES = {1019--1056},
      ISSN = {0003-486X},
   MRCLASS = {60D05 (82B26 82B41)},
  MRNUMBER = {2215138},
MRREVIEWER = {Michael Pr\"{a}hofer},
       DOI = {10.4007/annals.2006.163.1019},
       URL = {https://doi.org/10.4007/annals.2006.163.1019},
}

@article {KO07b,
    AUTHOR = {Kenyon, Richard and Okounkov, Andrei},
     TITLE = {Limit shapes and the complex {B}urgers equation},
   JOURNAL = {Acta Math.},
  FJOURNAL = {Acta Mathematica},
    VOLUME = {199},
      YEAR = {2007},
    NUMBER = {2},
     PAGES = {263--302},
      ISSN = {0001-5962},
   MRCLASS = {60D05 (35Q53 52B99 82B23 82B41)},
  MRNUMBER = {2358053},
MRREVIEWER = {Eugenii Shustin},
       DOI = {10.1007/s11511-007-0021-0},
       URL = {https://doi.org/10.1007/s11511-007-0021-0},
}

@phdthesis{Ma2015,
  author       = {Ningning Ma},
  title        = {Tropicalization of the Dimer Model and Cluster Integrable Systems},
  school       = {Brown University},
  year         = {2015},
  advisor      = {Richard Kenyon},
  type         = {PhD thesis},
  url          = {https://repository.library.brown.edu/}
}

@article {GalashinGeorge2024,
    AUTHOR = {Galashin, Pavel and George, Terrence},
     TITLE = {Move-reduced graphs on a torus},
   JOURNAL = {Trans. Amer. Math. Soc.},
  FJOURNAL = {Transactions of the American Mathematical Society},
    VOLUME = {377},
      YEAR = {2024},
    NUMBER = {6},
     PAGES = {4055--4099},
      ISSN = {0002-9947},
   MRCLASS = {05C10 (13F60)},
  MRNUMBER = {4748614},
       DOI = {10.1090/tran/9168},
       URL = {https://doi.org/10.1090/tran/9168},
}

@article {bocklandt2015dimerabc,
    AUTHOR = {Bocklandt, Raf},
     TITLE = {A dimer {ABC}},
   JOURNAL = {Bull. Lond. Math. Soc.},
  FJOURNAL = {Bulletin of the London Mathematical Society},
    VOLUME = {48},
      YEAR = {2016},
    NUMBER = {3},
     PAGES = {387--451},
      ISSN = {0024-6093},
   MRCLASS = {82B20 (13F60 14J33 37B99 82B26 82D60)},
  MRNUMBER = {3509904},
       DOI = {10.1112/blms/bdv101},
       URL = {https://doi.org/10.1112/blms/bdv101},
}

@article {InoueTakenawaSpectral,
    AUTHOR = {Inoue, Rei and Takenawa, Tomoyuki},
     TITLE = {Tropical spectral curves and integrable cellular automata},
   JOURNAL = {Int. Math. Res. Not. IMRN},
  FJOURNAL = {International Mathematics Research Notices. IMRN},
      YEAR = {2008},
    NUMBER = {9},
     PAGES = {Art ID. rnn019, 27},
      ISSN = {1073-7928,1687-0247},
   MRCLASS = {37K20 (14H70 37B15 37J35 37K60 68Q80)},
  MRNUMBER = {2429250},
MRREVIEWER = {Bruno\ Fabre},
       DOI = {10.1093/imrn/rnn019},
       URL = {https://doi.org/10.1093/imrn/rnn019},
}

@article{GoodwillieIgusaMalkiewichMerling,
  author  = {Goodwillie, Thomas and Igusa, Kiyoshi and
             Malkiewich, Cary and Merling, Mona},
  title   = {On the functoriality of the space of equivariant smooth
             {$h$}-cobordisms},
  eprint  = {2303.14892},
  archivePrefix = {arXiv},
  primaryClass = {math.AT},
  year    = {2023}
}

@article{BocklandtStrebel,
  author = {Bocklandt, Raf},
  title  = {Strebel Differentials and stable Matrix Factorizations},
  year   = {2016},
  eprint = {1612.06800},
  archivePrefix = {arXiv},
  primaryClass = {math.RT}
}

@article {Olarte,
    AUTHOR = {Olarte, Jorge Alberto},
     TITLE = {The moduli space of {H}arnack curves in toric surfaces},
   JOURNAL = {Forum Math. Sigma},
  FJOURNAL = {Forum of Mathematics. Sigma},
    VOLUME = {9},
      YEAR = {2021},
     PAGES = {Paper No. e43, 26},
      ISSN = {2050-5094},
   MRCLASS = {14M25 (14H10)},
  MRNUMBER = {4266669},
MRREVIEWER = {Boulos\ El Hilany},
       DOI = {10.1017/fms.2021.37},
       URL = {https://doi.org/10.1017/fms.2021.37},
}

@incollection {InoueTakenawaJacobian,
    AUTHOR = {Inoue, Rei and Takenawa, Tomoyuki},
     TITLE = {Tropical {J}acobian and the generic fiber of the
              ultra-discrete periodic {T}oda lattice are isomorphic},
 BOOKTITLE = {Expansion of integrable systems},
    SERIES = {RIMS K\^{o}ky\^{u}roku Bessatsu, B13},
     PAGES = {175--190},
 PUBLISHER = {Res. Inst. Math. Sci. (RIMS), Kyoto},
      YEAR = {2009},
   MRCLASS = {37K20 (14T05 37J35)},
  MRNUMBER = {2642635},
MRREVIEWER = {Andrey\ E.\ Mironov},
}

@article {InoueTakenawaFay,
    AUTHOR = {Inoue, Rei and Takenawa, Tomoyuki},
     TITLE = {A tropical analogue of {F}ay's trisecant identity and the
              ultra-discrete periodic {T}oda lattice},
   JOURNAL = {Comm. Math. Phys.},
  FJOURNAL = {Communications in Mathematical Physics},
    VOLUME = {289},
      YEAR = {2009},
    NUMBER = {3},
     PAGES = {995--1021},
      ISSN = {0010-3616,1432-0916},
   MRCLASS = {14H70 (14H40 14K25 14T05 37K20 37K60)},
  MRNUMBER = {2511658},
MRREVIEWER = {Eugenii\ Shustin},
       DOI = {10.1007/s00220-009-0815-3},
       URL = {https://doi.org/10.1007/s00220-009-0815-3},
}

@article {Iwao,
    AUTHOR = {Iwao, Shinsuke},
     TITLE = {Integration over tropical plane curves and
              ultradiscretization},
   JOURNAL = {Int. Math. Res. Not. IMRN},
  FJOURNAL = {International Mathematics Research Notices. IMRN},
      YEAR = {2010},
    NUMBER = {1},
     PAGES = {112--148},
      ISSN = {1073-7928},
   MRCLASS = {14T05 (14H50)},
  MRNUMBER = {2576286},
MRREVIEWER = {Hannah Markwig},
       DOI = {10.1093/imrn/rnp129},
       URL = {https://doi.org/10.1093/imrn/rnp129},
}

@incollection {InoueIwao2011,
    AUTHOR = {Inoue, Rei and Iwao, Shinsuke},
     TITLE = {Tropical spectral curves, {F}ay's trisecant identity, and
              generalized ultradiscrete {T}oda lattice},
 BOOKTITLE = {New trends in quantum integrable systems},
     PAGES = {101--116},
 PUBLISHER = {World Sci. Publ., Hackensack, NJ},
      YEAR = {2011},
   MRCLASS = {14H70 (14T05 37K20 37K60)},
  MRNUMBER = {2766982},
MRREVIEWER = {Andrey E. Mironov},
       DOI = {10.1142/9789814324373\_0006},
       URL = {https://doi.org/10.1142/9789814324373_0006},
}

@incollection {MZ,
    AUTHOR = {Mikhalkin, Grigory and Zharkov, Ilia},
     TITLE = {Tropical curves, their {J}acobians and theta functions},
 BOOKTITLE = {Curves and abelian varieties},
    SERIES = {Contemp. Math.},
    VOLUME = {465},
     PAGES = {203--230},
 PUBLISHER = {Amer. Math. Soc., Providence, RI},
      YEAR = {2008},
   MRCLASS = {14T05 (05C38 14H40 14H42)},
  MRNUMBER = {2457739},
       DOI = {10.1090/conm/465/09104},
       URL = {https://doi.org/10.1090/conm/465/09104},
}

@article {KS,
    AUTHOR = {Kenyon, Richard W. and Sheffield, Scott},
     TITLE = {Dimers, tilings and trees},
   JOURNAL = {J. Combin. Theory Ser. B},
  FJOURNAL = {Journal of Combinatorial Theory. Series B},
    VOLUME = {92},
      YEAR = {2004},
    NUMBER = {2},
     PAGES = {295--317},
      ISSN = {0095-8956},
   MRCLASS = {60C05 (05B45 05C10 60J10 60J20)},
  MRNUMBER = {2099145},
MRREVIEWER = {Ralph Neininger},
       DOI = {10.1016/j.jctb.2004.07.001},
       URL = {https://doi.org/10.1016/j.jctb.2004.07.001},
}

@article {BDdT,
    AUTHOR = {Boutillier, C\'{e}dric and Cimasoni, David and de Tili\`ere,
              B\'{e}atrice},
     TITLE = {Minimal bipartite dimers and higher genus {H}arnack curves},
   JOURNAL = {Probab. Math. Phys.},
  FJOURNAL = {Probability and Mathematical Physics},
    VOLUME = {4},
      YEAR = {2023},
    NUMBER = {1},
     PAGES = {151--208},
      ISSN = {2690-0998},
   MRCLASS = {82B20 (05C10 30F99)},
  MRNUMBER = {4567400},
MRREVIEWER = {N. N. Ganikhodjaev},
       DOI = {10.2140/pmp.2023.4.151},
       URL = {https://doi.org/10.2140/pmp.2023.4.151},
}

@book {Fay,
    AUTHOR = {Fay, John D.},
     TITLE = {Theta functions on {R}iemann surfaces},
    SERIES = {Lecture Notes in Mathematics, Vol. 352},
 PUBLISHER = {Springer-Verlag, Berlin-New York},
      YEAR = {1973},
     PAGES = {iv+137},
   MRCLASS = {30A48 (14H15)},
  MRNUMBER = {335789},
MRREVIEWER = {H. M. Farkas},
}

@article {GGK,
    AUTHOR = {George, T. and Goncharov, A. B. and Kenyon, R.},
     TITLE = {The inverse spectral map for dimers},
   JOURNAL = {Math. Phys. Anal. Geom.},
  FJOURNAL = {Mathematical Physics, Analysis and Geometry. An International
              Journal Devoted to the Theory and Applications of Analysis and
              Geometry to Physics},
    VOLUME = {26},
      YEAR = {2023},
    NUMBER = {3},
     PAGES = {Paper No. 24, 51},
      ISSN = {1385-0172},
   MRCLASS = {82B20 (13F60 14M25 37J38)},
  MRNUMBER = {4640328},
       DOI = {10.1007/s11040-023-09466-5},
       URL = {https://doi.org/10.1007/s11040-023-09466-5},
}

@article {GK,
    AUTHOR = {Goncharov, Alexander B. and Kenyon, Richard},
     TITLE = {Dimers and cluster integrable systems},
   JOURNAL = {Ann. Sci. \'{E}c. Norm. Sup\'{e}r. (4)},
  FJOURNAL = {Annales Scientifiques de l'\'{E}cole Normale Sup\'{e}rieure. Quatri\`eme
              S\'{e}rie},
    VOLUME = {46},
      YEAR = {2013},
    NUMBER = {5},
     PAGES = {747--813},
      ISSN = {0012-9593},
   MRCLASS = {37K30 (05C10 30F60 53D17 82B23)},
  MRNUMBER = {3185352},
MRREVIEWER = {Guizhang Tu},
       DOI = {10.24033/asens.2201},
       URL = {https://doi.org/10.24033/asens.2201},
}

@article {FoGo_cluster,
    AUTHOR = {Fock, Vladimir V. and Goncharov, Alexander B.},
     TITLE = {Cluster ensembles, quantization and the dilogarithm},
   JOURNAL = {Ann. Sci. \'{E}c. Norm. Sup\'{e}r. (4)},
  FJOURNAL = {Annales Scientifiques de l'\'{E}cole Normale Sup\'{e}rieure. Quatri\`eme
              S\'{e}rie},
    VOLUME = {42},
      YEAR = {2009},
    NUMBER = {6},
     PAGES = {865--930},
      ISSN = {0012-9593},
   MRCLASS = {53D30 (13F60 30F60 32G15 33B30 53D55)},
  MRNUMBER = {2567745},
MRREVIEWER = {Athanase Papadopoulos},
       DOI = {10.24033/asens.2112},
       URL = {https://doi.org/10.24033/asens.2112},
}

@article {FZ,
    AUTHOR = {Fomin, Sergey and Zelevinsky, Andrei},
     TITLE = {Cluster algebras. {I}. {F}oundations},
   JOURNAL = {J. Amer. Math. Soc.},
  FJOURNAL = {Journal of the American Mathematical Society},
    VOLUME = {15},
      YEAR = {2002},
    NUMBER = {2},
     PAGES = {497--529 (electronic)},
      ISSN = {0894-0347},
   MRCLASS = {16S99 (14M99 17B99)},
  MRNUMBER = {1887642},
MRREVIEWER = {Eric N. Sommers},
       DOI = {10.1090/S0894-0347-01-00385-X},
       URL = {http://dx.doi.org/10.1090/S0894-0347-01-00385-X},
}

@Article{postnikov2006total,
  author       = {Alexander Postnikov},
  title        = {{Total positivity, Grassmannians, and networks}},
  journal      = {\arxiv{math/0609764v1}},
  year         = 2006
}

@Article{Fock,
  author       = {V. V. Fock},
  title        = {{Inverse spectral problem for GK integrable system}},
  journal      = {\arxiv{1503.00289v1}},
  year         = 2015
}

@Article{BerggrenBorodinGeorge,
  author       = {Tomas Berggren and Alexei Borodin and Terrence George},
  title        = {{Dimer models on astroidal zig-zag graphs}},
  journal      = {\arxiv{2605.03896}},
  year         = 2026
}

@article {Berggren-Borodin-tropical,
    AUTHOR = {Berggren, Tomas and Borodin, Alexei},
     TITLE = {Crystallization of the {A}ztec diamond},
   JOURNAL = {Adv. Math.},
  FJOURNAL = {Advances in Mathematics},
    VOLUME = {500},
      YEAR = {2026},
     PAGES = {Paper No. 111100, 61},
      ISSN = {0001-8708,1090-2082},
   MRCLASS = {82B20 (14T90)},
  MRNUMBER = {5086408},
       DOI = {10.1016/j.aim.2026.111100},
       URL = {https://doi.org/10.1016/j.aim.2026.111100},
}

@article {Berggren-Borodin-Aztec,
    AUTHOR = {Berggren, Tomas and Borodin, Alexei},
     TITLE = {Geometry of the doubly periodic {A}ztec dimer model},
   JOURNAL = {Commun. Am. Math. Soc.},
  FJOURNAL = {Communications of the American Mathematical Society},
    VOLUME = {5},
      YEAR = {2025},
     PAGES = {475--570},
      ISSN = {2692-3688},
   MRCLASS = {60G55 (14H70 41A60 60B10 60D05 60K35 82B20)},
  MRNUMBER = {4954940},
MRREVIEWER = {Thomas\ Polaski},
       DOI = {10.1090/cams/52},
       URL = {https://doi.org/10.1090/cams/52},
}

@Article{BoutillierTiliereAztec,
  author       = {C{\'e}dric Boutillier and B{\'e}atrice de Tili{\`e}re},
  title        = {{Fock's dimer model on the Aztec diamond}},
  journal      = {Ann. Inst. Henri Poincar{\'e} Comb. Phys. Interact.},
  year         = {2026},
  doi          = {10.4171/AIHPD/223}
}

@incollection {InoueIwao,
    AUTHOR = {Inoue, Rei and Iwao, Shinsuke},
     TITLE = {Tropical curves and integrable piecewise linear maps},
 BOOKTITLE = {Tropical geometry and integrable systems},
    SERIES = {Contemp. Math.},
    VOLUME = {580},
     PAGES = {21--39},
 PUBLISHER = {Amer. Math. Soc., Providence, RI},
      YEAR = {2012},
      ISBN = {978-0-8218-7553-7},
   MRCLASS = {14T05 (14H99)},
  MRNUMBER = {2985385},
MRREVIEWER = {Hsian-Hua\ Tseng},
       DOI = {10.1090/conm/580/11489},
       URL = {https://doi.org/10.1090/conm/580/11489},
}

@article {InoueKunibaTakagi,
    AUTHOR = {Inoue, Rei and Kuniba, Atsuo and Takagi, Taichiro},
     TITLE = {Integrable structure of box-ball systems: crystal, {B}ethe
              ansatz, ultradiscretization and tropical geometry},
   JOURNAL = {J. Phys. A},
  FJOURNAL = {Journal of Physics. A. Mathematical and Theoretical},
    VOLUME = {45},
      YEAR = {2012},
    NUMBER = {7},
     PAGES = {073001, 64},
      ISSN = {1751-8113,1751-8121},
   MRCLASS = {37K60 (14H70 14T05 37B15 37J35 37K20 82B23)},
  MRNUMBER = {2892302},
MRREVIEWER = {Eugenii\ Shustin},
       DOI = {10.1088/1751-8113/45/7/073001},
       URL = {https://doi.org/10.1088/1751-8113/45/7/073001},
}

@article {InoueLamPylyavskyy,
    AUTHOR = {Inoue, Rei and Lam, Thomas and Pylyavskyy, Pavlo},
     TITLE = {Toric networks, geometric {$R$}-matrices and generalized
              discrete {T}oda lattices},
   JOURNAL = {Comm. Math. Phys.},
  FJOURNAL = {Communications in Mathematical Physics},
    VOLUME = {347},
      YEAR = {2016},
    NUMBER = {3},
     PAGES = {799--855},
      ISSN = {0010-3616,1432-0916},
   MRCLASS = {37K10 (14H40 17B37 17B80)},
  MRNUMBER = {3551255},
       DOI = {10.1007/s00220-016-2739-z},
       URL = {https://doi.org/10.1007/s00220-016-2739-z},
}

@article {SpeyerWilliams1,
    AUTHOR = {Speyer, David and Williams, Lauren},
     TITLE = {The tropical totally positive {G}rassmannian},
   JOURNAL = {J. Algebraic Combin.},
  FJOURNAL = {Journal of Algebraic Combinatorics. An International Journal},
    VOLUME = {22},
      YEAR = {2005},
    NUMBER = {2},
     PAGES = {189--210},
      ISSN = {0925-9899,1572-9192},
   MRCLASS = {14P99 (05E25 14M15)},
  MRNUMBER = {2164397},
MRREVIEWER = {Meirav\ Amram},
       DOI = {10.1007/s10801-005-2513-3},
       URL = {https://doi.org/10.1007/s10801-005-2513-3},
}

@article {SpeyerWilliams2,
    AUTHOR = {Speyer, David and Williams, Lauren K.},
     TITLE = {The positive {D}ressian equals the positive tropical
              {G}rassmannian},
   JOURNAL = {Trans. Amer. Math. Soc. Ser. B},
  FJOURNAL = {Transactions of the American Mathematical Society. Series B},
    VOLUME = {8},
      YEAR = {2021},
     PAGES = {330--353},
      ISSN = {2330-0000},
   MRCLASS = {14T15 (05E14 14M15 52C40)},
  MRNUMBER = {4241765},
MRREVIEWER = {Margherita\ Barile},
       DOI = {10.1090/btran/67},
       URL = {https://doi.org/10.1090/btran/67},
}

@article {ArkaniHamedLamSpradlin,
    AUTHOR = {Arkani-Hamed, Nima and Lam, Thomas and Spradlin, Marcus},
     TITLE = {Positive configuration space},
   JOURNAL = {Comm. Math. Phys.},
  FJOURNAL = {Communications in Mathematical Physics},
    VOLUME = {384},
      YEAR = {2021},
    NUMBER = {2},
     PAGES = {909--954},
      ISSN = {0010-3616,1432-0916},
   MRCLASS = {14M15 (05E14 14T15 52B20 52B40 55R80)},
  MRNUMBER = {4259378},
MRREVIEWER = {Felipe\ Zald\'{\i}var},
       DOI = {10.1007/s00220-021-04041-x},
       URL = {https://doi.org/10.1007/s00220-021-04041-x},
}

@unpublished{GalashinGeorge2,
  author = {Galashin, Pavel and George, Terrence},
  note   = {In preparation},
}

@article {Tourkine,
    AUTHOR = {Tourkine, Piotr},
     TITLE = {Tropical amplitudes},
   JOURNAL = {Ann. Henri Poincar\'{e}},
  FJOURNAL = {Annales Henri Poincar\'{e}. A Journal of Theoretical and
              Mathematical Physics},
    VOLUME = {18},
      YEAR = {2017},
    NUMBER = {6},
     PAGES = {2199--2249},
      ISSN = {1424-0637,1424-0661},
   MRCLASS = {81T30},
  MRNUMBER = {3649455},
       DOI = {10.1007/s00023-017-0560-7},
       URL = {https://doi.org/10.1007/s00023-017-0560-7},
}

@article {Bob2,
    AUTHOR = {Bobenko, Alexander I. and Bobenko, Nikolai},
     TITLE = {Dimers and {M}-curves: {L}imit shapes from {R}iemann surfaces},
   JOURNAL = {Duke Math. J.},
  FJOURNAL = {Duke Mathematical Journal},
    VOLUME = {175},
      YEAR = {2026},
    NUMBER = {10},
     PAGES = {1861--1930},
      ISSN = {0012-7094,1547-7398},
   MRCLASS = {82B20 (14H81 30F30 35 49Q10 82B26)},
  MRNUMBER = {5109905},
       DOI = {10.1215/00127094-2025-0068},
       URL = {https://doi.org/10.1215/00127094-2025-0068},
}
\end{document}